\documentclass[11pt]{amsart}   	
\usepackage[margin=1in]{geometry}
\usepackage{subfigure}					
\usepackage{amsmath}
\usepackage{amssymb, amsfonts}
\usepackage{amsthm}
\usepackage{hyperref}
\usepackage{mathtools}
\usepackage{tikz}
\usetikzlibrary{cd}
\usetikzlibrary{decorations.markings}
\usepackage[shortlabels]{enumitem}
\usepackage{makecell}
\setcellgapes{4pt}

\numberwithin{equation}{section}

\usepackage[noabbrev,capitalise,nameinlink]{cleveref}

\hypersetup{colorlinks=true, citecolor=lavender, linkcolor=lavender,urlcolor=lavender}

\definecolor{lavender}{rgb}{0.5,0,1.0}
\definecolor{NormalGreen}{RGB}{0,220,0}
\definecolor{ChillBlue}{RGB}{0,182,255}
\definecolor{MyOrange}{RGB}{255,150,0}
\definecolor{Maroon}{RGB}{150,0,0}
\definecolor{Pink}{RGB}{255,0,255}
\definecolor{MyPurple}{RGB}{180,0,255} 

\newtheorem{theorem}{Theorem}[section]
\newtheorem*{theorem*}{Theorem}
\newtheorem{lemma}[theorem]{Lemma}
\newtheorem{corollary}[theorem]{Corollary}
\newtheorem{proposition}[theorem]{Proposition}
\newtheorem{conjecture}[theorem]{Conjecture}
\newtheorem{question}[theorem]{Question}

\theoremstyle{definition}
\newtheorem{definition}[theorem]{Definition}

\newtheorem*{define*}{Definition}
\newtheorem{example}[theorem]{Example}

\newtheorem*{notation*}{Notation}

\newtheorem{remark}[theorem]{Remark}
\newtheorem*{remark*}{Remark}

\newcommand{\AD}{\mathrm{AD}}
\newcommand{\MAD}{\mathrm{MAD}}

\newcommand{\bip}{\times}
\newcommand{\PP}{\mathbf{P}}
\newcommand{\LL}{\mathbf{A}}
\newcommand{\RR}{\mathbf{B}}
\newcommand{\zz}{\mathbf{z}}
\newcommand{\wo}{w_\circ}
\newcommand{\Des}{\mathrm{Des}}
\newcommand{\vv}{\nu}
\newcommand{\R}{\zeta}
\newcommand{\QQ}{\mathsf{Q}}
\newcommand{\HHH}{\mathrm{H}}
\newcommand{\MM}{\mathtt{M}}
\newcommand{\HH}{\mathtt{H}}
\newcommand{\DD}{\mathtt{D}}
\newcommand{\UU}{\mathtt{U}}

\newcommand{\Heap}{\mathrm{Heap}}
\newcommand{\Weak}{\mathrm{Weak}}
\newcommand{\Camb}{\mathrm{Camb}}
\newcommand{\sort}{\mathsf{sort}}
\newcommand{\rr}{\mathsf{r}}

\newcommand{\s}{\mathsf{s}}

\newcommand{\uu}{\mathsf{u}}
\newcommand{\sfc}{\mathsf{c}}

\newcommand{\rank}{\mathrm{rank}}

\newcommand{\D}{\mathcal{D}}
\newcommand{\U}{\mathcal{U}}

\newcommand{\T}{\mathcal{T}}
\newcommand{\W}{\mathcal{W}}
\renewcommand{\SS}{\mathcal{S}}
\newcommand{\F}{\mathcal{F}}

\newcommand{\C}{\mathcal{C}}
\newcommand{\X}{\mathcal{X}}
\newcommand{\Y}{\mathcal{Y}}

\newcommand{\Hom}{\mathrm{Hom}}
\newcommand{\Ext}{\mathrm{Ext}}
\newcommand{\End}{\mathrm{End}}

\newcommand{\image}{\mathrm{Im}}
\renewcommand{\dim}{\mathrm{dim}}

\renewcommand{\ker}{\mathrm{ker}}
\newcommand{\coker}{\mathrm{coker}}

\newcommand{\undim}{\underline{\dim}}

\DeclareMathOperator{\mods}{\mathsf{mod}}

\DeclareMathOperator{\Db}{\mathcal{D}^b}

\DeclareMathOperator{\wide}{\mathsf{wide}}
\DeclareMathOperator{\tors}{\mathsf{tors}}
\DeclareMathOperator{\torsf}{\mathsf{torsf}}
\DeclareMathOperator{\sbrick}{\mathsf{sbrick}}
\DeclareMathOperator{\brick}{\mathsf{brick}}

\DeclareMathOperator{\smc}{2\textnormal{-}\mathsf{smc}}
\DeclareMathOperator{\sbp}{\mathrm{sbp}}
\DeclareMathOperator{\add}{\mathsf{add}}
\DeclareMathOperator{\Filt}{\mathsf{Filt}}
\DeclareMathOperator{\Gen}{\mathsf{Gen}}
\DeclareMathOperator{\Cogen}{\mathsf{Cogen}}

\newcommand{\tor}{\mathsf{T}}
\renewcommand{\top}{\mathrm{top}}

\newcommand{\lperp}[1]{\prescript{\perp}{}{#1}}
\newcommand{\lperpNum}[2]{\prescript{\perp_{#1}}{}{#2}}
\newcommand{\rperp}[1]{#1^\perp}

\newcommand{\inv}{\mathrm{inv}}

\newcommand{\join}{\vee}
\renewcommand{\Join}{\bigvee}
\newcommand{\meet}{\wedge}
\newcommand{\Meet}{\bigwedge}

\newcommand{\covered}{{\,\,<\!\!\!\!\cdot\,\,\,}}

\newcommand{\jirr}{\mathrm{JIrr}}
\newcommand{\mirr}{\mathrm{MIrr}}

\newcommand{\pidown}{\pi_\downarrow}
\newcommand{\piup}{\pi^\uparrow}

\newcommand{\pop}{\mathrm{pop}}
\newcommand{\popdown}{\mathrm{pop}^{\downarrow}}
\newcommand{\popup}{\mathrm{pop}^{\uparrow}}
\newcommand{\Orb}{\mathrm{Orb}}
\newcommand{\spine}{\mathrm{spine}}

\usepackage{soul}

\newcommand{\dfn}[1]{\textcolor{blue}{\emph{#1}}}

\title{Pop-Stack Operators for Torsion Classes and Cambrian Lattices}
\author{Emily Barnard}
\address{Department of Mathematical Sciences, DePaul University, Chicago, IL 60604, USA}\email{e.barnard@depaul.edu}

\author{Colin Defant}
\address{Department of Mathematics, Harvard University, Cambridge, MA 02139, USA}\email{colindefant@gmail.com}

\author{Eric J. Hanson}
\address{Department of Mathematics, North Carolina State University, Raleigh, NC 27695, USA}
\email{ejhanso3@ncsu.edu}

\begin{document}

\maketitle

\begin{abstract}
The \emph{pop-stack operator} of a finite lattice $L$ is the map $\mathrm{pop}^{\downarrow}_L\colon L\to L$ that sends each element $x\in L$ to the meet of $\{x\}\cup\text{cov}_L(x)$, where $\text{cov}_L(x)$ is the set of elements covered by $x$ in $L$. We study several properties of the pop-stack operator of $\tors\Lambda$, the lattice of torsion classes of a $\tau$-tilting finite algebra $\Lambda$ over a field $K$. We describe the pop-stack operator in terms of certain mutations of 2-term simple-minded collections. This allows us to describe preimages of a given torsion class under the pop-stack operator. 

We then specialize our attention to Cambrian lattices of a finite irreducible Coxeter group $W$. Using tools from representation theory, we provide simple Coxeter-theoretic and lattice-theoretic descriptions of the image of the pop-stack operator of a Cambrian lattice (which can be stated without representation theory). When specialized to a bipartite Cambrian lattice of type A, this result settles a conjecture of Choi and Sun. We also settle a related enumerative conjecture of Defant and Williams. When $L$ is an arbitrary lattice quotient of the weak order on $W$, we prove that the maximum size of a forward orbit under the pop-stack operator of $L$ is at most the Coxeter number of $W$; when $L$ is a Cambrian lattice, we provide an explicit construction to show that this maximum forward orbit size is actually equal to the Coxeter number. 
\end{abstract}

\tableofcontents

\section{Introduction}\label{sec:intro}

\subsection{Pop-stack} 
Let $L$ be a finite lattice with meet operation $\wedge$ and join operation $\vee$. The \dfn{pop-stack operator} $\popdown_L\colon L\to L$ and the \dfn{dual pop-stack operator} $\popup_L\colon L\to L$ are defined by
\[\popdown_L(x) = x \meet \left(\Meet \{y \mid y \covered x\}\right)\quad\text{and}\quad \popup_L(x) = x \join \left(\Join \{y \mid x \covered y\}\right),\] where we write $u\lessdot v$ to mean that $u$ is covered by $v$ in $L$. These operators have appeared in various contexts; they serve as both useful tools \cite{AP, BaH, enomoto, ES, facial_torsion, Muhle2019,ReadingShard, sakai} and objects of interest in their own right \cite{ChoiSun, ClaessonPop, ClaessonPop2, DefantPopCoxeter, DefantMeeting, DefantWilliamsSemidistrim, Hong2022, Pudwell, Ungar}. When the lattice $L$ is understood, we will omit subscripts and simply denote these operators by $\popdown$ and $\popup$. 

In \cite{ClaessonPop, DefantPopCoxeter, DefantMeeting, Hong2022, Pudwell, Ungar}, the pop-stack operator has been considered as a dynamical system. Given a map $f\colon L\to L$ and an element $x\in L$, the \dfn{forward orbit} of $x$ under $f$ is the set 
\[
\Orb_f(x)=\{x,f(x),f^2(x),f^3(x),\ldots\},
\]
where $f^t$ is the map obtained by composing $f$ with itself $t$ times. To ease notation, let us write \[\mathcal O_L(x)=\Orb_{\popdown_L}(x)\] for the forward orbit of $x$ under $\popdown_L$. If $t$ is sufficiently large, then $(\popdown_L)^t(x)$ is equal to the minimal element $\hat 0$ of $L$ (which is a fixed point of $\popdown_L$). Thus, $\lvert\mathcal O_L(x)\rvert-1$ is equal to the number of iterations of $\popdown_L$ needed to send $x$ to $\hat 0$. Given an interesting lattice $L$, one of the primary problems one can consider about its pop-stack operator is that of computing \[\max_{x\in L}|\mathcal O_L(x)|.\] For a fixed $t$, one can also consider the \dfn{$t$-pop-stack sortable} elements of $L$, which are the elements $x\in L$ such that $(\popdown_L)^t(x)=\hat 0$. 

Defant and Williams \cite{DefantWilliamsSemidistrim} also found that it is fruitful to study the image of the pop-stack operator when $L$ is a semidistrim lattice; this is because the image of $\popdown$ has numerous interesting properties, some of which relate to a certain bijective \emph{rowmotion operator} $\mathrm{row}\colon L\to L$. For example, $\lvert\popdown(L)\rvert$ and $\lvert\popup(L)\rvert$ are both equal to the number of elements $x\in L$ such that $\mathrm{row}(x)\leq x$. The images of $\popdown$ and $\popup$ are also naturally in bijection with the set of facets of a certain simplicial complex called the \emph{canonical join complex} of~$L$. 

\subsection{Lattices of torsion classes}\label{subsec:intro_torsion}
In this article, we take a representation-theoretic perspective and consider a finite-dimensional basic algebra $\Lambda$ over a field $K$. The set of torsion classes (see \cref{subsec:torsion} for the definition) of finitely-generated (right) $\Lambda$-modules forms a complete lattice \cite{IRTT} that we denote by $\tors\Lambda$. 
Examples of lattices arising in this way include the weak order and Cambrian lattices associated to any finite crystallographic Coxeter group \cite{AHIKM,mizuno,IT}. 
This paper continues a rich tradition of using lattices of torsion classes as a tool for proving new results in both representation theory and in algebraic combinatorics; see, e.g., \cite{DIRRT,enomoto,IRRT,RST,TW1,TW2} and many others.

We focus our attention on the case when $\Lambda$ is $\tau$-tilting finite (in the sense of \cite{DIJ}); this is equivalent to assuming $\tors\Lambda$ is finite. 
Our goal is to study the image and the dynamical properties of the pop-stack operator of $\tors\Lambda$.
While $\popdown$ has already appeared (sometimes under different names) in the theory of lattices of torsion classes \cite{AP, BaH, enomoto, ES, facial_torsion, sakai}, it has primarily been used as a tool rather than a dynamical operator worthy of its own investigation. 
Let us remark that the article \cite{BTZ} initiated the study of dynamical combinatorics of torsion classes by considering rowmotion.

In \cref{sec:mutation}, we consider certain pairs of sets of modules called \emph{2-term simple-minded collections} and \emph{semibrick pairs}. 
2-term simple-minded collections were first introduced in the special case when $\Lambda$ is a \emph{symmetric algebra} in \cite{RickardEquivalences}; they are special cases of the more general \emph{simple-minded collections} introduced in \cite{al-nofayee} (under the name \emph{spherical collections}). Semibrick pairs are a generalization of 2-term simple-minded collections that were introduced in \cite{HansonIgusa_picture} as a tool for studying a generalization of the \emph{picture group} of \cite{ITW}.
As we recall in \cref{subsec:smc}, there is a bijection between $\tors\Lambda$ (with $\Lambda$ $\tau$-tilting finite) and the set $\smc\Lambda$ of 2-term simple-minded collections \cite{asai}.
By the results of \cite{BCZ}, this bijection encodes information about cover relations and canonical join representations in $\tors\Lambda$.

The set $\smc\Lambda$ also comes equipped with a set of \emph{mutation operators} that categorify the mutation theory of cluster algebras.
See \cite[Section~3.7]{BY} and \cite[Section~7.2]{KY} (special cases also appeared earlier in \cite[Section~8.1]{KontsevichSoibelmanStability} and \cite[Section~3]{KingQiuExchangeGraphs}). These mutation formulas can also be applied to many of the more general semibrick pairs; see \cite{BaH_preproj,HI_pairwise}. 
Moreover, under the bijection between $\smc\Lambda$ and $\tors\Lambda$, the mutation operators give a representation-theoretic formula for how the corresponding 2-term simple-minded collection changes when one traverses a cover relation in $\tors\Lambda$. 

For each semibrick pair $(\X,\Y)$, there is classically one mutation operator for each module in $\X \cup \Y$. In the present paper, we more generally define a mutation operator for each nonempty subset $\X' \subseteq \X$ or $\Y' \subseteq \Y$. 
Our first main result (\Cref{thm:pop_mutation}) states that, under the bijection between $\tors\Lambda$ and $\smc\Lambda$, applying the pop-stack operator and its dual to a torsion class corresponds to performing certain mutations on the associated 2-term simple-minded collections. 

In \Cref{subsec:preimages}, we characterize the preimages of a prescribed torsion class under $\popdown_{\tors\Lambda}$ and $\popup_{\tors\Lambda}$.
As corollaries, we obtain descriptions of the $1$-pop-stack sortable and the $2$-pop-stack sortable elements of $\tors\Lambda$ (\cref{cor:tors_one_poppable,cor:tors_2_poppable}). 

\subsection{Cambrian lattices}
Let $W$ be a finite irreducible Coxeter group, and let $\Weak(W)$ denote the (right) weak order on $W$. Given a Coxeter element $c$ of $W$, one can construct the \emph{$c$-Cambrian lattice} $\Camb_c$, which is the sublattice of $\Weak(W)$ consisting of Reading's \emph{$c$-sortable elements} \cite{ReadingCambrian,reading2007clusters}. (The $c$-Cambrian lattice is also a lattice quotient of $\Weak(W)$.) When $W$ is crystallographic, we can realize $\Camb_c$ as the lattice of torsion classes of the \emph{tensor algebra} $KQ_c$ associated to the weighted Dynkin quiver associated to $c$ \cite{IT}.

Two very special Cambrian lattices are the Tamari lattice and the type-B Tamari lattice; the images of the pop-stack operators on these lattices were characterized and enumerated in \cite{Hong2022} and \cite{ChoiSun}, respectively. On the other hand, analogous results for the Cambrian lattice $\Camb_{c^\bip}$ associated to a type-A bipartite Coxeter element $c^\bip$ have remained elusive: a conjectural enumeration of the image was formulated by Defant and Williams \cite{DefantWilliamsSemidistrim}, while a characterization of the image was conjectured by Choi and Sun \cite{ChoiSun}. 

It turns out that our representation-theoretic perspective is quite useful for understanding the image of $\popdown_{\Camb_c}$. In \Cref{thm:algebraic_image_characterization}, we provide a representation-theoretic characterization of the image of $\popdown_{\tors\Lambda}$ whenever $\Lambda$ is \emph{hereditary}. When the Coxeter group $W$ is crystallographic, the tensor algebra $KQ_c$ is hereditary, and we can reformulate our description of the image of the pop-stack operator in purely lattice-theoretic and Coxeter-theoretic terms. We then check directly (by hand and by computer) that these lattice-theoretic and Coxeter-theoretic descriptions still holds when $W$ is not crystallographic. 

Our Coxeter-theoretic description of the image of $\popdown_{\Camb_c}$ is surprisingly simple: we will prove (\cref{thm:combinatorial_image_description}) that a $c$-sortable element $w\in \Camb_c$ is in the image of $\popdown_{\Camb_c}$ if and only if the right descents of $w$ all commute and $w$ has no left inversions in common with $c^{-1}$. 

To state our purely lattice-theoretic description of the image of $\popdown_{\Camb_c}$, let us write $s_1,\ldots,s_n$ for the simple reflections of $W$; these are the atoms of $\Camb_c$. For $1\leq i\leq n$, let $p_i$ be the unique maximal element of the set \[\{x \in \Camb_c \mid s_i \leq x \text{ and } s_j \not\leq x \text{ for all }s_j \in S \setminus \{s_i\}\}.\] We will prove that an element $w\in \Camb_c$ is in the image of $\popdown_{\Camb_c}$ if and only if the interval $[\popdown_{\Camb_c}(w),w]$ (in $\Camb_c$) is Boolean and $p_i \not\leq w$ for all $1\leq i\leq n$. When $W$ is crystallographic, the elements $p_1,\ldots,p_n$ correspond in a natural way to the projective indecomposable modules of the tensor algebra $KQ_c$; this is one reason why our representation-theoretic perspective was so useful for discovering and proving our description of the image of $\popdown_{\Camb_c}$. 
 
Along the way to proving our characterization of the image of $\popdown_{\Camb_c}$, we prove (as part of \cref{thm:distributive_intervals}) that for $w\in \Camb_c$, the right descents of $w$ all commute with each other if and only if the interval $[\popdown_{\Camb_c}(w),w]$ of $\Camb_c$ is equal to the interval $[\popdown_{\Weak(W)}(w),w]$ of $\Weak(W)$. By combining this surprising result with our characterization of the image of $\popdown_{\Camb_c}$, we obtain an equally surprising dynamical corollary (\cref{cor:dynamical_weak_Camb}); namely, for $w\in\Camb_c$, we have \begin{equation}\label{eq:surprising_intro}(\popdown_{\Weak(W)})^t(\popdown_{\Camb_c}(w))=(\popdown_{\Camb_c})^{t+1}(w)
\end{equation} for all $t\geq 0$.  

When $W$ is of type $A_n$ and $c=c^\bip$ is a bipartite Coxeter element, our description of the image of $\popdown_{\Camb_c}$ resolves the aforementioned conjecture of Choi and Sun \cite{ChoiSun}. We will also construct a bijection from the image of $\popdown_{\Camb_{c^\bip}}$ to a certain set of Motzkin paths (\Cref{thm:main_bijection}); this allows us to resolve the aforementioned enumerative conjecture of Defant and Williams \cite{DefantWilliamsSemidistrim}. This result, the bijective proof of which utilizes the combinatorics of \emph{arc diagrams}, provides an exact enumeration of the facets of the canonical join complex of a bipartite type-A Cambrian lattice. 

Finally, we turn our attention to forward orbits under $\popdown_L$ when $L$ is a lattice quotient of the weak order on a finite irreducible Coxeter group $W$. In this setting, we first show that $\max_{x\in L}\lvert\mathcal O_{L}(x)\rvert\leq h$, where $h$ is the Coxeter number of $W$. We then prove that this inequality is actually an equality when $L=\Camb_c$ for some Coxeter element $c$ of $W$. To do so, we utilize the combinatorial properties of the $c$-sorting word for the long element of $W$ to construct an element $\zz_c\in \Camb_c$ such that $\lvert\mathcal O_{\Camb_c}(\zz_c)\rvert=h$. 

\subsection{Organization}
\cref{sec:posets,sec:rep_background} provide necessary background on posets and representation theory, respectively. In \Cref{sec:mutation}, we discuss the theory of mutation of semibrick pairs, and we extend the theory to allow mutation at multiple bricks simultaneously; the proof of one of the main results of this section (\cref{thm:mutation_summary}) is postponed until \cref{sec:appendix}. \Cref{sec:pop-stack_description} provides a representation-theoretic description of the pop-stack operator on $\tors\Lambda$, describes preimages of a torsion class under $\popdown_{\tors\Lambda}$, and characterizes the image of $\popdown_{\tors\Lambda}$. Beginning in \cref{sec:weak_cambrian}, we fixate on Cambrian lattices. \Cref{sec:weak_cambrian} provides background on Coxeter groups, root systems, and Cambrian lattices (including their realizations as lattices of torsion classes).
\cref{sec:image} is devoted to characterizing the image of the pop-stack operator on an arbitrary Cambrian lattice and deducing \cref{eq:surprising_intro}. In \cref{sec:Cambrian_A}, we provide explicit combinatorial characterizations and enumerations of the images of pop-stack operators on Cambrian lattices of type A. In \cref{sec:orbits}, we prove our results concerning the maximum sizes of forward orbits under the pop-stack operators on lattice quotients of the weak order. \cref{sec:conclusion} collects several ideas for future work.   

\section{Posets and Lattices}\label{sec:posets}

Let $P$ be a finite poset. For $x,y\in P$ with $x\leq y$, the \dfn{interval} between $x$ and $y$ is the induced subposet $[x,y]=\{z\in P\mid x\leq z\leq y\}$ of $P$. We say $y$ \dfn{covers} $x$ and write $x\lessdot y$ if $x<y$ and $[x,y]=\{x,y\}$. The \dfn{Hasse diagram} of $P$ is the graph with vertex set $P$ in which $x$ and $y$ are adjacent whenever $x\lessdot y$ or $y\lessdot x$. We say $P$ is \dfn{connected} if its Hasse diagram is connected. A \dfn{rank function} on $P$ is a function $\rank\colon P\to\mathbb Z$ such that $\rank(y)=\rank(x)+1$ whenever $x\lessdot y$. We say $P$ is \dfn{ranked} if there exists a rank function on $P$. An \dfn{antichain} is a poset $P$ in which any two distinct elements are incomparable (that is, if $x,y\in P$ are such that $x\leq y$, then $x=y$). The \dfn{dual} of $P$ is the poset $P^*$ with the same underlying set as $P$ defined so that $x\leq y$ in $P^*$ if and only if $y\leq x$ in $P$. A subset $X$ of $P$ is called \dfn{convex} if for all $x,y\in X$ and all $z\in P$ with $x\leq z\leq y$, we have $z\in X$. An \dfn{order ideal} is a subset $I$ of $P$ such that if $x\in P$ and $y\in I$ satisfy $x\leq y$, then $x\in I$. An \dfn{upper set} of $P$ is a subset $U\subseteq P$ such that $P\setminus U$ is an order ideal. We write $J(P)$ for the set of order ideals of $P$, ordered by inclusion. 

A \dfn{lattice} is a poset $L$ such that any two elements $x,y\in L$ have a greatest lower bound, which is called their \dfn{meet} and denoted by $x\wedge y$, and a least upper bound, which is called their \dfn{join} and denoted by $x\vee y$. A finite lattice is \dfn{distributive} if it is isomorphic to $J(P)$ for some finite poset $P$. A lattice is \dfn{Boolean} if it is isomorphic to the power set of a set, ordered by inclusion. Hence, if $P$ is an antichain, then $J(P)$ is Boolean. 

A lattice $L$ is \dfn{complete} if every (possibly infinite) subset has a meet (i.e., greatest lower bound) and a join (i.e., least upper bound). We write $\bigwedge X$ and $\bigvee X$ for the meet and join, respectively, of a subset $X$ of a complete lattice. Given lattices $L$ and $L'$, a \dfn{lattice homomorphism} is a map $\phi\colon L\to L'$ such that $\phi(x\wedge y)=\phi(x)\wedge\phi(y)$ and $\phi(x\vee y)=\phi(x)\vee\phi(y)$ for all $x,y\in L$. We say $L'$ is a \dfn{lattice quotient} if there is a surjective lattice homomorphism from $L$ to $L'$. 

Assume $L$ is a finite lattice. Then $L$ has a unique minimal element $\hat 0=\bigwedge L$ and a unique maximal element $\hat 1=\bigvee L$. An element of $L$ that covers $\hat 0$ is called an \dfn{atom}, while an element covered by $\hat 1$ is called a \dfn{coatom}. An element $j\in L$ is called \dfn{join-irreducible} if there does not exist a set $X\subseteq L\setminus\{j\}$ such that $j=\bigvee X$. Equivalently, $j$ is join-irreducible if it covers exactly one element of $L$. (Note that $\hat 0$ is not join-irreducible because it is equal to $\bigvee\emptyset$.) Dually, an element $m\in L$ is \dfn{meet-irreducible} if there does not exist a set $X\subseteq L\setminus\{m\}$ such that $j=\bigwedge X$. Equivalently, $m$ is meet-irreducible if it is covered by exactly one element of $L$. (Note that $\hat 1=\bigwedge\emptyset$ is not meet-irreducible.) Let $\jirr_L$ (resp.\ $\mirr_L$) be the set of join-irreducible (resp.\ meet-irreducible) elements of $L$. For $j\in\jirr_L$ and $m\in\mirr$, let $j_*$ be the unique element covered by $j$, and let $m^*$ be the unique element that covers $m$. A set $A\subseteq L$ is \dfn{join-irredundant} (resp.\ \dfn{meet-irredundant}) if $\bigvee A'<\bigvee A$ (resp.\ $\bigwedge A'>\bigwedge A$) for every proper subset $A'$ of $A$.  The \dfn{canonical join representation} of an element $x\in L$ (if it exists) is the unique join-irredundant set $A\subseteq\jirr_L$ satisfying the following: \begin{itemize}
\item $x=\bigvee A$.
\item For every join-irredundant set $B\subseteq\jirr_L$ such that $x=\bigvee B$, there exist $a\in A$ and $b\in B$ such that $a\leq b$. 
\end{itemize}
Dually, the \dfn{canonical meet representation} of $x$ (if it exists) is the unique meet-irredundant set $A\subseteq\mirr_L$ satisfying the following: \begin{itemize}
\item $x=\bigwedge A$.
\item For every meet-irredundant set $B\subseteq\mirr_L$ such that $x=\bigvee B$, there exist $a\in A$ and $b\in B$ such that $a\geq b$. 
\end{itemize}

We say a lattice $L$ is \dfn{semidistributive} if for all $x,y,z\in L$, we have \[x\wedge y=x\wedge z\implies x\wedge y=x\wedge(y\vee z)\quad\text{and}\quad x\vee y=x\vee z\implies x\vee y=x\vee(y\wedge z).\] 
Suppose $L$ is finite and semidistributive. 
It is known that every element $v$ of $L$ has a canonical join representation $\D(v)$ and a canonical meet representation $\U(v)$; in fact, the existence of both representations for every $v \in L$ is equivalent to semidistributivity (see \cite[Theorem~2.24]{FJN}). 
Moreover, the collection of canonical join representations (resp.\ canonical meet representations) of elements of $L$ forms a simplicial complex called the \dfn{canonical join complex} (resp.\ \dfn{canonical meet complex}) of $L$. There is a unique bijection $\kappa\colon\jirr_L\to\mirr_L$ such that $\kappa(j)\wedge j=j_*$ and $\kappa(j)\vee j=(\kappa(j))^*$ for all $j\in\jirr_L$. The map $A\mapsto\kappa(A)$ is an isomorphism from canonical join complex of $L$ to the canonical meet complex of $L$ \cite[Corollary~5]{BarnardCJC}. Moreover, the number of facets in each of these simplicial complexes is equal to both $\lvert\popdown_L(L)\rvert$ and $\lvert\popup_L(L)\rvert$ by  \cite[Theorem~9.13]{DefantWilliamsSemidistrim}. 
Indeed, the facets of the canonical join complex (resp.\ canonical meet complex) of $L$ are precisely the canonical meet representations (resp.\ canonical join representations) of the elements of $\popdown_L(L)$ (resp.\ $\popup_L(L)$). Let $\PP_L(q)$ be the generating function that counts the facets of the canonical join complex (equivalently, the canonical meet complex) according to their sizes. Then 
\begin{equation}\label{eq:P}
\PP_L(q)=\sum_{v\in \popdown_L(L)}q^{|\U(v)|}=\sum_{v\in \popup_L(L)}q^{|\D(v)|}.
\end{equation}
The canonical join complex of $L$ is equal to the canonical meet complex of the dual lattice $L^*$, so 
\[
\PP_L(q)=\PP_{L^*}(q).
\]

The \dfn{Galois graph} of a finite semidistributive lattice $L$ is the loopless directed graph with vertex set $\jirr_L$ in which there is an arrow $j\to j'$ if and only if $j\neq j'$ and $j\not\leq\kappa(j')$. For each edge $x\lessdot y$ in the Hasse diagram of $L$, there is a unique join-irreducible element $j_{x,y}\in\jirr_L$ such that $j_{x,y}\leq y$ and $\kappa(j_{x,y})\geq x$ (see \cite[Proposition~2.2~\&~Lemma~3.3]{BarnardCJC}). We call $j_{x,y}$ the \dfn{shard label} of the edge $x\lessdot y$. The canonical join representation and canonical meet representation of an element $v\in L$ are given by $\mathcal D(v)=\{j_{x,v}:x\lessdot v\}$ and $\mathcal U(v)=\{\kappa(j_{v,y}):v\lessdot y\}$. Moreover, the canonical join complex of $L$ is just the collection of independent sets of the Galois graph of $L$. 

Note that intervals in semidistributive lattices are semidistributive. The following lemma was stated in \cite{DefantWilliamsSemidistrim} for \emph{semidistrim} lattices, which are more general than semidistributive lattices, but we only need to consider semidistributive lattices. (See also \cite[Section~4]{RST}.) 

\begin{lemma}[{\cite[Corollary~7.10]{DefantWilliamsSemidistrim}}]\label{lem:Cor710}
Let $L$ be a finite semidistributive lattice, and let $[u,v]$ be an interval in $L$. There is a bijection from $\{j\in\jirr_L\mid j\leq v\text{ and }\kappa(j)\geq u\}$ to $\jirr_{[u,v]}$ given by $j\mapsto u\vee j$. This bijection is an isomorphism from an induced subgraph of the Galois graph of $L$ to the Galois graph of $[u,v]$. If $x\lessdot y$ is an edge in $[u,v]$ whose shard label in $L$ is $j_{x,y}$, then the shard label of $x\lessdot y$ in $[u,v]$ is $u\vee j_{x,y}$. 
\end{lemma}

\begin{example}\label{exam:boolean}
Let $P$ be a finite antichain, and consider the Boolean lattice $J(P)$. The join-irreducible elements of $J(P)$ are the singleton subsets of $P$. If $A\lessdot B$ is an edge in $J(P)$, then there exists $b\in B$ such that $A=B\setminus\{b\}$. The shard label of $A\lessdot B$ is $j_{A,B}=\{b\}$. The Galois graph of $J(P)$ has no edges. 
\end{example}

\begin{example}\label{exam:dihedral}
Let $L$ be the lattice obtained by adding a minimal element $\hat 0$ and a maximal element $\hat 1$ to a disjoint union of two chains $a_1\lessdot\cdots\lessdot a_{m-1}$ and $a_m\lessdot\cdots\lessdot a_{2m-2}$ (so $a_1$ and $a_m$ are the atoms of $L$, while $a_{m-1}$ and $a_{2m-2}$ are the coatoms). This lattice is isomorphic to the weak order on the dihedral group $I_2(m)$ (see \cref{sec:weak_cambrian}). We have $\jirr_L=\mirr_L=\{a_1,\ldots,a_{2m-2}\}$, and the bijection $\kappa$ is given by $\kappa(a_i)=a_{i-1}$, where we let $a_0=a_{2m-2}$. For $1\leq i\leq m-1$, there is an arrow from $a_i$ to $a_{i'}$ in the Galois graph of $L$ if and only if $1\leq i'\leq i-1$ or $m+1\leq i'\leq 2m-2$. For $m\leq i\leq 2m-2$, there is an arrow from $a_i$ to $a_{i'}$ in the Galois graph of $L$ if and only if $2\leq i'\leq m-1$ or $m\leq i'\leq i-1$. See \cref{fig:exam_dihedral}. 
\end{example} 

\begin{figure}[ht]
\begin{center}\includegraphics[height=8.6cm]{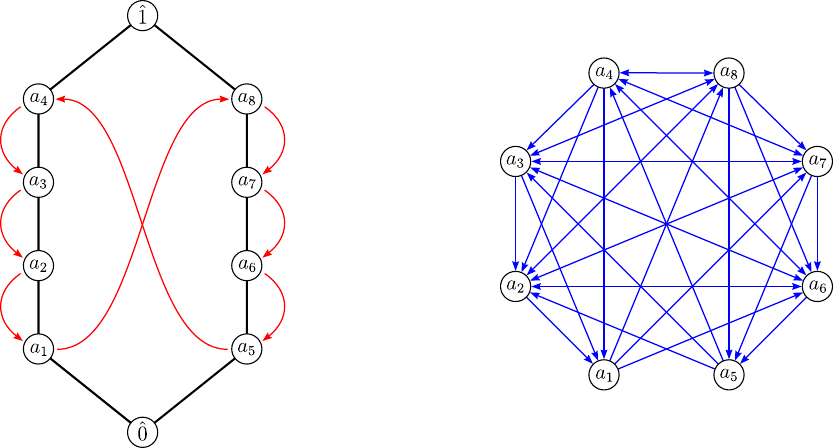}
  \end{center}
  \caption{The Hasse diagram (left) and Galois graph (right) of the lattice $L$ from \cref{exam:dihedral} with $m=5$. The red arrows on the left indicate the bijection $\kappa$.}
  \label{fig:exam_dihedral}
\end{figure}

\section{Finite-Dimensional Algebras}\label{sec:rep_background}

In this section, we recall background information on the representation theory of quivers and finite-dimensional algebras. We refer to the textbooks \cite{ARS,ASS,schiffler} as standard background references.

Let $\Lambda$ be a finite-dimensional basic\footnote{Recall that a module $M$ is \dfn{basic} if there does not exist a nonzero module $N$ such that $N \oplus N$ is a direct summand of $M$. An algebra $\Lambda$ is basic if it is basic as a free $\Lambda$-module.} algebra over a field $K$. In the second half of the paper, we will mainly consider the case where $\Lambda$ is the \emph{tensor algebra} over a \emph{weighted Dynkin quiver}; discussion of this special case is deferred to \cref{subsec:path_algebras}.

We denote by $\mods \Lambda$ the category of finitely-generated right $\Lambda$-modules and by $\Db(\mods\Lambda)$ the bounded derived category of $\mods\Lambda$ with shift functor $[1]$. We recall that $\Lambda$ is \dfn{hereditary} if $\Ext^m_\Lambda(M,N) = 0$ for all $M, N \in \mods\Lambda$ and $m > 1$. Given a basic object $M \in \mods\Lambda$, we denote by $|M|$ the number of indecomposable direct summands of $M$. For $M, N \in \mods\Lambda$, we will sometimes write $M \hookrightarrow N$ (resp.\ $M \twoheadrightarrow N$) to represent a monomorphism/injection (resp.\ epimorphism/surjection) in $\Hom_\Lambda(M,N)$. A module $X \in \mods\Lambda$ is called a \dfn{brick} if every nonzero endomorphism of $X$ is invertible. In particular, every brick must be indecomposable.

Let $n = |\Lambda|$, and choose an indexing $P(1),\ldots, P(n)$ of the indecomposable direct summands of $\Lambda$. Then for every indecomposable and projective module $Q \in \mods\Lambda$, there exists a unique $i$ such that $Q \cong P(i)$. Moreover, up to isomorphism, $\mods\Lambda$ contains $n$ simple modules. These can be indexed $S(1),\ldots,S(n)$ so that $\Hom_\Lambda(P(i),S(j)) \neq 0$ if and only if $i = j$. Moreover, each $S(i)$ is a brick, and we have $\Hom_\Lambda(P(i),S(i)) \cong \End_\Lambda(S(i))$. Given an arbitrary module $M \in \mods\Lambda$, we define its \dfn{dimension vector} to be
\begin{equation}\label{eqn:dim_vect}\undim M = \left(\dim_{\End_\Lambda(S(i))} \Hom_\Lambda(P(i),M)\right)_{i \in [n]} \in \mathbb{N}^n.\end{equation}

\begin{remark}
    As defined, the dimension vector $\undim M$ depends on the indexing of the projective modules. To avoid this dependency, one can instead define $\undim M$ to be the class $[M]$ of $M$ in the \dfn{Grothendieck group} $K_0(\mods\Lambda)$. Indeed, this group is free abelian and has $\{[S(i)] \mid i \in [n]\}$ as a basis. Thus, identifying each $[S(i)] \in K_0(\mods\Lambda)$ with the standard basis element $e_i \in \mathbb{Z}^n$ identifies the class $[M]$ with the vector $\undim M$ as defined above.   
\end{remark}

The purpose of the remainder of this section is to recall the definitions of torsion pairs and 2-term simple-minded collections, as well as the relationship between them. In doing so, we follow much of the exposition of \cite[Section~3]{BaH}.

\subsection{Subcategories and functorial finiteness}\label{subsec:ff}

By a \dfn{subcategory} of $\mods \Lambda$, we will always mean a subcategory that is full and closed under isomorphisms. For a subcategory $\C \subseteq \mods\Lambda$, we define several additional subcategories as follows:
\begin{itemize}
    \item $\add(\C)$ is the subcategory of $\mods\Lambda$ consisting of all direct summands of finite direct sums of the objects in $\C$.
    \item $\Gen(\C)$ is the subcategory of $\mods\Lambda$ consisting of all objects that can be written as quotients of objects in $\add(\C)$.
    \item $\Cogen(\C)$ is the subcategory of $\mods\Lambda$ consisting of all objects that can be written as subobjects of objects in $\add(\C)$.
    \item $\Filt(\C)$ is the subcategory of $\mods\Lambda$ consisting of all objects that can be written as iterated extensions of objects in $\C$; that is, $M \in \Filt(\C)$ if and only if there exists a chain
    $$0 = M_0 \subsetneq M_1 \subsetneq \cdots \subsetneq M_k = M$$
    of modules such that each factor $M_i/M_{i-1}$ is in $\add(\C)$.
    \item $\tor(\C) = \Filt(\Gen(\C))$.
        \item $\C^\perp = \C^{\perp_0} = \{Y \in \mods\Lambda \mid \Hom_\Lambda(-,Y)|_\C = 0\}$.
        \item $\lperp{\C} = \lperpNum{0}{\C} = \{Y \in \mods \Lambda \mid \Hom_\Lambda(Y,-)|_\C = 0\}$.
    \item $\C^{\perp_1} = \{Y \in \mods\Lambda \mid \Ext^1_\Lambda(-,Y)|_\C = 0\}$ 
    \item $\lperpNum{1}{\C} = \{Y \in \mods \Lambda \mid \Ext^1_\Lambda(Y,-)|_\C = 0\}$.
    \item $\C^{\perp_{0,1}} = \C^\perp \cap \C^{\perp_1}$ and $\lperpNum{0,1}{\C} = \lperp{\C} \cap \lperpNum{1}{\C}$.
\end{itemize}
We also apply the above definitions to the objects of $\mods\Lambda$ by considering (the isomorphism class of) a given object as a subcategory.

Now let $\C \subseteq \mods\Lambda$ be an arbitrary subcategory and $M \in \mods\Lambda$. A \dfn{minimal left $\C$-approximation} of $M$ is a morphism $g_{M,\C}\colon M \rightarrow C$ with $C \in \C$ such that 
\begin{enumerate}[(i)]
\item the induced map $(g_{M,\C})^*\colon\Hom_\Lambda(C,D)\rightarrow \Hom_\Lambda(M,D)$ is surjective for all $D \in \C$ and 
\item if $f\colon C \rightarrow C$ is such that $f \circ g_{M,\C} = g_{M,\C}$, then $f$ is an isomorphism. 
\end{enumerate} 
Dually, a \dfn{minimal right $\C$-approximation} of $M$ is a morphism $g_{\C,M}\colon C \rightarrow M$ with $C \in \C$ such that 
\begin{enumerate}[(i*)]
\item the induced map $(g_{\C,M})_*\colon\Hom_\Lambda(D,C)\rightarrow \Hom_\Lambda(D,M)$ is surjective for all $D \in \C$ and 
\item if $f\colon C \rightarrow C$ is such that $g_{\C,M} \circ f = g_{\C,M}$, then $f$ is an isomorphism.
\end{enumerate} It is well known that minimal left and right $\C$-approximations are unique up to isomorphism whenever they exist. If every object of $\mods \Lambda$ admits both a minimal left $\C$-approximation and a minimal right $\C$-approximation, then $\C$ is said to be \dfn{functorially finite}. Almost all of the subcategories we consider in this paper will be functorially finite.

\subsection{Torsion classes}\label{subsec:torsion}

A pair $(\T,\F)$ of subcategories of $\mods\Lambda$ is called a \dfn{torsion pair} if 
\begin{enumerate}[(i)] 
\item $\Hom_\Lambda(T,F) = 0$ for all $T \in \T$ and $F \in \F$ and
\item every $M \in \mods \Lambda$ is the middle term of an exact sequence
$t_\T M \hookrightarrow M \twoheadrightarrow f_\F M$
with $t_\T M \in \T$ and $f_\F M \in \F$.
\end{enumerate}
In this case, we say that $\T$ is a \dfn{torsion class} and that $\F$ is a \dfn{torsion-free class}. We denote by $\tors\Lambda$ and $\torsf \Lambda$ the sets of torsion classes and torsion-free classes in $\mods\Lambda$, respectively. Torsion pairs were introduced by Dickson in \cite{dickson} as a generalization of the torsion and torsion-free abelian groups. We refer to the expository article \cite{thomas_intro} for background information on torsion pairs, including formal statements for the facts mentioned in this section.

It is well known that a subcategory $\T \subseteq \mods\Lambda$ (resp.\ $\F \subseteq \mods\Lambda$) is a torsion class (resp.\ torsion-free class) if and only if it is closed under extensions and quotients (resp.\ extensions and submodules). Moreover, every torsion class determines a unique torsion pair and vice versa via the associations $\T \mapsto (\T,\T^\perp)$ and $(\T,\F) \mapsto \T$. Similarly, every torsion-free class determines a unique torsion pair and vice versa via the associations $\F \mapsto (\lperp{\F},\F)$ and $(\T,\F) \mapsto \F$.

For $\C \subseteq \mods\Lambda$ an arbitrary subcategory, we have that $\tor( \C)$ is the smallest torsion class containing $\C$; the corresponding torsion pair is $(\tor(\C), \rperp{\C})$. Dually, $\Filt(\Cogen (\C))$ is the smallest torsion-free class containing $\C$, and the corresponding torsion pair is $(\lperp{\C},\Filt(\Cogen(\C)))$.

We consider $\tors\Lambda$ and $\torsf \Lambda$ as posets under the containment relation. Each of these is a complete lattice whose meet operation is given by intersection and whose join operation is given by $\Join \mathcal{A} = \Filt\left(\bigcup_{\T \in \mathcal{A}} \T\right)$. The associations $\T \mapsto \T^\perp$ and $\F \mapsto \lperp{\F}$ thus give lattice anti-isomorphisms between $\tors\Lambda$ and $\torsf\Lambda$. In particular, the join operation in $\tors\Lambda$ can be rewritten as ${\T \join \U = \lperp({\T^\perp \cap \U^\perp})}$. The maximum element of $\tors\Lambda$ is $\hat{1} = \mods\Lambda$ (the torsion class consisting of all modules), and the minimum element is $\hat{0} = 0$ (the torsion class consisting only of the zero module). The atoms of $\tors\Lambda$ are precisely those torsion classes of the form $\Filt(S)$ for $S \in \mods\Lambda$ a simple module.

In this paper, we restrict our attention to algebras $\Lambda$ for which the lattice $\tors\Lambda$ is finite.  These are precisely the \dfn{$\tau$-tilting finite algebras} introduced in \cite{DIJ}. They are also precisely the algebras $\Lambda$ for which every torsion class in $\mods\Lambda$ is functorially finite. A related class is that of \dfn{representation-finite} algebras, which are characterized by the property that $\mods\Lambda$ contains only finitely many indecomposable modules (up to isomorphism). If $\Lambda$ is hereditary, then it is $\tau$-tilting finite if and only if it is representation-finite. For $\Lambda$ not hereditary, representation-finiteness implies $\tau$-tilting finiteness, but not vice versa. Another special class is that of \dfn{silting discrete} algebras \cite{AiharaMizuno}; an algebra $\Lambda$ belongs to this class when every algebra that is derived equivalent to $\Lambda$ (including $\Lambda$ itself) is $\tau$-tilting finite. All representation-finite hereditary algebras are silting discrete.

From now on, $\Lambda$ will always denote a $\tau$-tilting finite algebra. While many of the results of the following subsections admit generalizations to arbitrary finite-dimensional algebras, this allows us to streamline the exposition.

\subsection{Semibricks and 2-term simple-minded collections}\label{subsec:smc}

Recall that a module $X \in \mods\Lambda$ is called a \dfn{brick} if $\End_\Lambda(X)$ is a division ring. A set $\X$ of bricks is called a \dfn{semibrick} if $X \in Y^\perp \cap \lperp{Y}$ for all distinct $X,Y \in \X$. We adopt the convention of using the term $\X$ to represent both a set of bricks and the module $\bigoplus_{X \in \X} X$. We denote by $\brick\Lambda$ and $\sbrick\Lambda$ the sets of (isomorphism classes of) bricks and semibricks in $\mods\Lambda$, respectively. This leads to the following definition, which we state in this form in light of \cite[Remark~4.11]{BY}. See \cref{subsec:intro_torsion} for historical notes and references.

\begin{definition}
    Let $\X, \Y \in \sbrick(\Lambda)$.
    \begin{enumerate}[(1)]
        \item We say $(\X,\Y)$ is a \dfn{semibrick pair} if $\Y \in \X^{\perp_{0,1}}$.
        \item We say a semibrick pair $(\X,\Y)$ is a \dfn{2-term simple-minded collection} if the only triangulated subcategory of $\Db(\mods\Lambda)$ that contains $\X \oplus \Y[1]$ and is closed under direct summands is $\Db(\mods\Lambda)$ itself.
        \item We say a semibrick pair $(\X,\Y)$ is \dfn{completable} if there exists a 2-term simple-minded collection $(\X',\Y')$ with $\X \subseteq \X'$ and $\Y \subseteq \Y'$.
    \end{enumerate}
    We denote by $\sbp\Lambda$ the set of semibrick pairs and by $\smc\Lambda$ the set of 2-term simple-minded collections.
\end{definition}

\begin{samepage}
\begin{remark}\
    \begin{enumerate}[(1)]
        \item Not every semibrick pair is completable, even in the $\tau$-tilting finite case; see, e.g., \cite[Counterexample~1.9]{HansonIgusa_picture}.
        \item If $(\X,\Y)$ is a 2-term simple-minded collection, then $|\X| + |\Y| = |\Lambda|$. If the algebra $\Lambda$ is silting discrete, then the converse also holds by \cite[Corollary~6.12]{WemyssHara}; that is, if $(\X,\Y)$ is a semibrick pair with $|\X| + |\Y| = |\Lambda|$, then $(\X,\Y)$ must be a 2-term simple-minded collection.
    \end{enumerate}
\end{remark}\end{samepage}

\subsection{Semidistributivity and the brick labeling}

The lattice of torsion classes is known to be a (completely) semidistributive lattice \cite{DIRRT,GM}. Therefore, each edge in the Hasse diagram of $\tors\Lambda$ has a join-irreducible shard label as described at the end of \cref{sec:posets}. In this section, we recall from \cite{asai} ($\tau$-tilting finite case) and \cite{BCZ,DIRRT} (general case) how this labeling can be formulated and interpreted representation-theoretically. We note that many of the results recalled in this section have been specialized to the $\tau$-tilting finite case.

\begin{theorem}[\cite{BCZ,BTZ}]\label{thm:brick_label}
    Let $\Lambda$ be a $\tau$-tilting finite algebra. 
    \begin{enumerate}[(1)]
        \item There is a bijection $\brick\Lambda \rightarrow \jirr_{\tors\Lambda}$ given by $X \mapsto \tor(X)$.
        \item There is a bijection $\brick\Lambda \rightarrow \mirr_{\tors\Lambda}$ given by $X \mapsto \lperp{X}$.
        \item The bijection $\kappa\colon\jirr_{\tors\Lambda}\to\mirr_{\tors\Lambda}$ is given by $\kappa(\tor(X)) = \lperp{X}$ for all $X\in\brick\Lambda$. 
    \end{enumerate}
\end{theorem}

The \dfn{brick label} of a cover relation $\T \covered \T'$ in $\tors\Lambda$ is the unique brick $X$ for which $j_{\T,\T'} = \tor(X)$. For a fixed $\T \in \tors\Lambda$, we (slightly abusing notation) denote by $\D(\T)$ and $\U(\T)$ the sets of bricks that label cover relations of the form $\T' \covered \T$ and $\T \covered \T'$, respectively.

\begin{lemma}[{\cite[Corollary~3.9]{BCZ}}]\label{lem:proper_quotient}
    Let $\T \in \tors\Lambda$.
    \begin{enumerate}[(1)]
        \item If $Y \in \U(\T)$ and $Y \twoheadrightarrow Y'$ is a surjection with nonzero kernel, then $Y \in \T$.
        \item If $X \in \D(\T)$ and $X' \hookrightarrow X$ is an injection with nonzero cokernel, then $X' \in \T^\perp$.
    \end{enumerate}
\end{lemma}

\cref{thm:brick_label,lem:proper_quotient} immediately yield the following characterization of the pop-stack operators.

\begin{corollary}\label{cor:brick_label}
    For $\T \in \tors\Lambda$, we have
    \[\popdown_{\tors\Lambda}(\T) = \T \cap \lperp{\D(\T)}\quad\text{and}\quad\popup_{\tors\Lambda}(\T) = \Filt(\T \cup \U(\T)).\] 
\end{corollary}

    We now recall how \cref{thm:brick_label} relates to 2-term simple-minded collections. Once again, we note that the following theorem is simplified since we consider only the case where every torsion class is functorially finite. This result is essentially contained in \cite[Theorems~2.3 and~2.12]{asai}, but we give a short argument utilizing \cite[Corollary~5.1.8]{BTZ} and other results from \cite{asai}. Note that it is implicitly included in \cref{thm:torsion_smc}\ref{thm:torsion_smc_1} that the sets $\D(\T)$ and $\U(\T)$ are semibricks for any $\T \in \tors\Lambda$.

\begin{samepage}
\begin{theorem}\label{thm:torsion_smc}
    \hfill  
    \begin{enumerate}[(1)]
        \item\label{thm:torsion_smc_1} There are bijections $\sbrick\Lambda \rightarrow \tors\Lambda$ given by $\X \mapsto \tor(\X)$ and $\X \mapsto \lperp{\X}$. Their inverses are given by $\T \mapsto \D(\T)$ and $\T \mapsto \U(\T)$, respectively.
        \item\label{thm:torsion_smc_2} If $(\X,\Y) \in \sbp\Lambda$, then $\tor(\X) \subseteq \lperp{\Y}$, with equality if and only if $(\X,\Y) \in \smc\Lambda$. In particular, there is a bijection $\smc\Lambda \rightarrow \tors\Lambda$ given by $(\X,\Y) \mapsto \tor(\X) = \lperp{\Y}$. 
        The inverse is given by $\T \mapsto (\D(\T),\U(\T))$.
        \item\label{thm:torsion_smc_3} There are bijections $\smc\Lambda \rightarrow \sbrick\Lambda$ given by $(\X,\Y) \mapsto \X$ and $(\X,\Y) \mapsto \Y$.
    \end{enumerate}
\end{theorem}
\end{samepage}

\begin{proof}\ 

    \ref{thm:torsion_smc_1} The fact that these associations are bijections is from \cite[Theorem~2.3]{asai}. The explicit description of the inverses follows from \cite[Proposition~1.26]{asai} and \cite[Corollary~5.1.8]{BTZ}.

    \ref{thm:torsion_smc_2} That $\tor(X) \subseteq \lperp{\Y}$ follows immediately from the definition. The fact that we have equality if and only if $(\X,\Y) \in \smc\Lambda$ then follows from \cite[Theorem~2.3]{asai}.

    \ref{thm:torsion_smc_3} This follows from \cite[Theorem~2.3(i)]{asai}.
\end{proof}

In particular, \cref{thm:torsion_smc} implies that any semibrick pair of the form $(\X,\emptyset)$ or $(\emptyset,\Y)$ is completable. For $(\X,\Y) \in \smc\Lambda$, we write $\tor(\X,\Y) = \tor(\X) = \lperp{\Y}$.

\subsection{Wide and Serre subcategories}\label{subsec:wide}

A subcategory $\W \subseteq \mods\Lambda$ is called \dfn{wide} if it is closed under extensions, kernels, and cokernels. Alternatively, a subcategory is wide if it is an exact-embedded abelian subcategory. We denote by $\wide\Lambda$ the set of wide subcategories of $\mods\Lambda$.

It is a classical result of \cite{ringel} that there is a bijection $\sbrick\Lambda \rightarrow \wide\Lambda$ given by $\X \mapsto \Filt(\X)$. The inverse sends a wide subcategory $\W$ to the set of objects that are simple in $\W$ (i.e., those $X \in \W$ that admit no proper submodules in $\W$). Combining this with \cref{thm:torsion_smc}, we see that, under the $\tau$-tilting finite hypothesis, there are two bijections $\tors\Lambda \rightarrow \wide\Lambda$ given by $\T \mapsto \Filt(\D(\T))$ and $\T \mapsto \Filt(\U(\T))$. The inverses of these bijections are given by $\W \mapsto \tor(\W)$ and $\W \mapsto \lperp{\W}$, respectively. Moreover, these bijections correspond to the so-called \emph{Ingalls--Thomas bijections} of \cite{IT,MS}; that is, we have the following. (See also \cite[Remark~4.20]{enomoto}, which relates this result to the rowmotion operators.)

\begin{proposition}[{\cite[Theorem~2.3]{asai}}]\label{prop:IT} Let $\T \in \tors\Lambda$. 
    \begin{enumerate}[(1)]
        \item We have $\Filt(\D(\T)) = \{M \in \T \mid \ker f \in \T\text{ for all }N \in \T\text{ and }f\colon N \rightarrow M\}$.
        \item We have $\Filt(\U(\T)) = \{M \in \T^\perp \mid \coker f \in \T^\perp\text{ for all } N \in \T^\perp\text{ and }f\colon M \rightarrow N\}$.
    \end{enumerate}
\end{proposition}

\begin{remark}\label{rem:ff}
 Since $\Lambda$ is $\tau$-tilting finite, every wide subcategory of $\mods\Lambda$ is functorially finite by \cite[Corollary~3.11]{MS}. Equivalently (see \cite[Proposition~4.12]{enomoto_ff}), for every $\W \in \wide\Lambda$, there exists a finite-dimensional algebra $\Lambda'$ such that $\W$ is exact-equivalent to $\mods\Lambda'$.
 \end{remark}

We conclude this section with a brief discussion of Serre subcategories. A subcategory ${\W \subseteq \mods\Lambda}$ is called \dfn{Serre} if it is closed under extensions, submodules, and quotient modules. The following is well known.

\begin{proposition}\label{prop:serre}
    Let $\W \subseteq \mods\Lambda$. Then the following are equivalent.
    \begin{enumerate}[(1)]
        \item The subcategory $\W$ is Serre.
        \item The subcategory $\W$ is at least two of the following: a wide subcategory, a torsion class, and a torsion-free class.
        \item The subcategory $\W$ is a wide subcategory, a torsion class, and a torsion-free class.
        \item We have $\W = \Filt(\SS)$ for some set $\SS$ of modules that are simple in $\mods\Lambda$.
        \item There exists a projective module $P \in \mods\Lambda$ such that $\W = \rperp{P}$.
        \item There exists an injective module $I \in \mods\Lambda$ such that $\W = \lperp{I}$.
        \item The subcategory $\W$ is wide, and every object that is simple in $\W$ is also simple in $\mods\Lambda$.
    \end{enumerate}
\end{proposition}

Given a wide subcategory $\W$, we will also say that a wide subcategory $\W' \subseteq \W$ is \dfn{Serre in $\W$} if every object that is simple in $\W'$ is also simple in $\W$.

\section{Mutation of Semibrick Pairs}\label{sec:mutation}

In this section, we extend the theory of mutation of semibrick pairs established in \cite[Section~3]{HI_pairwise} and \cite[Section~3.4]{BaH_preproj}. This extends the mutation formulas for 2-term simple-minded collections from \cite[Section~7.2]{KY} and \cite[Section~3.7]{BY}. In those papers, one mutates a semibrick pair (resp.\ 2-term simple-minded collection) at a single brick. In the present paper, we extend this to be able to mutate at multiple bricks simultaneously.

Recall that $\Lambda$ denotes a $\tau$-tilting finite algebra, and let $(\X,\Y)$ be a semibrick pair. For $X \in \X$ and $\Y' \subseteq \Y$, denote by $g_{\Y',X}\colon Y'_X \rightarrow X$ a minimal right $\Filt(\Y')$-approximation of $X$. (Note that $Y'_X = 0$ if $\Y' = \emptyset$.) Likewise for $Y \in \Y$ and $\X' \subseteq \X$, denote by $g_{Y,\X'}\colon Y \rightarrow X'_Y$ a minimal left $\Filt(\X')$-approximation of $Y$. (Note that $X'_Y = 0$ if $\X' = \emptyset$.) We note that both $g_{\Y',X}$ and $g_{Y,\X'}$ exist and are unique up to isomorphism because $\Filt(\X')$ and $\Filt(\Y')$ are functorially finite wide subcategories; see \cref{subsec:ff,subsec:wide}.

\begin{definition}\label{def:SM}
    We say a semibrick pair $(\X,\Y)$ is \dfn{singly mutation (SM) compatible} if for all $X \in \X$, $Y\in\Y$, $\X'\subseteq\X$, and $\Y' \subseteq \Y$, the map $g_{\Y',X}$ is either injective or surjective and the map $g_{Y,\X'}$ is either injective or surjective.
\end{definition}

\begin{remark}
Note that SM compatibility is stronger than, but similar to, the notions of \emph{singly left mutation compatible} and \emph{singly right mutation compatible} used in \cite{BaH_preproj,HansonIgusa_picture,HI_pairwise}
\end{remark}

\begin{proposition}\label{prop:SM_compat}
    Every completable semibrick pair is SM compatible.
\end{proposition}

\begin{proof}
    It suffices to prove the result only for 2-term simple-minded collections. We use an argument similar to that of \cite[Proposition~5.2.1]{BTZ}. Let $(\X,\Y) \in \smc\Lambda$.
    
    Let $(\X,\Y) \in \smc\Lambda$, and denote $\T := \tor(\X)$ and $\F := \Filt(\Cogen \Y)$. Recall from \cref{thm:torsion_smc} that $(\T,\F)$ is a torsion pair, that $\T = \lperp{\Y}$ and $\F = \rperp{\X}$, and that $\X = \D(\T)$ and $\Y = \U(\T)$.
    
    Let $\X' \subseteq \X$ and $Y \in \Y$, and suppose that $g_{Y,\X'}$ is not injective.
    Then $Z := \mathrm{im}(g_{Y,\X'}) \in \T$ by \cref{lem:proper_quotient}. Denote by $\iota\colon Z \rightarrow X'_Y$ the inclusion map. Then for $W \in \T$ and $f\colon W \rightarrow Z$, we have $\ker(f) = \ker(\iota \circ f)$. Since $X'_Y \in \Filt(\X') \subseteq \Filt(\D(\T))$, \cref{prop:IT} implies that $\ker(f) \in \T$, and therefore that $Z \in \Filt(\X)$. The fact that $\Filt(\X')$ is Serre in $\Filt(\X)$ thus implies that $Z \in \Filt(\X')$. By the minimality of $g_{Y,\X'}$, we conclude that $Z = X'_Y$; i.e., $g_{Y,\X'}$ is surjective. The argument that each map $g_{X,\Y'}$ is either injective or surjective is dual.
\end{proof}

We now prepare to define the mutation of an SM compatible semibrick pair $(\X,\Y)$ at either a subset $\X' \subseteq \X$ or a subset $\Y' \subseteq \Y$. We will first formulate the mutation formulas within the category $\mods\Lambda$ (\cref{def:mutation}) and then explain how they can be restated using the language of the bounded derived category (\cref{rem:mutation_def}).

Fix an SM compatible semibrick pair $(\X,\Y)$ and subsets $\X'\subseteq\X$ and $\Y'\subseteq\Y$. Let $X \in \X \setminus \X'$. Since the wide subcategory $\Filt(\X)$ is functorially finite, it is well known that it contains a projective generator; that is, there exists $P_\X \in \Filt(\X)$ such that $\Filt(\X)$ is the subcategory obtained by closing $\add(P_\X)$ under cokernels. See \cite[Proposition~4.12]{enomoto_ff} for an explicit proof. Thus, we can consider a short exact sequence
    $$\Omega_{\X}X \hookrightarrow P_{\X,X} \twoheadrightarrow X,$$
where $P_{\X,X} \in \add(P_\X)$ and $\Omega_{\X}X \in \Filt(\X)$ are the projective cover and first syzygy of $X$ in the wide subcategory $\Filt(\X)$. Let $\gamma_{X,\X'}\colon \Omega_\X X \rightarrow X'_X$ be a minimal left $\Filt(\X')$-approximation of $\Omega_\X X$. Then $\gamma_{X,\X'}$ is surjective because $\Filt(\X')$ is a Serre subcategory of $\Filt(\X)$. Hence, we can form the following pushout diagram:
\begin{equation}\label{eqn:mutation_1}
    \begin{tikzcd}
        & \Omega_\X X \arrow[r,hookrightarrow]\arrow[d,two heads,"\gamma_{X,\X'}"] & P_{\X,X} \arrow[r,two heads]\arrow[d,two heads] & X \arrow[d,equals]\\
        \eta_{X,\X'}: &  X'_{X} \arrow[r,hookrightarrow] & E_{X,\X'} \arrow[r,two heads] & X.
    \end{tikzcd}
\end{equation}
We can consider the short exact sequence $\eta_{X,\X'} \in \Ext^1_\Lambda(X,X'_X)$ as a morphism $X \rightarrow X'_X[1]$ in the bounded derived category $\Db(\mods\Lambda)$. Then $\eta_{X,\X'}[1]$ is a minimal left $\Filt(\X')$-approximation of $X[-1]$ with cone $E_{X,\X'}$.

Dually, let $Y \in \Y \setminus \Y'$. Since $\Filt(\Y)$ is functorially finite, it contains an injective cogenerator; that is, there exists $I_\Y \in \Filt(\Y)$ such that $\Filt(\Y)$ is the subcategory obtained by closing $\add(I_\Y)$ under kernels. Thus, we can consider a short exact sequence
    $$Y \hookrightarrow I_{\Y,Y} \twoheadrightarrow \Sigma_{\Y} Y,$$
where $Y_{\Y,Y}$ and $\Sigma_\Y Y$ are the injective envelope and first cosyzygy of $Y$ in the wide subcategory $\Filt(\Y)$. Let $\gamma_{\Y',Y}\colon Y'_Y \rightarrow \Sigma_\Y Y$ be a minimal right $\Filt(\Y')$-approximation of $\Sigma_\Y Y$. Then $\gamma_{\Y',Y}$ is injective because $\Filt(\Y')$ is a Serre subcategory of $\Filt(\Y)$. Hence, we can form the following pullback diagram:
\begin{equation}\label{eqn:mutation_2}
    \begin{tikzcd}
        \eta_{\Y',Y}: & Y \arrow[r,hookrightarrow]\arrow[d,equals] & E_{\Y',Y} \arrow[r,two heads]\arrow[d,hookrightarrow] & Y'_Y \arrow[d,hookrightarrow,"\gamma_{\Y',Y}"]\\
        &  Y \arrow[r,hookrightarrow] & I_{\Y,Y} \arrow[r,two heads] & \Sigma_\Y Y.
    \end{tikzcd}
\end{equation}
As above, we can consider the short exact sequence $\eta_{\Y',Y} \in \Ext^1_\Lambda(Y'_Y,Y)$ as a morphism $Y'_Y \rightarrow Y[1]$ in the bounded derived category $\Db(\mods\Lambda)$. Then $\eta_{\Y',Y}$ is a minimal right $\Filt(\Y)$-approximation of $Y[1]$ with cocone $E_{\Y',Y}$.

We now present our generalized definition.

\begin{definition}\label{def:mutation}
    Let $(\X,\Y)$ be an SM compatible semibrick pair.
    \begin{enumerate}[(1)]
        \item\label{def:mutation_1} Let $\X' \subseteq \X$. Let
        \begin{eqnarray*}
                \mu_{\X'}(\X,\Y)_d &=& \{\coker(g_{Y,\X'}) \mid g_{Y,\X'} \text{ is injective}\} \cup \{E_{X,\X'} \mid X \in \X \setminus \X'\};\\
                \mu_{\X'}(\X,\Y)_u &=& \{\ker(g_{Y,\X'}) \mid g_{Y,\X'} \text{ is surjective}\} \cup \X'.
        \end{eqnarray*}
        The \dfn{left mutation} of $(\X,\Y)$ at $\X'$ is $\mu_{\X'}(\X,\Y) = (\mu_{\X'}(\X,\Y)_d,\mu_{\X'}(\X,\Y)_u)$.
    
        \item\label{def:mutation_2} Let $\Y' \subseteq \Y$. Let
            \begin{eqnarray*}
                \mu_{\Y'}(\X,\Y)_d &=& \{\coker(g_{\Y',X}) \mid g_{\Y',X} \text{ is injective}\} \cup \Y';\\
                \mu_{\Y'}(\X,\Y)_u &=& \{\ker(g_{\Y',X}) \mid g_{\Y',X} \text{ is surjective}\} \cup \{E_Y \mid Y \in \Y \setminus \Y'\}.
            \end{eqnarray*}
            The \dfn{right mutation} of $(\X,\Y)$ at $\Y'$ is $\mu_{\Y'}(\X,\Y) = (\mu_{\Y'}(\X,\Y)_d,\mu_{\Y'}(\X,\Y)_u)$.
        \end{enumerate}
\end{definition}

We note that $\mu_{\X'}(\X,\Y)$ and $\mu_{\Y'}(\X,\Y)$ are well defined by the existence and uniqueness (up to isomorphism) of projective covers, injective envelopes, minimal left $\Filt(\X')$-approximations, minimal right $\Filt(\Y')$-approximations, finite limits, and finite colimits. Note also that ${\mu_\emptyset(\X,\Y) = (\X,\Y)}$ by the definition of either left or right mutation. Since $\X \cap \Y = \emptyset$ for any semibrick pair, there is therefore no ambiguity in the notation.

Recalling that if $f\colon  M \rightarrow N$ is a monomorphism (resp.\ epimorphism) in $\mods\Lambda$, then its cone (resp.\ cocone) in $\Db(\mods\Lambda)$ is $\coker f$ (resp.\ $\ker f[1]$), we have the following reformulation of \cref{def:mutation}.

\begin{remark}\label{rem:mutation_def}
    Let $(\X,\Y) \in \sbp\Lambda$ be SM compatible, and consider $\X \oplus \Y[1] \in \Db(\mods\Lambda)$.
    \begin{enumerate}[(1)]
        \item Let $\X' \subseteq \X$. For $Z \in (\X \setminus \X') \cup \Y[1]$, let $g_{Z,\X'}: Z[-1] \rightarrow X'_Z$ be a minimal left $\Filt(\X')$-approximation in $\Db(\mods\Lambda)$. Then 
        $$\mu_{\X'}(\X,\Y)_d \oplus \mu_{\X'}(\X,\Y)_u = \{\mathrm{cone}(g_{Z,\X'}) \mid Z \in (\X \setminus \X') \cup \Y[1]\} \oplus \X'[1]\}.$$
        \item Let $\Y' \subseteq \Y$. For $Z \in (\Y[1] \setminus \Y'[1]) \cup \X$, let $g_{\Y',Z}[1]\colon  \Y'_Z[1] \rightarrow Z[1]$ be a minimal right $\Filt(\Y')[1]$-approximation in $\Db(\mods\Lambda)$. Then 
        $$\mu_{\Y'}(\X,\Y)_d \oplus \mu_{\Y'}(\X,\Y)_u = \{\mathrm{cocone}(g_{\Y',Z}[1]) \mid Z \in (\Y[1] \setminus \Y'[1]) \cup \X\} \oplus \Y'.$$
    \end{enumerate}
    In particular, in the cases $|\X'| = 1$ and $|\Y'| = 1$, \cref{def:mutation_1,def:mutation_2} of \cref{def:mutation} correspond to the notions of left and right mutations of semibrick pairs from \cite[Section~3]{HI_pairwise}.
\end{remark}

We conclude this section by stating the following result, which justifies the name \emph{mutation} and tabulates several useful facts about mutations of semibrick pairs. In \cref{sec:appendix}, we prove \cref{thm:mutation_summary} in multiple steps using arguments similar to those appearing in \cite[Section~3]{HI_pairwise} and \cite[Section~7.2]{KY}.

\begin{theorem}\label{thm:mutation_summary}
    Let $(\X,\Y) \in \sbp\Lambda$ be SM compatible. Let $\X' \subseteq \X$, and denote $\X_1 = \mu_{\X'}(\X,\Y)_d$ and $\Y_1 = \mu_{\X'}(\X,\Y)_u$. Then the following hold.
        \begin{enumerate}[(1)]
            \item\label{thm:mutation_1a} $(\X_1,\Y_1)$ is a semibrick pair.
            \item\label{thm:mutation_1b} If $(\X_1',\Y_1')$ is an SM compatible semibrick pair such that $\X_1 \subseteq \X'_1$ and $\Y_1 \subseteq \Y'_1$, then $\X \subseteq \mu_{\X'}(\X'_1,\Y'_1)_d$ and $\Y \subseteq \mu_{\X'}(\X'_1,\Y'_1)_u$.
            \item\label{thm:mutation_1c} $(\X,\Y) \in \smc\Lambda$ if and only if $(\X_1,\Y_1) \in \smc\Lambda$.
            \item\label{thm:mutation_1d} $(\X,\Y)$ is completable if and only if $(\X_1,\Y_1)$ is completable.
            \item\label{thm:mutation_1e} If $(\X,\Y) \in \smc\Lambda$, then $\tor(\mu_{\X'}(\X,\Y)) = \tor(\X,\Y) \cap \lperp{\X'}.$
        \end{enumerate}
            Similarly, let $\Y' \subseteq \Y$, and denote $\X_2 = \mu_{\Y'}(\X,\Y)_d$ and $\Y_2 = \mu_{\Y'}(\X,\Y)_u$. Then the following hold.
        \begin{enumerate}[(1*)]
            \item\label{thm:mutation_2a} $(\X_2,\Y_2)$ is a semibrick pair.
            \item\label{thm:mutation_2b} If $(\X_2',\Y_2')$ is an SM compatible semibrick pair such that $\X_2 \subseteq \X'_2$ and $\Y_2 \subseteq \Y'_2$, then $\X \subseteq \mu_{\Y'}(\X'_2,\Y'_2)_d$ and $\Y \subseteq \mu_{\Y'}(\X'_2,\Y'_2)_u$.
            \item\label{thm:mutation_2c} $(\X,\Y) \in \smc\Lambda$ if and only if $(\X_2,\Y_2) \in \smc\Lambda$.
            \item\label{thm:mutation_2d} $(\X,\Y)$ is completable if and only if $(\X_2,\Y_2)$ is completable.
            \item\label{thm:mutation_2e} If $(\X,\Y) \in \smc\Lambda$, then $\tor(\mu_{\Y'}(\X,\Y)) = \Filt(\Y' \cup \tor(\X,\Y))$.
        \end{enumerate}
\end{theorem}

\section{Pop-Stack Operators for Torsion Classes}\label{sec:pop-stack_description}

Recall the definition of the pop-stack and dual pop-stack operators from \Cref{sec:intro}. For $\Lambda$ a $\tau$-tilting finite algebra, recall the description of $\popdown_{\tors\Lambda}$ and $\popup_{\tors\Lambda}$ from \cref{cor:brick_label}. In this section, we study further properties of these operators. In \cref{subsec:mutation_pop}, we explain how the pop-stack operators interact with the mutation of 2-term simple-minded collection. In \cref{subsec:preimages}, we describe preimages under the pop-stack operators. In \cref{subsec:image}, we describe the images of the pop-stack operators.

\subsection{Pop-stack and mutation}\label{subsec:mutation_pop}

We now prove our first main result, which describes the relationship between the pop-stack operators and the mutation of 2-term simple-minded collections.

\begin{theorem}\label{thm:pop_mutation}
    For $(\X,\Y) \in \smc\Lambda$, let $\mu_\downarrow(\X,\Y) := \mu_\X(\X,\Y)$ and $\mu_\uparrow(\X,\Y) := \mu_\Y(\X,\Y)$. There are commutative diagrams as follows:
    $$
    \begin{tikzcd}
        \smc \Lambda \arrow[r,"\tor(-)"] \arrow[d,"\mu_\downarrow(-)"] & \tors\Lambda\arrow[d,"\popdown(-)"] && \smc \Lambda \arrow[r,"\tor(-)"] \arrow[d,"\mu_\uparrow(-)"] & \tors\Lambda\arrow[d,"\popup(-)"]\\
        \smc\Lambda \arrow[r,"\tor(-)"] & \tors\Lambda && \smc\Lambda \arrow[r,"\tor(-)"] & \tors\Lambda.
    \end{tikzcd}
    $$
\end{theorem}

\begin{proof}
    Let $(\X,\Y) \in \smc\Lambda$. Then by \cref{thm:mutation_summary}\ref{thm:mutation_1e} and \cref{cor:brick_label}, we have that
    \begin{eqnarray*}
        \tor(\mu_\downarrow(\X,\Y)) &=& \tor(\X) \cap \lperp{\X}\\
            &=& \tor(\X) \cap \left(\bigcap \left\{\tor(\X) \cap \lperp{X} \mid X \in \X\right\}\right)\\
            &=& \popdown(\tor(\X,\Y)).
    \end{eqnarray*}
    This shows that the left diagram commutes. Similarly, \cref{thm:mutation_summary}\ref{thm:mutation_2e} and \cref{cor:brick_label} imply that
    \begin{eqnarray*}
        \tor(\mu_\uparrow(\X,\Y)) &=& \Filt(\lperp\Y \cup \Y)\\
            &=& \Filt\left(\lperp{\Y} \cup \left(\bigcup \left\{\Filt(\lperp{\Y} \cup Y)\mid Y \in \Y\right\}\right)\right)\\
            &=& \popup(\tor(\X,\Y)).
    \end{eqnarray*}
    This shows that the right diagram commutes.
\end{proof}

\begin{remark}\label{rem:ungar}
    One can also generalize the proof of \cref{thm:pop_mutation} to show that the torsion classes of the form $\tor(\mu_{\X'}(\X,\Y))$ (with $\X' \subseteq \X$) are precisely those obtained by intersecting $\tor(\X)$ with a subset of the torsion classes it covers. (The dual result likewise holds for right mutation.) In the language of \cite{ungar_chains,ungar_games}, these are precisely the torsion classes obtained by applying \emph{Ungar moves} to $\tor(\X)$. Intervals of the form $[\tor(\mu_{\X'}(\X,\Y)),\tor(\X)]$ have also appeared in many of the previously-cited papers on lattices of torsion classes such as \cite{AP,facial_torsion}.
\end{remark}

\subsection{Preimages under the pop-stack operators}\label{subsec:preimages}

We now characterize the preimages of a given torsion class under the pop-stack operator and its dual. As a consequence, we describe the 1-pop-stack sortable elements and the 2-pop-stack sortable elements.

\begin{theorem}\label{thm:preimage}
    Let $\T, \T' \in \tors\Lambda$. Then  $\popdown_{\tors\Lambda}(\T') = \T$ if and only if both of the following hold.
        \begin{enumerate}[(1)]
            \item\label{thm:preimage_1a} There is an inclusion $\D(\T') \subseteq \U(\T)$.
            \item\label{thm:preimage_1b} For every $X \in \D(\T)$, there exists $Z \in \Filt(\D(\T'))$ admitting a surjection $Z \twoheadrightarrow X$.
        \end{enumerate}
        Moreover, if $\popdown_{\tors\Lambda}(\T') = \T$, then $(\D(\T'),\U(\T')) = \mu_{\D(\T')}(\D(\T),\U(\T))$.
        
        \noindent Dually, $\popup_{\tors\Lambda}(\T) = \T'$ if and only if both of the following hold.
        \begin{enumerate}[(1*)]
            \item There is an inclusion $\U(\T) \subseteq \D(\T')$. 
            \item For every $Y \in \U(\T')$, there exists $Z \in \Filt(\U(\T))$ admitting an injection $X \hookrightarrow Z$.
        \end{enumerate}
        Moreover, if $\popup_{\tors\Lambda}(\T) = \T'$, then $(\D(\T),\U(\T)) = \mu_{\U(\T)}(\D(\T'),\U(\T'))$.
\end{theorem}

\begin{proof}
    We prove only the first half of the theorem; the proof of the second half is dual.

    First suppose $\popdown_{\tors\Lambda}(\T') = \T$. Then \[(\D(\T),\U(\T)) = \mu_\downarrow(\D(\T'),\U(\T')) = \mu_{\D(\T')}(\D(\T'),\U(\T'))\] by \cref{thm:pop_mutation}, so in particular, \ref{thm:preimage_1a} holds. It follows from \cref{thm:mutation_summary}\ref{thm:mutation_1b} that \[{\mu_{\D(\T')}(\D(\T),\U(\T)) = (\D(\T'),\U(\T'))}.\] Finally, because $\mu_{\D(\T')}(\D(\T),\U(\T)) = (\D(\T'),\U(\T'))$, \ref{thm:preimage_1b} must hold by \cref{def:mutation}.

    Suppose now that \ref{thm:preimage_1a} and \ref{thm:preimage_1b} both hold. Then $\mu_{\D(\T')}(\D(\T),\U(\T))_d = \D(\T')$ by \cref{def:mutation}, and therefore $\tor(\mu_{\D(\T')}(\D(\T),\U(\T))) = \T'$. \cref{thm:mutation_summary}\ref{thm:mutation_1b} then implies that \[\mu_\downarrow(\D(\T'),\U(\T')) = \mu_{\D(\T')}(\D(\T'),\U(\T')) = (\D(\T),\U(\T)).\] We conclude that $\popdown_{\tors\Lambda}(\T') = \T$ by \cref{thm:pop_mutation}.
\end{proof}

We now characterize the 1-pop-stack sortable elements of $\tors\Lambda$.

\begin{corollary}\label{cor:tors_one_poppable}
    Let $\T \in \tors\Lambda$. The following are equivalent.
    \begin{enumerate}[(1)]
        \item\label{cor:tors_1} We have $\popdown_{\tors\Lambda}(\T) = \hat 0$.
        \item\label{cor:tors_2} Every brick in $\D(\T)$ is simple.
        \item\label{cor:tors_3} The torsion class $\T$ is a Serre subcategory of $\mods\Lambda$.
        \item\label{cor:tors_4} There exists a nonzero injective module $I$ such that $\T = \lperp{I}$.
    \end{enumerate}
    Dually, the following are equivalent.
    \begin{enumerate}[(1*)]
        \item We have $\popup_{\tors\Lambda}(\T) = \mods\Lambda$.
        \item Every brick in $\U(\T)$ is simple.
        \item The torsion-free class $\T^\perp$ is a Serre subcategory of $\mods\Lambda$.
        \item There exists a nonzero projective module $P$ such that $\T = \Gen(P)$.
    \end{enumerate}
\end{corollary}

\begin{proof}
    We prove only the first result as the second is dual. The equivalences among \ref{cor:tors_2}, \ref{cor:tors_3}, and \ref{cor:tors_4} are contained in \cref{prop:serre}.

    To see that \ref{cor:tors_1} implies \ref{cor:tors_2}, suppose $\popdown_{\tors\Lambda}(\T) = \hat 0$. Then every brick in $\U(\popdown_{\tors\Lambda}(\T))$ is simple. But $\D(\T') \subseteq \U(\popdown_{\tors\Lambda}(\T))$ by \cref{thm:preimage}.

    We now show that \ref{cor:tors_3} implies \ref{cor:tors_1}. Suppose $\T$ is a Serre subcategory. Then in particular, $\T$ is closed under submodules. Thus, for $Y \in \U(\T)$, the map $g_{Y,\X}$ must be surjective. \cref{thm:pop_mutation} then implies that $\D(\popdown_{\tors\Lambda}(\T)) = \emptyset$, so $\popdown_{\tors\Lambda}(\T) = 0$.
\end{proof}

The following characterizes all 2-pop-stack sortable elements of $\tors\Lambda$; it also describes the image under $\popdown_{\tors\Lambda}$ of each 2-pop-stack sortable element.

\begin{corollary}\label{cor:tors_2_poppable}
    Let $\T \in \tors\Lambda$, and let $\mathcal{S}$ be a set of simple modules. Then $\popdown_{\tors\Lambda}(\T) = \Filt(\SS)$ if and only if the following all hold.
    \begin{enumerate}[(1)]
        \item\label{cor:tors_2_poppable_1} We have $\Hom_\Lambda(\SS,\D(\T)) = 0 = \Ext^1_\Lambda(\SS,\D(\T))$.
        \item\label{cor:tors_2_poppable_2} The socle\footnote{Recall that the socle $\mathrm{soc}M$ of a module $M$ is the sum of all of its semisimple submodules.} $\mathrm{soc}(\D(\T))$ of $\D(\T)$ satisfies $\D(\T)/\mathrm{soc}(\D(\T)) \in \Filt(\SS)$.
        \item\label{cor:tors_2_poppable_3} There does not exist $S \in \SS$ with $\Hom_\Lambda(\D(\T),S) = 0$. (Equivalently, there is a surjection $\D(\T) \twoheadrightarrow \SS$.)
    \end{enumerate}
\end{corollary}

\begin{proof}
    Suppose first that $\popdown_{\tors\Lambda}(\T) = \Filt(\SS)$. Since $\D(\Filt(\SS)) = \SS$, it follows from \cref{thm:preimage} that $\D(\T) \subseteq \U(\Filt(\SS))$. Then the fact that \ref{cor:tors_2_poppable_1} holds follows from the definition of a 2-term simple-minded collection. \cref{thm:preimage} also implies that for every $S \in \SS$, there exists $Z \in \Filt(\D(\T))$ that admits a surjection $Z \twoheadrightarrow S$. This shows that \ref{cor:tors_2_poppable_3} holds. To prove \ref{cor:tors_2_poppable_2}, consider ${X \in \D(\T) \subseteq \U(\Filt(\SS))}$. Since $\Filt(\SS)$ is closed under submodules, the map $g_{X,\SS}\colon X \rightarrow S_X$ must be surjective. Moreover, we have that $\popdown_{\tors\Lambda}(\Filt(\SS)) = 0$ by \cref{cor:tors_one_poppable}. \cref{thm:pop_mutation} implies that $\ker g_{X,\SS}$ is in $\U(0)$, which is precisely the set of simple modules. If follows that $\ker g_{X,\SS} \subseteq \mathrm{soc}(X)$ and therefore that $X/\mathrm{soc}(X) \in \Filt(\SS)$.

    To prove the converse, suppose $\T$ satisfies \cref{cor:tors_2_poppable_1,cor:tors_2_poppable_2,cor:tors_2_poppable_3}. Then \ref{cor:tors_2_poppable_1} implies that $(\SS,\D(\T))$ is a semibrick pair. Moreover, for $X \in \D(\T)$, the map $g_{X,\SS}$ must be surjective since $\Filt(\SS)$ is closed under submodules. This map must also have a semisimple kernel by \ref{cor:tors_2_poppable_2}. We then have that $\mu_{\SS}(\SS,\D(\T)) = (\emptyset,\SS \cup \SS')$ for $\SS' = \{\ker g_{X,\SS} \mid X \in \D(\T)\}$. \cref{prop:mutation} therefore implies that $\ker g_{X,\SS}$ is simple for all $X$. Letting $\SS''$ denote the set of all simple modules, we find that $(\emptyset,\SS'')$ is a 2-term simple-minded collection satisfying $\SS \cup \SS' \subseteq \SS''$. By \cref{prop:mutation_completable}, $(\SS,\Y):=\mu_{\SS}(\emptyset,\SS'')$ is a 2-term simple-minded collection satisfying $\D(\T) \subseteq \Y$. Now note that $\Y = \U(\Filt(\SS))$ by \cref{thm:torsion_smc}. Hence, $\D(\T) \subseteq \U(\Filt(\SS))$. Together with \ref{cor:tors_2_poppable_3}, this allows us to use \cref{thm:preimage} to conclude that $\popdown_{\tors\Lambda}(\T) = \Filt(\SS)$.
\end{proof}

\subsection{The image of pop-stack}\label{subsec:image}

We now build toward a complete description of the images of $\popdown_{\tors\Lambda}$ and $\popup_{\tors\Lambda}$ (and, hence, the facets of the canonical join complex of $\tors\Lambda$) for any $\tau$-tilting finite algebra $\Lambda$. Our first result is proven more generally for arbitrary semidistrim lattices in \cite[Section~9]{DefantWilliamsSemidistrim}. We give a new representation-theoretic proof for lattices of torsion classes.

\begin{proposition}\label{prop:image_up_down}
    Let $\T \in \tors\Lambda$. 
    \begin{enumerate}[(1)]
        \item\label{prop:image_1} We have $\popdown_{\tors\Lambda}(\popup_{\tors\Lambda}(\T)) \subseteq \T$.
        \item\label{prop:image_2} We have $\popup_{\tors\Lambda}(\popdown_{\tors\Lambda}(\T)) \supseteq \T$.
        \item\label{prop:image_3} We have $\popdown_{\tors\Lambda}(\popup_{\tors\Lambda}(\popdown_{\tors\Lambda}(\T))) = \popdown_{\tors\Lambda}(\T)$.
        \item\label{prop:image_4} We have $\popup_{\tors\Lambda}(\popdown_{\tors\Lambda}(\popup_{\tors\Lambda}(\T))) = \popup_{\tors\Lambda}(\T)$.
    \end{enumerate}
\end{proposition}

\begin{proof}
    We prove only \ref{prop:image_1} and \ref{prop:image_3} as the proofs of \ref{prop:image_2} and \ref{prop:image_4} are similar. 

    We have $\U(\T) \subseteq \D(\popup_{\tors\Lambda}(\T)) \subseteq \U(\popdown_{\tors\Lambda}(\popup_{\tors\Lambda}(\T)))$ by \cref{thm:pop_mutation}. By \cref{thm:brick_label}, this implies \ref{prop:image_1}.

    To prove \ref{prop:image_3}, let $(\X,\Y) = (\D(\T),\U(\T))$. Let $(\X_1,\Y_1) = \mu_\downarrow(\X,\Y)$ and $(\X_2,\Y_2) = \mu_\uparrow(\X_1,\Y_1)$. By \cref{thm:pop_mutation}, we need to show that $\mu_\downarrow(\X_2,\Y_2) = (\X_1,\Y_1)$. By \cref{def:mutation}, this is equivalent to showing that $\X_2 = \Y_1$, which is equivalent to showing that $g_{\Y_1,Z}$ is surjective for every $Z \in \X_1$.

    Now recall that $\X_1 = \{\coker(g_{Y,\X}) \mid Y \in \Y \text{ and } g_{Y,\X} \text{ is injective}\}$. Thus, for $Z \in \X_1$, we have a quotient map $X_Y \twoheadrightarrow Z$ for some $X_Y \in \Filt(\X)$. Since $\X \subseteq \Y_1$, this map must factor through the minimal right $\Filt(\Y_1)$-approximation $g_{\Y_1,Z}\colon (Y_1)_Z \rightarrow Z$. This implies that $g_{\Y_1,Z}$ must be surjective.
\end{proof}

As a consequence, we can characterize the image of the pop-stack operators as follows.

\begin{corollary}\label{cor:algebraic_image_characterization}
    Let $\T \in \tors\Lambda$. Then the following are equivalent.
    \begin{enumerate}[(1)]
        \item There exists $\T' \in \tors\Lambda$ such that $\T = \popdown_{\tors\Lambda}(\T')$.
        \item We have $\popdown_{\tors\Lambda}(\popup_{\tors\Lambda}(\T)) = \T$.
        \item The semibricks $\U(\T)$ and $\D(\popup_{\tors\Lambda}(\T))$ coincide.
        \item For all $X \in \D(\T)$, the map $g_{\U(\T),X}$ is surjective.
        \item For all $X \in \D(\T)$, there exists $Z \in \Filt(\U(\T))$ such that there is a surjective map $Z \twoheadrightarrow X$.
    \end{enumerate}
    Dually, the following are equivalent.
    \begin{enumerate}[(1*)]
        \item There exists $\T' \in \tors\Lambda$ such that $\T = \popup_{\tors\Lambda}(\T')$.
        \item We have $\popdown_{\tors\Lambda}(\popup_{\tors\Lambda}(\T)) = \T$.
        \item The semibricks $\D(\T)$ and $\U(\popdown_{\tors\Lambda}(\T))$ coincide.
        \item For all $Y \in \U(\T)$, the map $g_{Y,\D(\T)}$ is injective.
        \item For all $Y \in \U(\T)$, there exists $Z \in \Filt(\D(\T))$ such that there is an injective map $Y \hookrightarrow Z$.
    \end{enumerate}
\end{corollary}

We now wish to further characterize the image of the pop-stack operators for representation-finite hereditary algebras. We first prove the following result (still in the generality of an arbitrary $\tau$-tilting finite algebra).

\begin{samepage}
\begin{lemma}\label{lem:no_projective}
    Let $\T \in \tors\Lambda$.
    \begin{enumerate}[(1)]
        \item\label{lem:projective_1} If there exists a nonzero projective module $P \in \T$, then $\T$ is not in the image of $\popdown_{\tors\Lambda}$.
        \item\label{lem:projective_2} If there exists a nonzero injective module $I \in \T^\perp$, then $\T$ is not in the image of $\popup_{\tors\Lambda}$.
    \end{enumerate}
\end{lemma}
\end{samepage}

\begin{proof}
    We prove only \ref{lem:projective_1} since \ref{lem:projective_2} is dual. It suffices to consider the case where $P$ is indecomposable.

    We claim that there exists $X \in \D(\T)$ such that there is a surjection $P \twoheadrightarrow X$ and there does not exist $Z \in \Filt(\U(\T))$ admitting a surjection $Z \twoheadrightarrow X$. We prove the claim by induction on the length of the shortest path from $\Gen(P)$ to $\T$ in $\tors\Lambda$.

    For the base case, we consider $\T = \Gen(P)$. In this case, there is a short exact sequence
    $$0 \rightarrow M \rightarrow P \xrightarrow{f} X \rightarrow 0$$
    such that $M \in \T$ and $\D(\T) = X$. (If $P$ is a brick, then $X = P$ and $M = 0$. Otherwise, $M$ is the sum of the images of the noninvertable endomorphisms of $P$.) Now suppose for a contradiction that there exist $Z \in \Filt(\U(\T))$ and a surjection $Z \xrightarrow{g} X$. Then by the definition of \emph{projective}, there is a nonzero map $h\colon P \rightarrow Z$ such that $f = g \circ h$. But $Z \in \Filt(\U(\T)) \subseteq P^\perp$, which is a contradiction.

    For the induction step, suppose $\Gen(P) \subsetneq \T$, and choose some $\T' \covered \T$ such that $\Gen(P) \subseteq \T'$. Let $Y \in \U(\T')$ be the brick label of the cover relation $\T'\covered\T$. By the induction hypothesis, there exists $X \in \D(\T')$ such that there is a surjection $P \twoheadrightarrow X$ and there does not exist $Z \in \Filt(\U(\T'))$ admitting a surjection $Z \twoheadrightarrow X$. Let $g_{Y,X}\colon Y_X \rightarrow X$ be a minimal right $\Filt(Y)$-approximation of $X$. Then $g_{Y,X}$ is injective since there is no surjection from an object in $\Filt(\U(\T'))$ onto $X$.  Thus, $\coker(g_{Y,X}) \in \D(\T)$ by \cref{thm:mutation_summary}. Moreover, $\coker(g_{Y,X})$ is a quotient of $P$ since $X$ is, and we have that $\T^\perp \subseteq P^\perp$. Consequently, no map from an object in $\Filt(\U(\T))$ to $\coker(g_{Y,X})$ can be surjective for the same reason as in the base case.
    
    Given the claim, the result follows from \cref{cor:algebraic_image_characterization}.
\end{proof}

Note that a torsion class $\T$ contains a projective module $P$ if and only if $\Gen(P) \subseteq \T$. Moreover, recall that the atoms of $\tors\Lambda$ are precisely the torsion classes of the form $\Filt(S)$ for $S \in \mods\Lambda$ simple. We can thus give a combinatorial interpretation of what it means for a torsion class to contain a projective module as follows. 

\begin{proposition}\label{prop:proj}
        For $\T \in \tors\Lambda$, the following are equivalent.
        \begin{enumerate}[(1)]
            \item\label{prop:proj_1} There exists a projective module $P$ such that $\T = \Gen(P)$.
            \item\label{prop:proj_2} There exists a set $\mathcal A$ of atoms of $\tors\Lambda$ such that $\T$ is the maximal element of $\tors\Lambda$ that lies (weakly) above all of the atoms in $\mathcal{A}$ and does not lie (weakly) above any of the atoms that are not in $\mathcal{A}$.
            \item\label{prop:proj_3} There exists a set $\mathcal{S}$ of simple modules such that $\T$ is the largest torsion class that contains all of the modules in $\mathcal{S}$ and does not contain any simple module that is not in $\mathcal{S}$.
        \end{enumerate}
        When these conditions hold, the semibrick $\mathcal{S}$ is the top of the module $P$ (equivalently, $P$ is the projective cover of $\mathcal{S}$); in particular, $P = 0$ if and only if $\mathcal{A} = \emptyset = \mathcal{S}$.
\end{proposition}

\begin{proof}
    The equivalence of \ref{prop:proj_2} and \ref{prop:proj_3} follows from the description of the atoms of $\tors\Lambda$.

    Let us prove that \ref{prop:proj_1} implies \ref{prop:proj_3}. Let $P$ be a projective module, and let $\SS$ be the set of simple modules that are direct summands of $\top(P)$. Then $\SS$ is a quotient of $P$, so $\Filt(\SS) \subseteq \Gen(P)$. Moreover, for a simple $S \notin \SS$, we have $\Hom_\Lambda(P,S) = 0$, so $S \notin \Gen(P)$. It remains to show that $\Gen(P)$ is maximal with respect to these properties. Let $\T$ be another torsion class satisfying $\SS \in \T$ and $S \notin \T$ for any simple $S \notin \SS$. Then $\top(\D(\T)) \in \add(\SS) \subseteq \Filt(\SS)$. Thus, the projective cover of $\D(\T)$ lies in $\add(P)$, so $\D(\T) \in \Gen(P)$. It follows that $\T \subseteq \Gen(P)$.

    We now show that \ref{prop:proj_3} implies \ref{prop:proj_1}. Let $\SS$ be a set of simple modules, and let $\T$ be the largest torsion class of $\tors\Lambda$ that contains $\SS$ and does not contain any simple that does not lie in $\SS$. Then, as shown above, we have $\T = \Gen(P)$, where $P$ is the projective cover of $\SS$.
\end{proof}

We now give a characterization of the image of the pop-stack operator when $\Lambda$ is hereditary (and representation-finite). We give a combinatorial interpretation of this characterization in \cref{thm:combinatorial_image_description}.

\begin{theorem}\label{thm:algebraic_image_characterization}
    Suppose $\Lambda$ is hereditary, and let $\T \in \tors\Lambda$. Then $\T$ is in the image of $\popdown_{\tors\Lambda}$ if and only if both of the following hold:
    \begin{enumerate}[(1)]
        \item\label{thm:algebraic_1} $\Ext^1_{\Lambda}(X,X') = 0$ for all $X, X' \in \D(\T)$.
        \item\label{thm:algebraic_2} There does not exist a nonzero projective module $P \in \T$.
    \end{enumerate}
    Dually, $\T$ is in the image of $\popup_{\tors\Lambda}$ if and only if both of the following hold.
    \begin{enumerate}[(1*)]
        \item $\Ext^1_{\Lambda}(Y,Y') = 0$ for all $Y, Y' \in \U(\T)$.
        \item There does not exist a nonzero injective module $I \in \T^\perp$.
    \end{enumerate}
\end{theorem}

\begin{proof}
    We prove only the first statement since the second is dual. We first note that \ref{thm:algebraic_2} is necessary by \cref{lem:no_projective}.

    Suppose that \ref{thm:algebraic_1} holds and that $\T$ is not in the image of $\popdown_{\tors \Lambda}$. By \cref{cor:algebraic_image_characterization}, this means there exists $X \in \D(\T)$ such that $g_{\U(\T),X}$ is injective. We will prove that $X$ is projective by induction on $|\D(\T)|$. First suppose $|\D(\T)| = 1$. Then \cref{thm:pop_mutation} implies that ${|\D(\popup_{\tors \Lambda}(\T))| = |\Lambda|}$. Now, $|\D(\popup_{\tors \Lambda}(\T))| + |\U(\popup_{\tors \Lambda}(\T))| = |\Lambda|$ by \cref{thm:torsion_smc} and \cite[Corollary~5.5]{KY}\footnote{This is a classical and well-known fact that was proved much earlier than \cite{KY}, but this logical deduction suits the organization of this paper.}, so $\popup_{\tors \Lambda}(\T) = \mods\Lambda$. In particular, $\U(\T) \subseteq \D(\popup_{\tors \Lambda}(\T))$ consists of only simple modules, so $\T^\perp$ is a Serre subcategory. By \cref{prop:serre}, this means $\T^\perp = X^\perp = P^\perp$ for some projective module $P$. It is straightforward to show that $X = P$ must be the projective cover of the unique simple module that does not lie in $\U(\T)$.

    For the induction step, suppose that the result holds for $|\D(\T)| = k$ and that $|\D(\T)| = k + 1$. Let $X \neq Z \in \D(\T)$, and let $\D' := \D(\T) \setminus \{Z\}$. By the assumption of \ref{thm:algebraic_2}, we have that $\D' \cup \U(\T) \subseteq Z^{\perp_{0,1}}$. Furthermore, since $\Lambda$ is hereditary, $Z^{\perp_{0,1}}$ is a wide subcategory\footnote{The \emph{perpendicular subcategory} $Z^{\perp_{0,1}}$ was first considered in \cite{GL}.}. Consequently, $\T' := Z^\perp \cap \T$ is a torsion class of $Z^{\perp_{0,1}}$ that satisfies $\D(\T') = \D'$ and $\U(\T') = \U(\T)$. In particular, \ref{thm:algebraic_1} holds for the torsion class $\T'$, and the injective map $g_{X,\U(\T)}$ is a minimal right $\Filt(\U(\T'))$-approximation. By the induction hypothesis, we conclude that $X$ is projective in $Z^{\perp_{0,1}}$; i.e., $\Ext^1_{\Lambda}(X,M) = 0$ for all $M \in Z^{\perp_{0,1}}$.

    Recall from \cref{subsec:torsion} that every $M \in \mods \Lambda$ admits a short exact sequence
    ${t_\T M \hookrightarrow M \twoheadrightarrow f_\F M}$
    with $t_\T M \in \T$ and $f_\F M \in \T^\perp = \Filt(\Cogen(\U(\T)))$. The module $t_\T M$ then fits into a short exact sequence $t_{\Gen(Z)}(t_\T M) \hookrightarrow t_\T M \twoheadrightarrow f_{Z^\perp}(t_\T M)$ with $t_{\Gen(Z)}(t_\T M) \in \Gen(Z)$ and $f_{Z^\perp}(t_\T M) \in Z^\perp$. Note that $f_{Z^\perp}(t_\T M)$ is in $\T'$ since it is a quotient of $t_\T M$ and $t_\T M \in \T$. By the above paragraph, we have  $\Ext^1_\Lambda(X,f_{Z^\perp}(t_\T M)) = 0$. Similarly, $\Ext^1_\Lambda(X,t_{\Gen(Z)}(t_\T M)) = 0$ since $\Ext^1_\Lambda(X,Z) = 0$ (by assumption) and $\Lambda$ is hereditary. Thus, to show that $X$ is projective in $\mods\Lambda$, it suffices to show that $\Ext^1_{\Lambda}(X,N) = 0$ for all $N \in \Cogen(\U(\T))$. Let $Y \in \U(\T)$, and let $N \subseteq Y$ be a submodule. Then $N \in Z^\perp$. If $\Ext^1_{\Lambda}(Z,N) = 0$, we already know that $\Ext^1_{\Lambda}(X,N) = 0$. Otherwise, the fact that $\Ext^1_{\Lambda}(Z,Z) = 0$ and $\Hom_{\Lambda}(Z,N) = 0$ means that there exist $Z^k \in \add(Z)$ and a short exact sequence
    $$0 \rightarrow N \rightarrow E \rightarrow Z^k \rightarrow 0$$
    such that the induced map $\Hom_{\Lambda}(Z,Z^k) \rightarrow \Ext^1_{\Lambda}(Z,N)$ is surjective. If we take $k$ to be minimal with respect to this property, then there is an induced long exact sequence
    $$0 = \Hom_{\Lambda}(Z,E) \rightarrow \Hom_{\Lambda}(Z,Z^k) \rightarrow \Ext^1_{\Lambda}(Z,N) \rightarrow \Ext^1_{\Lambda}(Z,E) \rightarrow \Ext^1_{\Lambda}(Z,Z) = 0.$$
    (We know that the first term is $0$ because $Z$ is a brick and $\Hom_{\Lambda}(Z,N) = 0$, so any nonzero morphism $Z \rightarrow E$ would have trivial intersection with $N$ and would consequently be a splitting of a composition $E \twoheadrightarrow Z^k \twoheadrightarrow Z$.) It follows that $E \in Z^{\perp_{0,1}}$ and therefore that $\Ext^1_{\Lambda}(X,E) = 0$ (since $X$ is projective in that wide subcategory). Hence, there is an induced exact sequence
    $$0 = \Hom_{\Lambda}(X,Z^k) \rightarrow \Ext^1_{\Lambda}(X,N) \rightarrow \Ext^1_{\Lambda}(X,E) = 0.$$
    We conclude that $\Ext^1_{\Lambda}(X,N) = 0$. Since $\Cogen(\U(\T))$ is the closure under $\Filt$ of all indecomposable quotients of the bricks in $\U(\T)$, this proves that $X$ is projective in $\mods\Lambda$.
    
    It remains to show that if $\T$ is in the image of $\popdown_{\tors \Lambda}$, then \ref{thm:algebraic_1} holds. To see this, suppose that $\T$ is in the image of $\popdown_{\tors \Lambda}$, and let $X_1, X_2 \in \D(\T)$. If $X_1 = X_2$, then there is nothing to show since all bricks over $\Lambda$ are also rigid. Otherwise, by \cref{cor:algebraic_image_characterization}, the map $g_{\U(\T),X_2}: U_{X_2} \rightarrow X_2$ must be injective. Therefore, there is an induced exact sequence
    $$\Ext^1_{\Lambda}(X_1,\ker(g_{\U(\T),X_2})) \rightarrow \Ext^1_{\Lambda}(X_1,U_{X_2}) \rightarrow \Ext^1_{\Lambda}(X_1,X_2) \rightarrow 0,$$
    where the last 0 comes from the fact that $\Lambda$ is hereditary. Since $\Ext^1_{\Lambda}(X_1,U_{X_2}) = 0$, it follows that $\Ext^1_{\Lambda}(X_1,X_2) = 0$; i.e., \ref{thm:algebraic_1} holds.
\end{proof}

\begin{remark}
If $\Lambda$ is not hereditary, then the Ext-orthogonality of the bricks in $\D(\T)$ is generally not a necessary or sufficient condition for $\T$ to be in the image of $\popdown_{\tors\Lambda}$. See \cref{rem:ext_orthogonal} for an explicit example.
\end{remark}

We conclude this section by combining several known results into a lattice-theoretic interpretation of the Ext-orthogonality condition in \cref{thm:algebraic_image_characterization}. As we discuss in \cref{rem:preprojective}, a special case of this result can also be interpreted (and proved) using Coxeter combinatorics.

\begin{lemma}\label{lem:ext_boolean}
    Let $\Lambda$ be an arbitrary $\tau$-tilting finite algebra, and let $\T \in \tors\Lambda$. The following are equivalent.
    \begin{enumerate}[(1)]
        \item\label{lem:ext_1} $\Ext^1_\Lambda(X,X') = 0$ for all distinct $X,X' \in \D(\T)$.
        \item\label{lem:ext_2} The interval $[\popdown_{\tors\Lambda}(\T),\T] \subseteq \tors\Lambda$ is a distributive lattice.
        \item\label{lem:ext_3} The interval $[\popdown_{\tors\Lambda}(\T),\T] \subseteq \tors\Lambda$ is a Boolean lattice.
        \item\label{lem:ext_4} The map $\X \mapsto \Filt(\T \cup  \X)$ is an isomorphism from the Boolean lattice of subsets of $\D(\T)$ to the interval $[\popdown_{\tors\Lambda}(\T),\T]$ of $\tors\Lambda$. 
        \item\label{lem:ext_5} There exist a set of finite-dimensional local $K$-algebras $\{\Lambda_X \mid X \in \D(\T)\}$ and an equivalence of categories $(\popdown_{\tors\Lambda}(\T))^\perp \cap \T \cong \mods(\prod_{X \in \D(\T)} \Lambda_X)$. 
    \end{enumerate}
\end{lemma}

\begin{proof}
    Denote $\U = \popdown_{\tors\Lambda}(\T)$. By \cite[Theorem~6.3]{AP}, \cref{cor:brick_label}, and \cref{prop:IT}, we have that $\U^\perp \cap \T = \Filt(\D(\T))$. Since this is a functorially finite wide subcategory, there exists a finite-dimensional algebra $\Lambda'$ such that $\Filt(\D(\T))$ is equivalent to $\mods\Lambda'$. Moreover, \cite[Theorem~1.4]{AP} says that the interval $[\U,\T]$ is isomorphic to $\tors \Lambda'$. (See also \cite[Theorem~3.12]{jasso} and \cite[Section~4.2]{DIRRT}.) It therefore suffices to assume $\T = \mods\Lambda$, in which case $\D(\T)$ is the set of simple modules and $\popdown_{\tors\Lambda}(\T) = 0$. The result then follows from \cite[Theorem~1.1]{LuoWei}. 
    \end{proof}

\section{The Weak Order and Cambrian Lattices}\label{sec:weak_cambrian} 

\subsection{Coxeter groups} A standard reference for much of the material in this subsection is \cite{BjornerBrenti}. 

Let $(W,S)$ be a finite Coxeter system. This means that $S$ is a finite set and that $W$ is a finite group with a presentation of the form \[\langle S\mid(ss')^{m(s,s')}=e\text{ for all }s,s'\in S\rangle,\] where $e$ is the identity element of $W$ and we have \[m(s,s)=1\quad\text{and}\quad m(s,s')=m(s',s)\in\{2,3,\ldots\}\]for all distinct $s,s'\in S$. (We often refer only to the Coxeter group, tacitly assuming $S$ is part of the data of $W$.)

The elements of $S$ are called the \dfn{simple reflections}. A \dfn{reflection} is an element of the form $wsw^{-1}$ for $w\in W$ and $s\in S$. The \dfn{Coxeter graph} of $W$ is the graph $\Gamma_W$ with vertex set $S$ in which two simple reflections $s$ and $s'$ are connected by an edge whenever $m(s,s')\geq 3$; this edge is labeled with the number $m(s,s')$ if $m(s,s')\geq 4$.
We will assume that $W$ is \dfn{irreducible}, which means that $\Gamma_W$ is connected. We say $W$ is \dfn{simply-laced} if $m(s,s')\leq 3$ for all $s,s'\in S$. There is a well known characterization of finite irreducible Coxeter groups (see \cite[Appendix~A1]{BjornerBrenti}). 

We will use sans serif font when we write words over the alphabet $S$; this allows us to distinguish a word $\s_1\cdots\s_k$ from the element $s_1\cdots s_k\in W$ that it represents.

A \dfn{reduced word} for an element $w\in W$ is a word over the alphabet $S$ that represents $w$ and is as short as possible. The number of letters in a reduced word for $w$ is called the \dfn{length} of $w$ and is denoted $\ell(w)$. A \dfn{right inversion} (resp.\ \dfn{left inversion}) of $w$ is a reflection $t$ such that $\ell(wt)<\ell(w)$ (resp.\ $\ell(tw)<\ell(w)$). 
The \dfn{(right) weak order} is the partial order $\leq$ on $W$ defined so that $u\leq v$ if and only if there exists a reduced word for $v$ that has a reduced word for $u$ as a prefix.
Equivalently, $u\leq v$ if and only if every left inversion of $u$ is a left inversion of $v$. Let $\Weak(W)$ denote the poset $(W,\leq)$. 
It is well known that $\Weak(W)$ is a ranked semidistributive lattice ($\ell$ is a rank function). 
In fact, if $W$ is crystallographic, then $\Weak(W)$ is isomorphic to the lattice of torsion classes of the \emph{preprojective algebra} of type $W$ (see \cref{rem:preprojective}).

A \dfn{cover reflection} of an element $w\in W$ is a reflection $t$ of $W$ such that $\ell(tw)=\ell(w)-1$ and $tw=ws$ for some $s\in S$.
Thus, the cover reflections of $w$ corresponds bijectively to the elements covered by $w$ in $\Weak(W)$. Because $\Weak(W)$ is a semidistributive lattice, each element $w\in W$ has a canonical join representation \[\mathcal D(w)=\{j_{tw,w}\mid t\text{ is a cover reflection of }w\}.\] The following result, which appears as \cite[Theorem~8.1]{ReadingSpeyer2011} (see also \cite[Proposition~4.13]{bicat}), describes the elements of $\mathcal D(w)$ more explicitly. 

\begin{proposition}\label{prop:coxeter_cjr}
Let $t$ be a cover reflection of an element $w\in W$. The shard label $j_{tw,w}$ of the edge $tw\lessdot w$ in $\Weak(w)$ is the unique minimal element of \[\{x\in \Weak(W)\mid x\leq w\text{ and }t\text{ is a left inversion of }x\}.\]
\end{proposition}

The \dfn{long element} of $W$ is the unique element $\wo$ of $W$ that has maximum length. It is known that $\wo^2=e$. If $J\subseteq S$, then the subgroup $W_J$ of $W$ generated by $J$ is called a \dfn{(standard) parabolic subgroup} of $W$. The pair $(W_J,J)$ is a finite Coxeter system; the long element of $W_J$ is denoted by $\wo(J)$. 

A \dfn{descent} of an element $w\in W$ is a simple reflection $s$ such that $\ell(ws)<\ell(w)$; in other words, it is a simple reflection that is also a right inversion of $w$. Let $\Des(w)$ denote the set of descents of $w$. A simple reflection is a descent of $w$ if and only if there is a reduced word for $w$ that ends in $s$. If $u^{-1}\leq v^{-1}$, then $\Des(u)\subseteq\Des(v)$. The pop-stack operator on $\Weak(W)$ has an alternative description \cite{DefantPopCoxeter} given by 
\begin{equation}\label{eq:pop_weak}
\popdown_{\Weak(W)}(x)=x\wo(\Des(x)).
\end{equation}

Let $n = |S|$. For $\alpha=(\alpha_1,\ldots,\alpha_n)$ and $ \beta=(\beta_1,\ldots,\beta_n)$ in $\mathbb{R}^n$, let 
\begin{equation}
    (\alpha,\beta) = \sum_{i,j \in [n]} -2\cos(\pi/m(i,j))\cdot \alpha_i\beta_j\label{eqn:symmetric}.
\end{equation} 
Because $W$ is finite, it is known that $(-,-)$ is a positive definite symmetric bilinear form. For $i\in[n]$, we let $e_i$ denote the $i$-th standard basis vector of $\mathbb R^n$, and we let $e_i^\vee = 2e_i/(e_i,e_i)$. 
By abuse of notation, we use $s_i \in S$ to denote the linear transformation $s_i\colon \mathbb{R}^n \rightarrow \mathbb{R}^n$ given by $s_i(\alpha) = \alpha - (e_i^\vee,\alpha)e_i$. This definition extends to a linear representation of $W$ on $\mathbb{R}^n$. A vector $\alpha \in \mathbb{Z}^n$ is called a \dfn{root} of $W$ if there exists $w \in W$ such that $\alpha = w(e_i)$ for some $i$. A root is \dfn{positive} if its coordinates are all nonnegative. We denote by $\Phi_W$ and $\Phi_W^+$ the sets of roots and positive roots of $W$, respectively.

The action of $W$ induces a permutation on the set of roots; that is, $\Phi_W = \{w(\alpha) \mid \alpha \in \Phi_W\}$ for all $w \in W$. Given a reflection $t$ in $W$, we can write $t=ws_iw^{-1}$ for some $w\in W$ and $s_i\in S$; let $\beta_t$ be the unique positive root in $\{\pm w(e_i)\}$. It is known that $\beta_t$ does not depend on the choices of $w$ and $s_i$. Moreover, the map $t\mapsto\beta_t$ is a bijection from the set of reflections of $W$ to the set $\Phi_W^+$ of positive roots. 

For $w \in W$, let
$$\mathrm{inv}(w) = \{\alpha \in \Phi^+_W \mid w^{-1}(\alpha) \in -\Phi_W^+\}.$$ It is known that $\beta_t\in\inv(w)$ if and only if $t$ is a left inversion of $w$. Thus, for $u,v\in W$, we have $u\leq v$ in the weak order if and only if $\mathrm{inv}(u) \subseteq \mathrm{inv}(v)$.

\subsection{Coxeter elements and Cambrian lattices}\label{subsec:cambrian} 
A \dfn{(standard) Coxeter element} of $W$ is an element $c$ obtained by multiplying the simple reflections in some order (with each appearing once in the product). Thus, a reduced word for $c$ is a word representing $c$ in which each simple reflection appears exactly once. Let us orient each edge $\{s,s'\}$ in $\Gamma_W$ from $s$ to $s'$ if and only if $s$ appears before $s'$ in some (equivalently, every) reduced word for $c$. The result is an acyclic orientation of $\Gamma_W$. This construction establishes a one-to-one correspondence between Coxeter elements of $W$ and acyclic orientations of $\Gamma_W$. We denote the acyclic orientation corresponding to a Coxeter element $c$ by $Q_c$. As is standard in representation theory, we will sometimes call $Q_c$ a \dfn{quiver} instead of a directed graph. We write $S=\{s_1,\ldots,s_n\}$ and denote by $(Q_c)_0 = [n]$ the set of vertices of $Q_c$, identifying each index $i\in [n]$ with $s_i$. Likewise, we denote by $(Q_c)_1$ the set of arrows of $Q_c$. Our convention is to use the notation $a_{i,j} \in (Q_c)_1$ to represent an arrow pointing from the vertex $i$ to the vertex $j$.

Fix a reduced word $\sfc$ for a Coxeter element $c$, and consider the infinite word $\sfc^\infty=\sfc^{(1)}\sfc^{(2)}\cdots$, where each $\sfc^{(k)}$ is a copy of $\sfc$. Following Reading \cite{reading2007clusters}, we define the \dfn{$\sfc$-sorting word} of an element $w\in W$ to be the reduced word $\sort_{\sfc}(w)$ for $w$ that is lexicographically first as a subword of $\sfc^\infty$. Let ${\bf I}_c^{(k)}(w)$ be the set of simple reflections that are taken from $\sfc^{(k)}$ when we form $\sort_{\sfc}(w)$ as the lexicographically first subword of $\sfc^\infty$. Although ${\bf I}_c^{(k)}(w)$ depends on the Coxeter element $c$, it does not depend on the choice of the reduced word $\sfc$. The element $w$ is called \dfn{$c$-sortable} if ${\bf I}_c^{(1)}(w)\supseteq {\bf I}_c^{(2)}(w)\supseteq\cdots$. 

The set of $c$-sortable elements of $W$ forms a sublattice of $\Weak(W)$ called the \dfn{$c$-Cambrian lattice}, which we denote by $\Camb_c$. We will recall in \cref{subsec:path_algebras} that $\Camb_c$ can be realized as the lattice of torsion classes of a finite-dimensional (hereditary) algebra when $W$ is crystallographic.

\begin{example}
A basic yet important example of a Coxeter group is the dihedral group $I_2(m)$, whose Coxeter graph is \!\!\!$\begin{array}{l}
	\begin{tikzpicture}[roundnode/.style={circle, inner sep=0pt, minimum size=3mm}]
		
		\node (A5) at (0.75,0) {$s_2$};
		\node (A4) at (-0.75,0) {$s_1$};
            \node[above] at (0,0) {$m$};
            
		\path (A5) edge [thick] (A4);
	\end{tikzpicture}
\end{array}$\!\!\! (if $m=2$, then there is no edge). Let $c=s_1s_2$. The $c$-sortable elements of $I_2(m)$ are $s_2$ and the $m+1$ elements $e,s_1,s_1s_2,s_1s_2s_1,\ldots,\wo$. 
\end{example}

For each $w\in W$, the set $\Camb_c\cap\{v\in W\mid v\leq w\}$ has a unique maximal element in the weak order; we denote this element by $\pi_\downarrow^c(w)$. The map $\pi_\downarrow^c$ is a surjective lattice homomorphism from $\Weak(W)$ to $\Camb_c$, so $\Camb_c$ is a lattice quotient of $\Weak(W)$ \cite{ReadingCambrian,reading2007clusters}. The fibers of $\pi_\downarrow^c$ are the equivalence classes of an equivalence relation on $W$ known as the \dfn{$c$-Cambrian congruence}. 

According to \cite[Theorem~3.2]{DefantMeeting}, we have 
\begin{equation}\label{eq:pop_Cambrian}
\popdown_{\Camb_c}=\pi_\downarrow^c\circ\popdown_{\Weak(W)}. 
\end{equation}

The following lemma, which is immediate from the definition of a $c$-sortable element (see also \cite[Lemmas~2.4~\&~2.5]{reading2007clusters}), gives a recursive characterization of $c$-sortable elements. 

\begin{lemma}\label{lem:Cambrian_recurrence}
Let $W$ be a finite Coxeter group, and let $\mathsf{c}$ be a reduced word for a Coxeter element $c$ of $W$. Let $s$ be the first letter in $s$. Let $w\in W$. If $\ell(sw)>\ell(w)$, then $w$ is $c$-sortable if and only if it is an $(sc)$-sortable element of the standard parabolic subgroup $W_{S\setminus\{s\}}$. If $\ell(sw)<\ell(w)$, then $w$ is $c$-sortable if and only if $sw$ is $(scs)$-sortable. 
\end{lemma} 

We will also make use of the following characterization of $c$-sortable elements in terms of canonical join representations. 

\begin{lemma}\label{lem:sortable_CJR}
Let $c$ be a Coxeter element of a finite Coxeter group $W$. An element $w\in W$ is $c$-sortable if and only if every element of the canonical join representation of $w$ in $\Weak(W)$ is $c$-sortable. If $w$ is $c$-sortable, then the canonical join representation of $w$ in $\Weak(W)$ is equal to the canonical join representation of $w$ in $\Camb_c$. 
\end{lemma}
\begin{proof}
According to \cite[Proposition~8.2]{ReadingSpeyer2011}, if $w$ is $c$-sortable, then each element of its canonical join representation is $c$-sortable.
On the other hand, if each element of the canonical join representation of $w$ is $c$-sortable, then $w$ is $c$-sortable because $\Camb_c$ is a sublattice of the weak order.
\end{proof}

\subsection{Cambrian lattices as lattices of torsion classes}\label{subsec:path_algebras}
 
Suppose for this subsection that $W$ is crystallographic (i.e., $W$ is not of type $H_3$, $H_4$, or $I_2(m)$ for $m \notin \{3,4,6\}$). Fix a Coxeter element $c \in W$. Associated to the acyclic orientation $Q_c$ of $c$ is a finite-dimensional algebra $KQ_c$ known as the \dfn{tensor algebra} of $Q_c$ (More precisely, the definition of $KQ_c$ depends on a choice of symmetrizable Cartan matrix with Weyl group $W$.) In this subsection, we recall the properties of $KQ_c$ that we will use to interpret \cref{thm:algebraic_image_characterization} in terms of Coxeter combinatorics. Notably, the definition of the algebra $KQ_c$ is not needed in order to state these results and is thus omitted from this paper. Readers interested in more details are referred to \cite{DlabRingelIndecomposable}.

Recall the notation of the projective modules $P(i)$ and simple modules $S(i)$ from \cref{sec:rep_background}. The indexing of these modules can be chosen so that $\undim(S(i)) = e_i = \beta_{s_i}$ for all $i \in [n]$. We fix this indexing for the remainder of this section. For $i \in [n]$, we then denote
\begin{equation}\label{eqn:dim_vects} 
\rho_i = s_n s_{n-1}\cdots s_{i+1}(e_i) .
\end{equation}
Note that $\{\rho_i \mid i \in [n]\} = \inv(c^{-1})$. The roots $\rho_i$ are sometimes called \dfn{projective roots}, a name that is justified by \cref{prop:gabriel_5} in \cref{prop:gabriel} below.

We now recall the following well-known properties of the algebra $KQ_c$. Note that \cref{prop:gabriel_2,prop:gabriel_3,prop:gabriel_4} together constitute 
\emph{Gabriel's Theorem}. These results can be found in the classical references \cite{DlabRingelFinite,DlabRingelIndecomposable}.

\begin{samepage}
\begin{proposition}\label{prop:gabriel}
\hspace{0.5cm}
    \begin{enumerate}[(1)]
        \item\label{prop:gabriel_1} The algebra $KQ_c$ is hereditary.
        \item\label{prop:gabriel_2} The algebra $KQ_c$ is representation-finite.
        \item\label{prop:gabriel_3} The association $M \mapsto \undim M$ is a bijection from the set of (isomorphism classes of) indecomposable modules in $\mods KQ_c$ to the set $\Phi_W^+$ of positive roots.
        \item\label{prop:gabriel_4} Every indecomposable module in $\mods KQ_c$ is a brick.
        \item \label{prop:gabriel_5} For $i \in [n]$, the dimension vector of the indecomposable projective module ${P(i)\in\mods KQ_c}$ is $\undim P(i)=\rho_i$. 
    \end{enumerate}
\end{proposition}
\end{samepage}

\begin{remark}\label{rem:hereditary_proj}
      Let $\T \in \tors KQ_c$ be a torsion class that satisfies $\Ext_{KQ_c}^1(X,X') = 0$ for all ${X, X' \in \D(\T)}$. Then \cref{prop:gabriel_4} in \cref{prop:gabriel} implies that there exists a nonzero projective module $P \in \T$ if and only if there exists an indecomposable projective module $P(i) \in \D(\T)$. 
\end{remark}

We now recall the explicit bijection between the Cambrian lattice $\Camb_c$ and the lattice of torsion classes $\tors(KQ_c)$ established in \cite{IT}. Note that while the result is only proved explicitly for the simply-laced case in \cite{IT}, it is remarked in \cite[Section~4.4]{IT} that the results can be generalized using folding arguments.

\begin{theorem}[{\cite[Theorem~4.3]{IT}}]\label{thm:cambrian_hereditary}
    There is a lattice isomorphism $\varphi_c\colon \Camb_c \rightarrow \tors KQ_c$ characterized by the condition that \[\varphi_c(w) \cap \brick KQ_c = \{M \in \brick KQ_c\mid \undim M \in \mathrm{inv}(w)\}\] for every $c$-sortable element $w \in W$. 
\end{theorem}

\section{The Image of Pop-Stack on a Cambrian Lattice}\label{sec:image}

In this section, we recast \cref{thm:algebraic_image_characterization} as a combinatorial description of the image of $\popdown_{\Camb_c}$ for any finite irreducible Coxeter group $W$ and any Coxeter element $c$. As we discuss in the proof of \cref{thm:combinatorial_image_description}, the results of \cref{sec:pop-stack_description} only yield a proof for $W$ of crystallographic type, while the non-crystallographic types can be verified by computer. Thus, when not directly specified, we assume that $W$ is an arbitrary finite Coxeter group (not necessarily crystallographic). When $W$ is crystallographic, we use the notation $KQ_c$ as in \cref{subsec:path_algebras}. In any case, we denote by $c$ a fixed Coxeter element of $W$.

We first give an interpretation of the Ext-orthogonality condition in \cref{thm:algebraic_image_characterization}\ref{thm:algebraic_1}. For each $w \in \Camb_c$, \cref{lem:ext_boolean,thm:cambrian_hereditary} say that the torsion class $\varphi_c(w)$ satisfies the Ext-orthogonality condition if and only if $[\popdown_{\Camb_c}(w),w]$ is Boolean. \cref{thm:distributive_intervals} below provides several equivalent conditions characterizing when this holds. Before we state and prove this result, we require some lemmas. 

\begin{lemma}\label{lem:pop_distributive_boolean}
Let $L$ be a finite lattice, and let $x\in L$. If the interval $[\popdown_L(x),x]$ is distributive, then it is Boolean. 
\end{lemma}

\begin{proof}
Suppose $[\popdown_L(x),x]$ is distributive, and let $P$ be a finite poset such that $[\popdown_L(x),x]$ is isomorphic to $J(P)$. The bottom and top elements of $J(P)$ are $\emptyset$ and $P$, respectively, so ${\emptyset=\popdown_{J(P)}(P)=\bigcap_{y\in\max(P)}(P\setminus\{y\})}$, where $\max(P)$ is the set of maximal elements of $P$. This implies that $P=\max(P)$, so $P$ is an antichain. Hence, $J(P)$ is Boolean. 
\end{proof}

\begin{lemma}\label{lem:2^k}
If $L$ is a finite semidistributive lattice with $k$ coatoms, then $|L|\geq 2^k$. 
\end{lemma}
\begin{proof}
The map $v\mapsto\mathcal D(v)$ is a bijection from $L$ to the canonical join complex of $L$. Since $|\mathcal D(\hat 1)|=k$, the canonical join complex of $L$ must have at least $2^k$ faces. 
\end{proof}

The next result follows from \cite[Lemma~3.8]{ReadingLatticeCongruences}; we include a short self-contained proof.

\begin{lemma}[{\cite[Lemma~3.8]{ReadingLatticeCongruences}}]\label{lem:cover_reflections}
If $y_1\lessdot z_1$ and $y_2\lessdot z_2$ are two edges of $\Weak(W)$ with the same shard label, then $z_1y_1^{-1}=z_2y_2^{-1}$. 
\end{lemma}
\begin{proof}
Let $t_1=z_1y_1^{-1}$ and $t_2=z_2y_2^{-1}$. \cref{prop:coxeter_cjr} tells us that the join-irreducible element $j_{y_1,z_1}$ has $t_1$ as a left inversion while $(j_{y_1,z_1})_*$ does not have $t_1$ as a left inversion. If follows that $t_1$ is the unique cover reflection of $j_{y_1,z_1}$. Similarly, $t_2$ is the unique cover reflation of $j_{y_2,z_2}$. Since $j_{y_1,z_1}=j_{y_2,z_2}$, we have $t_1=t_2$. 
\end{proof}

The following lemma is a special case of \cite[Proposition~5.7]{ReadingShard}. 

\begin{lemma}[{\cite[Proposition~5.7]{ReadingShard}}]\label{lem:shard_intersection}
If $w$ and $z$ are elements of $\Weak(W)$ such that the canonical join representation of $w$ contains that of $z$, then the set of shard labels of edges in the interval $[\popdown_{\Weak(W)}(z),z]$ is contained in the set of shard labels of the edges in the interval $[\popdown_{\Weak(W)}(w),w]$. 
\end{lemma}

\begin{lemma}\label{lem:commute_implies_sortable}
Suppose $w$ is a $c$-sortable element of $W$ whose descents commute. The interval $[\popdown_{\Weak(W)}(w),w]$ of $\Weak(W)$ is Boolean, and $\popdown_{\Camb_c}(w)$ is $c$-sortable. 
\end{lemma}
\begin{proof}
The fact that the interval $[\popdown_{\Weak(W)}(w),w]$ is Boolean is immediate from the observation that it consists of the elements of the form $wv^{-1}$ such that $v$ is the product of some subset of $\Des(w)$. Let us now show that $\popdown_{\Weak(W)}(w)$ is $c$-sortable. Our proof will proceed by induction on $|S|$ and $\ell(w)$ (the base cases are trivial). To ease notation, let us write $w'=\popdown_{\Weak(W)}(w)$. 

Let $s$ be the first letter in some reduced word for $c$. Suppose first that $\ell(sw)>\ell(w)$. According to \cref{lem:Cambrian_recurrence}, $w$ is an $(sc)$-sortable element of $W_{S\setminus\{s\}}$. By induction, $\popdown_{\Weak(W_{S\setminus\{s\}})}(w)$ is an $(sc)$-sortable element of $W_{S\setminus\{s\}}$. Invoking \cref{lem:Cambrian_recurrence} again, we find that $\popdown_{\Weak(W_{S\setminus\{s\}})}(w)$ is $c$-sortable. The desired result now follows from the fact that $w'=\popdown_{\Weak(W_{S\setminus\{s\}})}(w)$. 

Now assume that $\ell(sw)<\ell(w)$. Then $(sw)^{-1}\leq w^{-1}$, so every left inversion of $(sw)^{-1}$ is a left inversion of $w^{-1}$. Hence, every right inversion of $sw$ is a right inversion of $w$. It follows that $\Des(sw)\subseteq\Des(w)$. We consider two cases. 

\medskip 

\noindent {\bf Case 1.} Suppose $s\leq w'$ (equivalently, $\ell(sw')=\ell(w')-1$). Let $v=sw'$. We have $w=sv\wo(\Des(w))$, and $\ell(w)=\ell(v)+\ell(\wo(\Des(w)))+1$. Hence, $sw=v\wo(\Des(w))$, and $\ell(sw)=\ell(v)+\ell(\wo(\Des(w)))$. This shows that $sw$ has a reduced word that contains a reduced word for $\wo(\Des(w))$ as a suffix. Hence, $\Des(w)=\Des(\wo(\Des(w)))\subseteq\Des(sw)$. 

This shows that $\Des(w)=\Des(sw)$, so $\popdown_{\Weak(W)}(sw)=sw\wo(\Des(w))=v$. Because $w$ is $c$-sortable, we know by \cref{lem:Cambrian_recurrence} that $sw$ is $(scs)$-sortable. Furthermore, the descents of $sw$ commute. Because $\ell(sw)<\ell(w)$, we can use induction (replacing $c$ by $scs$) to find that $\popdown_{\Weak(W)}(sw)=v$ is $(scs)$-sortable. But since $v=sw'$ and $\ell(sw')<\ell(w')$, we can invoke \cref{lem:Cambrian_recurrence} once again to find that $w'$ is $c$-sortable. 

\medskip 

\noindent {\bf Case 2.} Suppose $s\not\leq w'$ (equivalently, $\ell(sw')=\ell(w')+1$). Since $w'$ is the meet of the elements covered by $w$ in $\Weak(W)$, there must exist an element $x$ that is covered by $w$ such that $s\not\leq x$. The unique left inversion of $w$ that is not a left inversion of $x$ is $wx^{-1}$. Since $s$ is a left inversion of $w$ but not a left inversion $x$, we must have $wx^{-1}=s$. Hence, $x=sw$. 

We have seen that $\Des(sw)\subseteq\Des(w)$, and we have just shown that $sw\lessdot w$. There exists $s'\in S$ such that $sw=ws'$. Because $[w',w]$ is Boolean, we must have $\Des(sw)=\Des(w)\setminus\{s'\}$, and it follows from \cref{lem:Cor710,exam:boolean} that $\mathcal D(sw)=\mathcal D(w)\setminus\{s\}$. Hence, \[\popdown_{\Weak(W)}(sw)=sw\prod_{s''\in\Des(sw)}s''=ws'\prod_{s''\in\Des(w)\setminus\{s'\}}s''=w\prod_{s''\in\Des(w)}s''=w'.\]  Since $w$ is $c$-sortable, it follows from \cref{lem:sortable_CJR} that $sw$ is also $c$-sortable. The descents of $sw$ commute, so we can use induction to find that $\popdown_{\Weak(W)}(sw)=w'$ is $c$-sortable. 
\end{proof}

\begin{theorem}\label{thm:distributive_intervals}
    Let $c$ be a Coxeter element of a finite Coxeter group $W$. For every $c$-sortable element $w$, the following are equivalent. 
    \begin{enumerate}[(1)]
        \item\label{thm:distributive_1} The descents of $w$ commute.
        \item\label{thm:distributive_2} The interval $[\popdown_{\Weak(W)}(w),w]$ of $\Weak(W)$ is a distributive lattice.
        \item\label{thm:distributive_3} The interval $[\popdown_{\Weak(W)}(w),w]$ of $\Weak(W)$ is a Boolean lattice.
        \item\label{thm:distributive_4} The interval $[\popdown_{\Camb_c}(w),w]$ of $\Camb_c$ is a distributive lattice.
        \item\label{thm:distributive_5} The interval $[\popdown_{\Camb_c}(w),w]$ of $\Camb_c$ is a Boolean lattice.
        \item\label{thm:distributive_6} The interval $[\popdown_{\Camb_c}(w),w]$ of $\Camb_c$ equals the interval $[\popdown_{\Weak(W)}(w),w]$ of $\Weak(W)$. 
    \end{enumerate}
\end{theorem}

\begin{proof}
It is immediate from \cref{lem:pop_distributive_boolean} that \ref{thm:distributive_2} and \ref{thm:distributive_3} are equivalent and that \ref{thm:distributive_4} and \ref{thm:distributive_5} are equivalent. \cref{lem:commute_implies_sortable} tells us that \ref{thm:distributive_1} implies \ref{thm:distributive_3}. To see that \ref{thm:distributive_3} implies \ref{thm:distributive_1}, suppose $[\popdown_{\Weak(W)}(w),w]$ is Boolean, and suppose $s$ and $s'$ are distinct descents of $W$. Let $w'=ws\wedge ws'$. The interval $[w',w]$ in $\Weak(W)$ is contained in $[\popdown_{\Weak(W)}(w),w]$, so $[w',w]$ is Boolean. We also know that $[w',w]$ is isomorphic to the weak order on the dihedral group $I_2(m(s,s'))$. This implies that $m(s,s')=2$, so $s$ and $s'$ commute. 

According to \cref{lem:sortable_CJR}, the canonical join representations of $w$ in $\Weak(W)$ and $\Camb_c$ are equal; thus, we can unambiguously write $\mathcal D(w)$ for this canonical join representation. 

Let us show that \ref{thm:distributive_1} implies \ref{thm:distributive_6}. Suppose the descents of $w$ commute. Let ${k=\mathcal D(w)=|\Des(w)|}$. According to \cref{lem:commute_implies_sortable}, the interval $[\popdown_{\Weak(W)}(w),w]$ in $\Weak(W)$ is Boolean, and the element $\popdown_{\Weak(W)}(w)$ is $c$-sortable. This implies that $[\popdown_{\Camb_c}(w),w]$ is contained in $[\popdown_{\Weak(W)}(w),w]$ and that $[\popdown_{\Weak(W)}(w),w]$ has size $2^k$. But $[\popdown_{\Camb_c}(w),w]$ is a semidistributive lattice with $k$ coatoms, so \cref{lem:2^k} tells us that it has size at least $2^k$. This implies that $[\popdown_{\Camb_c}(w),w]$ equals $[\popdown_{\Weak(W)}(w),w]$. 

In \cite{ThomasTrim}, Thomas introduced the notion of a \emph{trim} lattice; he proved that intervals in trim lattices are trim and that a lattice is distributive if and only if it is ranked and trim. Ingalls and Thomas \cite{IT} proved that Cambrian lattices are trim. These facts allow us to prove that \ref{thm:distributive_6} implies \ref{thm:distributive_2}. Indeed, $[\popdown_{\Weak(W)}(w),w]$ is an interval in the weak order on $W$, so it is ranked. If $[\popdown_{\Weak(W)}(w),w]$ is equal to an interval of $\Camb_c$, then it is trim, so it must be distributive. 

We now know that \cref{thm:distributive_1,thm:distributive_2,thm:distributive_3,thm:distributive_6} are equivalent and that \cref{thm:distributive_4,thm:distributive_5} are equivalent. It is also obvious that \ref{thm:distributive_3} and \ref{thm:distributive_6} together imply \ref{thm:distributive_5}. Therefore, to complete the proof, it suffices to show that \ref{thm:distributive_5} implies \ref{thm:distributive_1}. 

Assume $[\popdown_{\Camb_c}(w),w]$ is Boolean. Suppose by way of contradiction that $s$ and $s'$ are descents of $w$ that do not commute. Because $\{j_{ws,w},j_{ws',w}\}$ is contained in $\mathcal D(w)$, it follows from \cref{lem:sortable_CJR} that there exists a $c$-sortable element $z$ of $\Weak(W)$ whose canonical join representation is $\{j_{ws,w},j_{ws',w}\}$. Let $y$ and $y'$ be the two elements covered by $z$, and assume without loss of generality that $j_{y,z}=j_{ws,w}$ and $j_{y',z}=j_{ws',w}$. Deleting the last letter in the $c$-sorting word for $z$ yields the $c$-sorting word for a $c$-sortable element covered by $z$; this element must be $y$ or $y'$. Assume without loss of generality that $y$ is $c$-sortable. By \cref{lem:cover_reflections}, we have $zy^{-1}=wsw^{-1}$ and $z(y')^{-1}=ws'w^{-1}$. It follows that $zy^{-1}$ and $z(y')^{-1}$ do not commute. Hence, the interval $[\popdown_{\Weak(W)}(z),z]=[y\wedge y',z]$ is isomorphic to the weak order on the dihedral group $I_2(m)$ for some $m\geq 3$. Let $x$ be the unique element of $[\popdown_{\Weak(W)}(z),z]$ covered by $y$. Since $y$ is $c$-sortable, \cref{lem:sortable_CJR} tells us that the shard label $j_{x,y}$ is $c$-sortable. According to \cref{lem:shard_intersection}, $j_{x,y}$ is the shard label of an edge in $[\popdown_{\Weak(W)}(w),w]$. Since $j_{x,y}$ is $c$-sortable, it is also a shard label of an edge in $[\popdown_{\Camb_c}(w),w]$. Because $[\popdown_{\Camb_c}(w),w]$ is Boolean, it follows from \cref{lem:Cor710} that $j_{x,y}\in\mathcal D(w)$. Because $\mathcal D(w)$ is an independent set of the Galois graph of $\Weak(W)$ and $j_{y,z}=j_{ws,w}\in\mathcal D(w)$, there cannot be an arrow from $j_{y,z}$ to $j_{x,y}$ in the Galois graph of $\Weak(W)$. 
However, since $[y\wedge y',z]$ is isomorphic to the weak order on $I_2(m)$, it follows from \cref{lem:Cor710,exam:dihedral} that there \emph{is} an arrow from $j_{y,z}$ to $j_{x,y}$ in the Galois graph of $\Weak(W)$. This is our desired contradiction. 
\end{proof}

\begin{remark}\label{rem:preprojective}
   It is also possible to prove many of the equivalences in \cref{thm:distributive_intervals} using arguments from representation theory, or conversely to use \cref{thm:distributive_intervals} to prove special cases of \cref{lem:ext_boolean}. Since these arguments highlight further connections between Coxeter groups and lattices of torsion classes, we give a short overview of them here.
    
    In addition to the hereditary algebras $KQ_c$, one can also associate a (non-hereditary) \emph{preprojective algebra} $\Pi(W)$ to each finite crystallographic Coxeter group $W$. (For background outside the simply-laced case, see \cite[Section~4]{kulshammer} and the references therein.) It 
    is shown in \cite{mizuno} (simply-laced case) and \cite[Section~7]{AHIKM} (general case) that there is an isomorphism $\varphi: \Weak(W) \rightarrow \tors(\Pi(W))$ induced by the symmetric bilinear form $(-,-)$.
    
    For any Coxeter element $c$, the hereditary algebra $KQ_c$ is a quotient of the preprojective algebra. Thus the algebra quotient $\Pi(W) \rightarrow KQ_c$ induces a lattice quotient \[\Weak(W) \cong \tors(\Pi(W)) \rightarrow \tors(KQ_c) \cong \Camb_c,\] and it is shown (in the simply-laced case) in \cite{MizunoThomas} that this lattice quotient coincides with the map $\pidown^c$. The algebra quotient also induces a (fully faithful) inclusion $\brick KQ_c \subseteq \brick \Pi(W)$, which is known to preserve Ext-orthogonality. Given this, \cref{thm:distributive_intervals} becomes equivalent to the specialization of \cref{lem:ext_1,lem:ext_2,lem:ext_3,lem:ext_4} in \cref{lem:ext_boolean} to the cases $\Lambda = KQ_c$ and $\Lambda = \Pi(W)$.
\end{remark}

We now turn toward interpreting \cref{thm:algebraic_2} in \cref{thm:algebraic_image_characterization}. Recall that $S$ is precisely the set of atoms of $\Camb_c$. In the crystallographic case, \cref{thm:cambrian_hereditary} says that these atoms can be indexed (by $[n]$) so that $\varphi_c(s_i) = \Filt(S_i)$ for all $i$. In the non-crystallographic case, we fix an arbitrary indexing of $S$ by $[n]$. In either case, for $s_i \in S$, we let \[\Theta_i = \{x \in \Camb_c \mid s_i \leq x \text{ and } s_j \not\leq x \text{ for all }s_j \in S \setminus \{s_i\}\}\] and let $p_i=\bigvee\Theta_i$. (One can show that $p_i$ is actually the unique maximal element of $\Theta_i$.) If $W$ is crystallographic, then \cref{prop:proj,thm:cambrian_hereditary} imply that $\varphi_c(p_i) = \Gen (P(i))$ for each $i$. In particular, \cref{rem:hereditary_proj} implies that, for $w \in \Camb_c$, the torsion class $\varphi_c(w)$ contains a projective module if and only if $p_i \leq w$ for some $i$.

We now state the main result of this section, which provides purely lattice-theoretic and Coxeter-theoretic descriptions of the image of the pop-stack operator on any Cambrian lattice.

\begin{theorem}\label{thm:combinatorial_image_description}
    Let $W$ be a finite irreducible Coxeter group, and let $c \in W$ be a Coxeter element. For $w\in\Camb_c$, the following are equivalent. 
\begin{enumerate}[(1)]
\item $w$ is in the image of $\popdown_{\Camb_c}$. 
\item The descents of $w$ all commute, and $w$ has no left inversions in common with~$c^{-1}$. 
\item The interval $[\popdown_{\Camb_c}(w),w]$ is Boolean, and $p_i \not\leq w$ for all $i\in[n]$.
\end{enumerate} 
\end{theorem}

\begin{proof}
    Suppose first that $W$ is crystallographic. By \cref{thm:cambrian_hereditary}, we have that $w$ is in the image of $\popdown_{\Camb_c}$ if and only if $\varphi_c(w)$ is in the image of $\popdown_{\tors KQ_c}$. Furthermore, it was established above that $\varphi_c(w)$ contains a nonzero projective module if and only if there exists $i \in [n]$ such that $p_i \leq w$, and we know that $p_i\leq w$ if and only if the root $\rho_i$ (defined in \cref{eqn:dim_vects}) is in $\inv(w)$. Finally, we have that the descents of $w$ all commute with one another if and only if $\Ext^1_{KQ_c}(X,X') = 0$ for all distinct $X,X' \in \D(\varphi_c(w))$ by \cref{lem:ext_boolean,thm:distributive_intervals}. Therefore, the desired result follows from \cref{thm:algebraic_image_characterization,thm:distributive_intervals} and the fact that $\{\rho_i\mid i\in[n]\}=\inv(c^{-1})$. 

    It remains to consider the non-crystallographic cases. We verified the theorem when $W$ is of type $H_3$ or $H_4$ using Sage \cite{sagemath}. When $W$ is of type $I_2(m)$, it is straightforward to check the desired result by hand.  
\end{proof}




\begin{remark}\label{rem:ext_orthogonal}
The naive analogue of \cref{thm:combinatorial_image_description} fails for the weak order on $A_n$. For example, the pop-stack operator on $\Weak(A_4)$ sends the permutation $52341$ to the permutation $25314$, but the descents of $25314$ ($s_2$ and $s_3$) do not commute with each other. In the language of representation theory, this means that there exist bricks $X, X' \in \D(\varphi(25314))$ such that ${\Ext^1_{\Pi(A_4)}(X,X') \neq 0}$ (the notation $\varphi$ is from \cref{rem:preprojective}). Thus, the naive analogue of \cref{thm:algebraic_image_characterization} likewise fails for the preprojective algebra $\Pi(A_4)$. 

It is also worth mentioning that the image of $\popdown_{\Camb_c}$ is not necessarily contained in the image of $\popdown_{\Weak(W)}$. For example, if $W=A_9$ and $c=s_1s_3s_5s_7s_9s_2s_4s_6s_8$, then one can use \cite[Theorem~1]{ABBHL} and \cref{cor:choi_sun} below to show that the permutation $1,2,8,10,6,9,4,5,3,7$ is in the image of $\popdown_{\Camb_c}$ but not the image of $\popdown_{\Weak(W)}$. 
\end{remark}

The following surprising consequence of \cref{thm:distributive_intervals,} tells us that when we compute a forward orbit of $\popdown_{\Camb_c}$, all applications of the pop-stack operator after the first can be computed in the weak order. We believe this could have interesting further implications concerning the dynamical properties of $\popdown_{\Camb_c}$. 

\begin{corollary}\label{cor:dynamical_weak_Camb}
Let $c$ be a Coxeter element of a finite irreducible crystallographic Coxeter group $W$. If $w\in \Camb_c$, then \[(\popdown_{\Weak(W)})^t(\popdown_{\Camb_c}(w))=(\popdown_{\Camb_c})^{t+1}(w)\] for all $t\geq 0$.
\end{corollary}

\begin{proof}
The result is trivial if $t=0$, so we may assume $t\geq 1$ and assume inductively that \[(\popdown_{\Weak(W)})^{t-1}(\popdown_{\Camb_c}(w))=(\popdown_{\Camb_c})^{t}(w).\] Let $u=(\popdown_{\Camb_c})^{t}(w)$. Since $u$ is in the image of $\popdown_{\Camb_c}$, we know by \cref{thm:combinatorial_image_description} that all of the descent of $u$ commute with each other. Therefore, it follows from \cref{thm:distributive_intervals} that $\popdown_{\Weak(W)}(u)=\popdown_{\Camb_c}(u)$; this is equivalent to the desired result.  
\end{proof}

\section{Cambrian Lattices in Type A}\label{sec:Cambrian_A}
In this section, we restrict to Cambrian lattices of type A. We first recall the permutation model for the Coxeter group $A_n$ and the definition of bipartite Coxeter elements. In \cref{subsec:arc,subsec:Motzkin}, we recall two classes of combinatorial objects: \emph{arc diagrams} and \emph{Motzkin paths}. We use these in \cref{subsec:bijection} to resolve a conjecture (stated in \cref{eq:conjectured_gf}) of Defant and Williams from \cite{DefantWilliamsSemidistrim}. This yields an explicit formula for the generating function that counts the images of the pop-stack operators on bipartite Cambrian lattices of type A.

\subsection{Permutations and Coxeter elements}\label{subsec:perms}

The Coxeter group $A_n$ is the same as the symmetric group whose elements are permutations of the set $[n+1]=\{1,\ldots,n+1\}$. We will frequently represent a permutation $w\in A_n$ in \dfn{one-line notation} as the word $w(1)\cdots w(n+1)$. The simple reflections of $A_n$ are $s_1,\ldots,s_n$, where $s_i$ is the transposition that swaps $i$ and $i+1$. The Coxeter graph $\Gamma_{A_n}$ is a path that contains an (unlabeled) edge $\{s_i,s_{i+1}\}$ for each $i\in[n]$. A simple reflection $s_i$ is a descent of $w$ if and only if $w(i)>w(i+1)$; when this is the case, we will also refer to the index $i$ as a descent of $w$. A left inversion of $w$ is a transposition $(i\,\, j)$ such that $i<j$ and $w^{-1}(i)>w^{-1}(i)$. 

Let \[c_1=\prod_{\substack{i\in[n]\\ i\text{ odd}}}s_i\quad \text{and}\quad c_2=\prod_{\substack{i\in[n]\\ i\text{ even}}}s_i.\] The \dfn{bipartite Coxeter elements} of $A_n$ are $c_{(n)}^{\bip}=c_1c_2$ and $c_{(n)}^{\bip\bip}=c_2c_1$. These will be the focus of much (but not all) of this section. The following result shows that it generally suffices to consider only $c^\bip_{(n)}$.

\begin{proposition}
The canonical join complexes of $\Camb_{c_{(n)}^\bip}$ and $\Camb_{c_{(n)}^{\bip\bip}}$ are isomorphic. 
\end{proposition}

\begin{proof}
The map $x\mapsto \wo x\wo$ is an automorphism of $\Weak(A_n)$; when $n$ is even, it restricts to an isomorphism from $\Camb_{c_{(n)}^\bip}$ to $\Camb_{c_{(n)}^{\bip\bip}}$. Thus, the desired result is immediate when $n$ is even. The map $x\mapsto \wo x$ is an antiautomorphism of $\Weak(A_n)$; when $n$ is odd, it restricts to an isomorphism from $\Camb_{c_{(n)}^\bip}$ to $\Camb_{c_{(n)}^{\bip\bip}}$. The canonical join complex of the dual of $\Camb_{c_{(n)}^{\bip\bip}}$ is equal to the canonical meet complex of $\Camb_{c_{(n)}^{\bip\bip}}$. Therefore, when $n$ is odd, the desired result follows from the fact that the canonical join complex and the canonical meet complex of a semidistributive lattice are isomorphic. 
\end{proof}

Recall that if $L$ is a finite semidistributive lattice, then the generating function \[\PP_L(q)=\sum_{v\in \popdown_L(L)}q^{|\U(v)|}=\sum_{v\in \popup_L(L)}q^{|\D(v)|}\] defined in \cref{eq:P} counts the facets of the canonical join complex of $L$ according to their sizes.
In \cite[Conjecture~11.2]{DefantWilliamsSemidistrim}, Defant and Williams conjectured\footnote{This conjecture was stated in \cite{DefantWilliamsSemidistrim} in terms of an explicit formula for $\PP_{\Camb_{c^{\bip}_{(n)}}}(q)$ for each particular $n$, but we prefer to write it here in terms of the generating function that encompasses all $n\geq 1$.} that 
\begin{equation}\label{eq:conjectured_gf}
\sum_{n\geq 1}\PP_{\Camb_{c^{\bip}_{(n)}}}(q)z^n=\frac{1}{qz}\left(\frac{2}{1-qz(1-2z)+\sqrt{1+q^2z^2-2qz(1+2z)}}-1\right)-1.
\end{equation} 
We will prove this conjecture in \cref{subsec:bijection}.

\subsection{The image of pop-stack in $A_n$}\label{subsec:image_A}

For this subsection, we let $c$ denote an arbitrary (not necessarily bipartite) Coxeter element of $A_n$. Define $\vv_c\colon\{2,\ldots,n\}\to\{\LL,\RR\}$ by 
\begin{equation}\label{eq:varphi}
\vv_c(i)=\begin{cases} \LL & \mbox{if }s_{i}\mbox{ precedes }s_{i-1}\mbox{ in every reduced word for }c; \\ \RR & \mbox{if }s_{i-1}\mbox{ precedes }s_{i}\mbox{ in every reduced word for }c.\end{cases}
\end{equation} 
The map $c\mapsto\vv_c$ is a bijection from the set of Coxeter elements of $A_n$ to the set of functions from $\{2,\ldots,n\}$ to $\{\LL,\RR\}$.

\begin{theorem}[{\cite[Example~4.9]{ReadingNoncrossing}}]\label{thm:reading_c-sortable_combinatorial}
Let $c$ be a Coxeter element of $A_n$. An element $w\in A_n$ is $c$-sortable if and only if the following hold for all $i,j\in[n+1]$ such that $w(j+1)<w(i)<w(j)$: \begin{itemize}
\item If $\vv_c(i)=\LL$, then $j<i$. 
\item If $\vv_c(i)=\RR$, then $i<j$. 
\end{itemize}
\end{theorem}


For $w\in A_n$, we recall that a simple reflection $s_i$ is in $\Des(w)$ if and only if $w(i) > w(i+1)$. We say that $i$ is a \dfn{double descent} of $w$ if $i \neq n$ and $w(i) > w(i+1) > w(i+2)$. Note that the descents of $w$ all commute with one another if and only if $w$ does not have any double descents.

\begin{example}
Let $c=s_1s_2\cdots s_n\in A_n$. Then $\nu_c(i) = \RR$ for all $i \in \{2,\ldots,n\}$. A permutation $w\in A_n$ is $c$-sortable if and only if it avoids the pattern $312$ (i.e., there do not exist indices $i_1<i_2<i_3$ such that $w(i_2)<w(i_3)<w(i_1)$). The $c$-Cambrian lattice is often called the $(n+1)$-st \dfn{Tamari lattice}. 

The one-line notation of $c^{-1}$ is $(n+1)123\cdots n$. Hence, the left inversions of $c^{-1}$ are the transpositions of the form $(i\,\, n+1)$ for $i\in[n]$. It follows from \cref{thm:combinatorial_image_description} that a $c$-sortable permutation $w$ is in the image of $\popdown_{\Camb_c}$ if and only if $w$ has no double descents and the one-line notation of $w$ ends with $n+1$. This recovers one of the main results of \cite{Hong2022}. 
\end{example}

\begin{example}\label{exam:ChoiSun}
        Let $c$ be the bipartite Coxeter element $c^\bip_{(7)}$ of $A_7$. We have
        $$\nu_{c}(i) = \begin{cases} \LL & \text{if }i \equiv 1 \pmod{2};\\\RR & \text{if }i\equiv 0\pmod{2}\end{cases}$$
        for all $i \in \{2,\ldots,7\}$. 

        One can readily compute that the left inversions of $c^{-1}$ are 
        \[s_2=(2\,\,3),\quad s_4=(4\,\,5),\quad s_6=(6\,\,7),\] 
        \[s_2s_1s_2=(1\,\,3),\quad s_2s_4s_3s_4s_2=(2\,\,5),\quad s_4s_6s_5s_6s_4=(4\,\,7),\quad s_6s_7s_6=(6\,\,8).\]
        According to \cref{thm:combinatorial_image_description}, a $c$-sortable permutation $w$ is in the image of $\popdown_{\Camb_{c}}$ if and it has no double descents and does not have any of the seven transpositions in this list as left inversion. 
\end{example}

A straightforward extension of \cref{exam:ChoiSun} yields the following corollary to \cref{thm:combinatorial_image_description}, which settles \cite[Conjecture~4.2]{ChoiSun}. 

\begin{corollary}\label{cor:choi_sun}
    An element $w\in\Camb_{c^\bip_{(n)}}$ is in the image of $\popdown_{\Camb_{c^\bip_{(n)}}}$ if and only if all of the following hold:
    \begin{itemize}
        \item $w$ does not have any double descents;
        \item $w^{-1}(2k) < w^{-1}(2k+1)$ for all integers $k$ with $3 \leq 2k+1 \leq n+1$;
        \item $w^{-1}(2k) < w^{-1}(2k + 3)$ for all integers $k$ with $5 \leq 2k + 3 \leq n+1$;
        \item $w^{-1}(n-1) < w^{-1}(n+1)$ if $n \equiv 1 \pmod{2}$;
        \item $w^{-1}(1) < w^{-1}(3)$ if $n \neq 1$.
    \end{itemize}
\end{corollary}

\subsection{Arc diagrams}\label{subsec:arc}

Arrange $n+1$ points along a horizontal line, and identify them with the numbers $1,\ldots,n+1$ from left to right. An \dfn{arc} is a curve that moves monotonically rightward from a point $i$ to another point $j$ (for some $i<j$), passing above or below each of the points $i+1,\ldots,j-1$. Two arcs are considered to be the same if they have the same endpoints and they pass above the same collection of numbered points. A \dfn{noncrossing arc diagram} (of type~$A_n$) is a collection of arcs that can be drawn so that no two arcs have the same left endpoint or have the same right endpoint or cross in their interiors. We write $|\delta|$ for the number of arcs in a noncrossing arc diagram $\delta$. Let $\AD_n$ be the set of noncrossing arc diagrams of type~$A_n$. 

Given a permutation $w\in A_n$, form a noncrossing arc diagram $\Delta(w)\in\AD_n$ as follows. For each descent $i$ of $w$, draw an arc from $w(i+1)$ to $w(i)$ such that for each integer $k$ satisfying ${w(i+1)<k<w(i)}$, the arc passes above (resp.\ below) the point $k$ if $w^{-1}(k)>i+1$ (resp.\ ${w^{-1}(k)<i}$). This defines a map $\Delta\colon A_n\to\AD_n$, and it is straightforward to check that $\Delta$ is a bijection. (See \cite{ReadingNoncrossing}.)

Given a Coxeter element $c$ of $A_n$, say an arc $\mathfrak{a}$ with left endpoint $i$ and right endpoint $j$ is \dfn{$c$-sortable} if for every $k\in\{i+1,\ldots,j-1\}$, $\mathfrak{a}$ passes above (resp.\ below) $k$ if $\vv_c(k)=\LL$ (resp.\ $\vv_c(k)=\RR$). Note that for all $1\leq i<j\leq n+1$, there is a unique $c$-sortable arc from $i$ to $j$. Let $\AD(c)=\Delta(\Camb_c)$ be the set of noncrossing arc diagrams of $c$-sortable elements of $A_n$. It is immediate from \Cref{thm:reading_c-sortable_combinatorial} that a noncrossing arc diagram is in $\AD(c)$ if and only if all of its arcs are $c$-sortable. Hence, $\AD(c)$ is a simplicial complex whose vertices are the $c$-sortable arcs.

\begin{example}
Suppose $n=8$, and let $c=s_1s_2s_4s_8s_3s_5s_7s_6$. Then $\vv_c^{-1}(\LL)=\{4,7,8\}$, and $\vv_c^{-1}(\RR)=\{2,3,5,6\}$. Let $w$ be the permutation $325148679$. The noncrossing arc diagram $\Delta(w)$ is depicted in \Cref{fig:Arc1}. In this figure, each of the points $i\in[9]$ is represented by a circle filled with the number $i$; for $2\leq i\leq 8$, a blue semicircle appears on the top (resp.\ bottom) of the circle if $\vv_c(i)=\LL$ (resp.\ if $\vv_c(i)=\RR$). All of the arcs in $\Delta(w)$ are $c$-sortable, so $w$ is a $c$-sortable permutation. 
\end{example}

\begin{figure}[ht]
  \begin{center}\includegraphics[height=1.02cm]{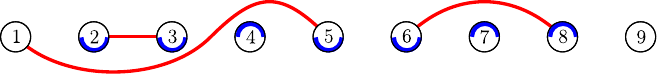}
  \end{center}
  \caption{The noncrossing arc diagram $\Delta(325148679)$. For each $2\leq i\leq 8$, a blue semicircle appears on the top (resp.\ bottom) of the circle containing $i$ if $\vv_c(i)=\LL$ (resp.\ if $\vv_c(i)=\RR$), where $c=s_1s_2s_4s_8s_3s_5s_7s_6$.}\label{fig:Arc1}
\end{figure}

\begin{remark}
    Noncrossing arc diagrams and other similar objects have been used to model the bricks over hereditary and preprojective algebras of type A. In particular, ($c$-sortable) arcs are in bijection with bricks, and ($c$-sortable) noncrossing arc diagrams are in bijection with semibricks. One can then use noncrossing arc diagrams to compute information about morphisms and extensions between bricks, and thus also to model 2-term simple-minded collections. See, e.g., \cite{BCZ,BaH_preproj,enomoto_bruhat,GIMO,hanson,HY,mizuno2} and references therein. 
\end{remark}

Cambrian lattices are semidistributive, so we can consider the canonical join complex and the canonical meet complex of $\Camb_c$ (and we know these simplicial complexes are isomorphic by \cite[Corollary~5]{BarnardCJC}). An element $v\in \Camb_c$ is join-irreducible if and only if it has exactly one descent, and this occurs if and only if $\Delta(v)$ contains a single arc. Therefore, $\Delta$ establishes a one-to-one correspondence between the join-irreducible elements of $\Camb_c$ and the $c$-sortable arcs. Then for each $w\in\Camb_c$, the noncrossing arc diagram $\Delta(w)$ corresponds to the canonical join representation of $w$. It follows that the simplicial complex $\AD(c)$ is isomorphic to the canonical join complex of $\Camb_c$. Say a noncrossing arc diagram in $\AD(c)$ is \dfn{maximal} if it is a facet of $\AD(c)$. In other words, a noncrossing arc diagram in $\AD(c)$ is maximal if it is not properly contained in another noncrossing arc diagram in $\AD(c)$. Let $\MAD(c)$ denote the set of maximal noncrossing arc diagrams in $\AD(c)$. 

The preceding discussion yields the identity \begin{equation}\label{eq:MArc}
\PP_{\Camb_c}(q)=\sum_{\delta\in\MAD(c)}q^{|\delta|}.
\end{equation}

\subsection{Motzkin paths} \label{subsec:Motzkin}
A \dfn{Motzkin path} is a lattice path in the plane that consists of up (i.e., $(1,1)$) steps, down (i.e., $(1,-1)$) steps, and horizontal (i.e., $(1,0)$) steps, starts at the origin, never passes below the horizontal axis, and ends on the horizontal axis. Let $\UU$, $\DD$, and $\HH$ denote up, down, and horizontal steps, respectively. Given a word $P$ over the alphabet $\{\UU,\DD,\HH\}$, let $\#_{\UU}(P)$, $\#_{\DD}(P)$, and $\#_\HH(P)$ denote the number of $\UU's$, the number of $\DD$'s, and the number of $\HH$'s in $P$, respectively. We can think of a Moztkin path as a word $M$ over the alphabet $\{\UU,\DD,\HH\}$ such that $\#_\UU(M)=\#_\DD(M)$ and $\#_\UU(P)\geq\#_\DD(P)$ for every prefix $P$ of $M$. For example, the word $\UU\HH\HH\DD\HH\UU\DD\UU\HH\UU\DD\DD$ represents the Motzkin path in \Cref{fig:Motzkin}. 

\begin{figure}[ht]
  \begin{center}\includegraphics[height=1.2cm]{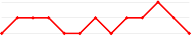}
  \end{center}
  \caption{The Motzkin path $\UU\HH\HH\DD\HH\UU\DD\UU\HH\UU\DD\DD$. }\label{fig:Motzkin}
\end{figure}

A \dfn{peak} of a Motzkin path $M$ is a point $(j,k)$ where an up step in $M$ ends and a down step in $M$ begins; the \dfn{height} of this peak is the number $k$. If we view $M$ as a word over $\{\UU,\DD,\HH\}$, then a peak corresponds to a consecutive occurrence of $\UU\DD$, and the height of the peak is $\#_\UU(P)-\#_\DD(P)$, where $P$ is the prefix of $M$ that ends with the up step involved in the peak. The two peaks of the Motzkin path in \Cref{fig:Motzkin} are the points $(6,1)$ and $(10,2)$. 

Let $\mathcal M_n$ denote the set of Motzkin paths of length $n$, and let $\overline{\mathcal M}_n$ be the subset of $\mathcal M$ consisting of Motzkin paths that do not have any peaks of height $1$. Let $\mathcal M=\bigcup_{n\geq 0}\mathcal M_n$ and $\overline{\mathcal M}=\bigcup_{n\geq 0}\overline{\mathcal M}_n$. Let \[{\bf M}(q,z)=\sum_{n\geq 0}\sum_{M\in\mathcal M_n}q^{\#_\UU(M)}z^{n}\quad\text{and}\quad{\overline{\bf M}}(q,z)=\sum_{n\geq 0}\sum_{M\in\overline{\mathcal M}_n}q^{\#_\UU(M)}z^{n}.\]
If $M$ is a nonempty Motzkin path, then either $M=\UU M'\DD M''$ for some $M',M''\in\mathcal M$, or $M=\HH M'''$ for some $M'''\in\mathcal M$. This translates into the functional equation 
\[
{\bf M}(q,z)-1=qz^2{\bf M}(q,z)^2+z{\bf M}(q,z),
\]
which we can solve to find that 
\[
{\bf M}(q,z)=\frac{1-z-\sqrt{1-2z+(1-4q)z^2}}{2qz^2}.
\]
If $M$ is a nonempty Motzkin path with no peaks of height $1$, then either $M=\UU M'\DD M''$ for some nonempty $M'\in\mathcal M$ and some $M''\in\overline{\mathcal M}$, or $M=\HH M'''$ for some $M'''\in\overline{\mathcal M}$. This translates into the functional equation 
\[
\overline{{\bf M}}(q,z)-1=qz^2({\bf M}(q,z)-1)\overline{\bf M}(q,z)+z\overline{{\bf M}}(q,z).
\]
We can solve this equation and use the above formula for ${\bf M}(q,z)$ to obtain 
\begin{align}
\overline{{\bf M}}(q,z)&=\frac{1}{1-z-qz^2({\bf M}(z,q)-1)} \nonumber\\ &=\frac{2}{1-z+2qz^2+\sqrt{1-2z+(1-4q)z^2}}. \label{eq:M_bar}
\end{align}

Using \Cref{eq:M_bar}, one can readily check that the expression on the right-hand side of \cref{eq:conjectured_gf} is \[\frac{1}{qz}\left(\overline{\bf M}(1/q,qz)-1\right)-1=\sum_{n\geq 1}\sum_{M\in\overline{\mathcal M}_{n+1}}q^{n-\#_\UU(M)}z^{n}.\] Therefore, in order to prove \cref{eq:conjectured_gf}, it suffices (by \cref{eq:MArc}) to exhibit a bijection $\Psi\colon\MAD(c_{(n)}^{\bip})\to\overline{\mathcal M}_{n+1}$ such that $|\delta|=n-\#_\UU(\Psi(\delta))$ for every $\delta\in\MAD(c_{(n)}^{\bip})$. 

\subsection{The bijection} \label{subsec:bijection}

Throughout the remainder of this section, let us fix a positive integer $n$ and write $c^{\bip}=c_{(n)}^{\bip}$. The map $\vv_{c^{\bip}}\colon\{2,\ldots,n\}\to\{\LL,\RR\}$ is given by \[\vv_{c^\bip}(i)=\begin{cases} \LL & \mbox{if }i\mbox{ is odd}; \\ \RR & \mbox{if }i\mbox{ is even}.\end{cases}\] As in \Cref{subsec:arc}, we consider noncrossing arc diagrams drawn on $n+1$ points that are arranged horizontally and identified with the numbers $1,\ldots,n+1$ from left to right. 

\begin{lemma}\label{lem:no_UD}
If $\delta\in\MAD(c^\bip)$ and $i\in\{2,\ldots,n\}$, then $i-1$ is the left endpoint of an arc in $\delta$, or $i+1$ is the right endpoint of an arc in $\delta$ (or both). 
\end{lemma}

\begin{proof}
Suppose instead that $i-1$ is not a left endpoint of an arc in $\delta$ and that $i+1$ is not a right endpoint of an arc in $\delta$. Let $\mathfrak{a}$ be the unique $c^\bip$-sortable arc whose endpoints are $i-1$ and $i+1$. Because $\vv_{c^\bip}(i-1)=\vv_{c^\bip}(i+1)\neq\vv_{c^\bip}(i)$, it is straightforward to check that $\delta\sqcup\{\mathfrak{a}\}$ is a noncrossing arc diagram in $\AD(c^\bip)$. This contradicts the maximality of $\delta$. 
\end{proof}

Suppose $\delta\in\MAD(c^{\bip})$. Let $\Psi(\delta)$ be the word $\MM_1\cdots\MM_{n+1}$, where for $1\leq i\leq n+1$, we define 
\begin{equation}\label{eq:MM}
\MM_i=\begin{cases} \UU & \mbox{if }i\leq n\mbox{ and }i+1\mbox{ is not the right endpoint of an arc in }\delta; \\ \DD & \mbox{if }i\geq 2\mbox{ and }i-1\mbox{ is not the left endpoint of an arc in }\delta; \\ \HH & \mbox{otherwise}. \end{cases}
\end{equation} 
\Cref{lem:no_UD} guarantees that $\Psi(\delta)$ is well defined (i.e., that no letter in $\Psi(\delta)$ can be both $\UU$ and $\DD$). See \Cref{fig:Arc_Motzkin} for an illustration.

\begin{figure}[ht]
  \begin{center}
  \includegraphics[height=6.827cm]{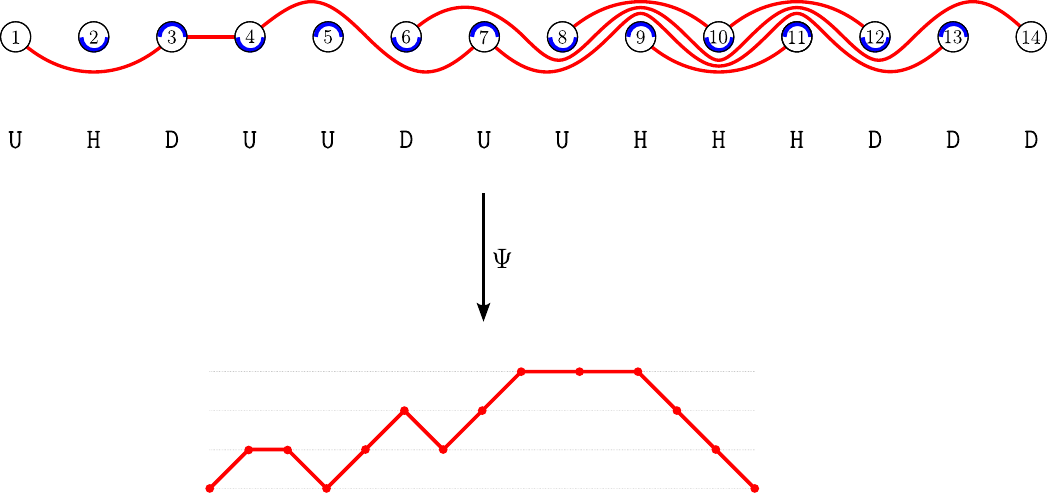}
  \end{center}
  \caption{When $n=13$, the map $\Psi$ sends a noncrossing arc diagram of type $A_{13}$ to a Motzkin path of length $14$ with no peaks of height $1$. For each $2\leq i\leq 13$, a blue semicircle appears on the top (resp.\ bottom) of the circle containing $i$ if $\vv_{c^\bip}(i)=\LL$ (resp.\ if $\vv_{c^\bip}(i)=\RR$). (The letters drawn below the noncrossing arc diagram represent the Moztkin path; they are not part of the noncrossing arc diagram.)}\label{fig:Arc_Motzkin}
\end{figure}

\begin{proposition}\label{prop:main_bijection}
For each $\delta\in\MAD(c^\bip)$, we have $\Psi(\delta)\in\overline{\mathcal M}_{n+1}$ and $|\delta|=n-\#_\UU(\Psi(\delta))$. \end{proposition}

\begin{proof}
Let $\Psi(\delta)=\MM_1\cdots\MM_{n+1}$, where the steps $\MM_1,\ldots,\MM_{n+1}$ are as defined in \cref{eq:MM}. It is immediate from \cref{eq:MM} that $n-\#_\UU(\Psi(\delta))$ and $n-\#_\DD(\Psi(\delta))$ are both equal to $|\delta|$. In particular, $\#_\UU(\Psi(\delta))=\#_\DD(\Psi(\delta))$.

Suppose $k\in\{0,\ldots,n\}$ is such that the prefix $P=\MM_1\cdots\MM_k$ of $\Psi(\delta)$ satisfies $\#_U(P)=\#_\DD(P)$. If $k\geq 1$, then $k-\#_\UU(P)$ is the number of elements of $\{2,\ldots,k+1\}$ that are right endpoints of arcs in $\delta$. Similarly, if $k\geq 1$, then $k-1-\#_\DD(P)$ is the number of elements of $\{1,\ldots,k-1\}$ that are left endpoints of arcs in $\delta$. It follows that either $k=0$, or there is an arc from $k$ to $k+1$ in $\delta$. In either case, it is immediate from \cref{eq:MM} that $\MM_{k+1}\neq \DD$. Moreover, none of the arcs with left endpoints in $\{1,\ldots,k\}$ have right endpoints in $\{k+2,\ldots,n+1\}$. If we had $\MM_{k+1}=\UU$ and $\MM_{k+2}=\DD$, then we could add the arc from $k+1$ to $k+2$ to $\delta$ in order to form a larger noncrossing arc diagram in $\AD(c^\bip)$, contradicting the maximality of $\delta$. Together, this shows that $\Psi(\delta)$ cannot pass below the horizontal axis and cannot have any peaks of height $1$. Hence, $\Psi(\delta)\in\overline{\mathcal M}_{n+1}$. 
\end{proof}

We can now state and prove the main theorem of this section; as mentioned at the end of \Cref{subsec:Motzkin}, this theorem and \Cref{prop:main_bijection} imply the identity in \cref{eq:conjectured_gf}, thereby settling the conjecture of Defant and Williams.

\begin{theorem}\label{thm:main_bijection}
The map $\Psi\colon\MAD(c^\bip)\to\overline{\mathcal M}_{n+1}$ is a bijection. 
\end{theorem}

\begin{proof}
Given an arc $\mathfrak{a}$ and a vertical line $\mathfrak l$ that lies strictly between the endpoints of $\mathfrak{a}$, let $\mathfrak{a}_{\mathfrak l}$ be the curve obtained from $\mathfrak{a}$ by deleting the portion of $\mathfrak{a}$ that lies to the right of $\mathfrak l$. We call a curve $\mathfrak{a}_{\mathfrak l}$ obtained in this manner a \dfn{partial arc}. 

Consider $M=\MM_1\cdots\MM_{n+1}\in\overline{\mathcal M}_{n+1}$. Let us try to construct ${\delta\in\MAD(c^\bip)}$ such that $\Psi(\delta)=M$. We can build $\delta$ step by step from left to right. This process is illustrated in \Cref{fig:step_by_step}. For the first step, we create a small partial arc that starts at the point $1$. In general, at the $k$-th step, we consider the point $k$. At this point in time, there may be some partial arcs that were produced in the previous $k-1$ steps. If $k\geq 2$ and $\MM_{k-1}\neq \UU$, then we extend one of these partial arcs and attach it to the point $k$ in order to create an arc, and we extend the rest of the partial arcs above (if $k$ is odd) or below (if $k$ is even) the point $k$. 
The condition that arcs cannot cross uniquely determines which partial arc we must attach to the point $k$. 
During the same $k$-th step, if $k\leq n$ and $\MM_{k+1}\neq \DD$, then we create a new small partial arc that starts at the point $k$. 

To see that the above procedure succeeds in producing a noncrossing arc diagram in $\AD(c^\bip)$, we must check two things: 
\begin{itemize}
\item If $k\geq 2$ and $\MM_{k-1}\neq \UU$, then there is actually a partial arc from the first $k-1$ steps that we can attach to the point $k$. 
\item The total number of steps during which we create a new partial arc is equal to the total number of steps during which we attach a partial arc to a point to create an arc. 
\end{itemize}
For the first item, note that the number of partial arcs remaining after the first $k-1$ steps is $(k-1-\#_\DD(\MM_2\cdots\MM_k))-(k-2-\#_\UU(\MM_1\cdots\MM_{k-2}))=\#_\UU(\MM_1\cdots\MM_{k-2})-\#_\DD(\MM_2\cdots\MM_k)+1$. Because $M$ is a Motzkin path, this quantity is nonnegative; furthermore, it can only be $0$ if the steps $\MM_{k-1}$ and $\MM_k$ are $\UU$ and $\DD$ (respectively) and form a peak of height $1$. But $M$ has no peaks of height $1$, so the number of partial arcs remaining after $k-1$ steps must actually be at least $1$. The second bulleted item follows from the fact that $\#_\UU(M)=\#_\DD(M)$. 

Let us now show that the noncrossing arc diagram $\delta$ constructed through the above procedure is actually maximal in $\AD(c^\bip)$; it will then follow immediately from the construction that $\Psi(\delta)=M$. Note also that the steps outlined above for constructing $\delta$ were forced upon us; in other words, $\Psi$ is injective.

Let $\mathfrak{a}_{i,j}$ denote the unique $c^\bip$-sortable arc from $i$ to $j$. Suppose by way of contradiction that there exist $i<j$ such that $\mathfrak{a}_{i,j}$ is not in $\delta$ and $\delta\sqcup\{\mathfrak{a}_{i,j}\}$ is a noncrossing arc diagram. We will assume $i$ is even; the case when $i$ is odd is almost identical. We consider three cases. 

\medskip 

\noindent {\bf Case 1.} Assume $j=i+1$. Then we have $\MM_i=\UU$ and $\MM_{i+1}=\DD$, so the steps $\MM_i$ and $\MM_{i+1}$ form a peak. Since $\mathfrak{a}_{i,i+1}$ does not cross any of the arcs in $\delta$, none of the arcs in $\delta$ can pass below $i$ and above $i+1$. Hence, every arc in $\delta$ whose left endpoint is in $\{1,\ldots,i-1\}$ has its right endpoint in $\{2,\ldots,i\}$. It follows from the construction of $\delta$ that $\#_\UU(\MM_1\cdots\MM_{i-1})=\#_\DD(\MM_1\cdots\MM_{i-1})$. But this implies that the peak formed by $\MM_i$ and $\MM_{i+1}$ has height $1$, which contradicts the fact that $M$ has no peaks of height $1$. 

\medskip 

\noindent {\bf Case 2.} Assume $j=i+2$. Then the point $i$ is not the left endpoint of an arc in $\delta$, and the point $i+2$ is not the right endpoint of an arc in $\delta$. However, this implies that $\MM_{i+1}$ is both $\UU$ and $\DD$, which is impossible. 

\medskip 

\noindent {\bf Case 3.} Assume $j\geq i+3$. Note that $\MM_{i+1}=\DD$. Since $\MM_{i+1}\neq\UU$, there must be an arc $\mathfrak{a}'$ in $\delta$ whose right endpoint is $i+2$. The arc $\mathfrak{a}_{i,j}$ passes above $i+1$ and below $i+2$. The left endpoint of $\mathfrak{a}'$ cannot be $i+1$ since, if it were, the arcs $\mathfrak{a}_{i,j}$ and $\mathfrak{a}'$ would cross. Therefore, $\mathfrak{a}'$ passes above $i+1$. The lower endpoint of $\mathfrak{a}'$ also cannot be $i$ because $i$ is the left endpoint of $\mathfrak{a}_{i,j}$. Therefore, $\mathfrak{a}'$ passes below $i$. However, this forces $\mathfrak{a}_{i,j}$ and $\mathfrak{a}'$ to cross, which is a contradiction. 
\end{proof}

\begin{figure}[ht]
  \begin{center}\includegraphics[height=12.708cm]{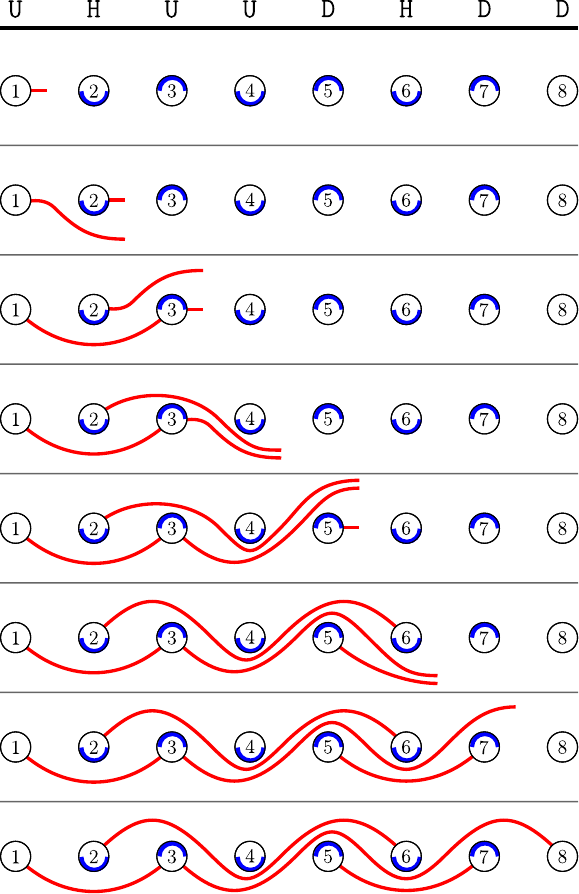}
  \end{center}
  \caption{The step-by-step procedure for constructing the noncrossing arc diagram $\delta$ such that $\Psi(\delta)=\UU\HH\UU\UU\DD\HH\DD\DD$.}\label{fig:step_by_step}
\end{figure}

\section{Maximum-Size Pop-Stack Orbits}\label{sec:orbits}

Let $W$ be a finite irreducible Coxeter group. All Coxeter elements of $W$ are conjugate to each other, so they have the same group-theoretic order; this order is called the \dfn{Coxeter number} of $W$ and is denoted by $h$. (It is also known that $h=2N/n$, where $N$ and $n$ are the number of reflection in $W$ and the number of simple reflections in $W$, respectively.)
Recall that, given a finite lattice $L$, we write $\mathcal O_L(x)$ for the forward orbit of an element $x\in L$ under the pop-stack operator. 
If $c$ is a Coxeter element of $W$, then $\Camb_c$ is a quotient of $\Weak(W)$.
Let $L$ be an arbitrary lattice quotient $W_{\equiv}$ on $W$.
We begin this section by showing that $h$ is an upper bound for $\max_{x\in W_{\equiv}}|\mathcal O_{W_{\equiv}}(x)|$ when $L$ is an arbitrary quotient of $\Weak(W)$ (\cref{thm:upper_quotient}).
In the process, we also provide a formula for computing $\pop_{W_\equiv}$ for any lattice quotient $W_{\equiv}$ of the weak order.
We then prove that this upper bound is tight when $W_{\equiv}$ is a Cambrian lattice (\cref{thm:max_orbit}). We also discuss how the techniques used in this section relate back to representation theory in \cref{subsec:heap_rep}.

\subsection{Quotients of the weak order}
As above, let $L$ be a finite lattice.
An equivalence relation $\equiv$ on $L$ is a \dfn{lattice congruence} provided that equivalence classes respect the meet and join operations in $L$.
That is, $\equiv$ is a lattice congruence if for all $x,y,z\in L$ satisfying $x\equiv y$, we have $x\join z \equiv y\join z$ and $x\meet z \equiv y\meet z$.

Let $\equiv$ be a lattice congruence. Let $[x]_\equiv$ denote the equivalence class of an element $x\in L$ under $\equiv$. One can define a lattice structure on the set $L/\equiv$ of equivalence classes as follows.
We set $[x]_\equiv\le [y]_{\equiv}$ if and only if there exist $x'\in [x]_\equiv$ and $y'\in [y]_{\equiv}$ such that $x'\le y'$. We then have $[x]_\equiv \meet [y]_\equiv = [x\meet y]_\equiv$ and $[x]_\equiv \join [y]_\equiv = [x\join y]_\equiv$. The lattice $L/\equiv$ is called a \dfn{lattice quotient} of $L$. Observe that there is a surjective lattice homomorphism $\eta_\equiv\colon L\to L/\equiv$ that sends each element $x\in L$ to its equivalence class $[x]_\equiv$.
(See \cite[Proposition~9.5-4]{reading_book}.)

Given an equivalence relation $\equiv$ on $L$, one can check that $\equiv$ is a lattice congruence order-theoretically using the following result.
\begin{proposition}[{\cite[Proposition~9-5.2]{reading_book}}]
An equivalence relation $\equiv$ on a finite lattice $L$ is a lattice congruence if and only if the following three conditions hold:
    \begin{itemize}
    \item Each equivalence class is an interval of $L$.
    \item The map $\pidown^\equiv$ that sends each element $x\in L$ to the unique minimal element of its equivalence class is order-preserving.
    \item The map $\piup_\equiv$ that sends each element $x\in L$ to the unique maximal element of its equivalence class is order-preserving.
    \end{itemize}
\end{proposition}

Finally, we note that if $\equiv$ is a lattice congruence on $L$, then the map $\pidown^\equiv\colon L\to L$ is a lattice homomorphism whose image $\pidown^\equiv(L)$ (taken as an induced subposet of $L$) is isomorphic to the quotient $L/\equiv$.
(Dually, the image of $\piup_\equiv$ is also isomorphic to the quotient $L/\equiv$.)
\begin{remark}
Although the subposet $\pidown^\equiv(L)$ is a lattice in its own right, it meet and join operations may not coincide with those of $L$. Said another way, $\pidown^\equiv(L)$ is not generally a \emph{sublattice} of $L$. In general, the join operation of $\pidown^\equiv(L)$ does coincide with that of $L$; dually, the meet operation of $\piup_\equiv(L)$ does coincide with that of $L$.
However, each $c$-Cambrian lattice (of type $W$) \emph{is} both a lattice quotient and sublattice of the weak order on $W$.
\end{remark}

Below, we realize an arbitrary lattice quotient of the weak order as the subposet induced by the image of $\pidown^\equiv$, where $\equiv$ is a lattice congruence.
We have two main tasks as we consider the pop-stack operator on $\pidown^\equiv(W)$. First, we must we identify the lower covers of each $x\in \pidown^\equiv(W)$;
this is done in the following lemma.

\begin{lemma}\label{lem:quotient_covers}
Let $L$ be a finite lattice, and let $\equiv$ be a lattice congruence on $L$.
If $x\in \pidown^\equiv(L)$, then the map $\pidown^\equiv$ restricts to a bijection from the set of elements covered by $x$ in $L$ to the set of elements covered by $x$ in $\pidown^\equiv(L)$.
\end{lemma}
\begin{proof}
First assume that $a\lessdot x$ in $L$, and let $a'=\pidown^\equiv(a)$.
Since $\pidown^\equiv\colon L\to L$ is order-preserving, we have $a'\le x$. Suppose that there exists $b\in \pidown^\equiv(L)$ with $a'<b<x$.
We cannot have $a'<b<a$ because $a'$ is the largest element of $\pidown^\equiv(L)$ lying below $a$.
Moreover, we cannot have $a'<a<b$ because $a\covered x$. Therefore, $a$ and $b$ are incomparable.
This implies that $a\join b = x$ (because $a\covered x$ and $b<x$).
Since $a\equiv a'$, we have $x = a\join b \equiv a'\join b= b$.
But this contradicts the fact that $b$ and $x$ are distinct elements of $\pidown^\equiv(L)$.
From this contradiction, we deduce that $a'\covered x$ in $\pidown^\equiv(L)$. 

We have shown that $\pidown^\equiv$ restricts to a map from the set of elements covered by $x$ in $L$ to the set of elements covered by $x$ in $\pidown^\equiv(L)$; we must show that it is bijective. To show it is injective, suppose instead that there exist distinct elements $a_1,a_2$ that are covered by $x$ in $L$ such that $\pidown^\equiv(a_1)=\pidown^\equiv(a_2)$. Then $a_1\equiv a_2$, so $a_1\equiv a_1\join a_2=x$, contradicting the fact that $x$ is the unique minimal element of $[x]_\equiv$. To prove surjectivity, suppose $d\covered x$ in $\pidown^\equiv(L)$.
Then there exists $d'\in L$ such that $d\le d'\covered x$. Since $\pidown^\equiv$ is order-preserving, we have $d \le \pidown^{\equiv}(d')<x$.
Thus, $d=\pidown^\equiv(d')$.
\end{proof}

Now that we have identified the lower covers of the elements of a lattice quotient $\pidown^\equiv(L)$, we describe the pop-stack operator on a lattice quotient of $L$. It will be convenient to write $\meet_L$ and $\meet_\equiv$ for the meet operation of $L$ and the meet operation of $\pidown^\equiv(L)$, respectively.

\begin{proposition}\label{prop:quotient_pop}
Let $L$ be a finite lattice, and let $\equiv$ be a lattice congruence on $L$. For $x\in \pidown^\equiv(L)$, we have
\[\popdown_{\pidown^\equiv(L)}(x) = \pidown^\equiv\left(\Meet\nolimits_L \{ \pidown^\equiv(a)\mid a\covered x \text{ in $L$}\}\right).\]
\end{proposition}

\begin{proof}
By \cref{lem:quotient_covers}, it suffices to show that for $x,y\in \pidown^\equiv(L)$, we have $x\meet_{\equiv} y = \pidown^\equiv(x\meet_L y)$.
Observe that $x\meet_L y\ge x\meet_\equiv y\ge \pidown^{\equiv}(x\meet_L y)$ in $L$.
Since $x\meet_\equiv y$ and $\pidown^\equiv(x\meet_L y)$ are both in the image of $\pidown^\equiv$, the desired result follows from the fact that $\pidown^\equiv(x\meet_L y)$ is the largest element of $\pidown^\equiv(L)$ lying below $x\meet_L y$.
\end{proof}

\begin{remark}\label{rmk:cambrian_vs_general} 
Let $L$ be a the weak order on a finite Coxeter group $W$. Let $\equiv$ be the $c$-Cambrian congruence on $W$.
The fact that each $c$-Cambrian lattice is a sublattice of the weak order implies that $\pidown^\equiv$ and the meet operation of the weak order commute (see, e.g., \cite[Lemma~3.1]{DefantMeeting}).
Hence,
\begin{align*}
\popdown_{\Camb_c}(x)&=\pidown^\equiv\left(\Meet\nolimits_L \{ \pidown^\equiv(a): a\covered x \text{ in $L$}\}\right) =\pidown^\equiv\left(\Meet\nolimits_{\equiv} \{ \pidown^\equiv(a): a\covered x \text{ in $L$}\}\right)\\&=
\Meet\nolimits_{\equiv} \{ \pidown^\equiv(a): a\covered x \text{ in $L$}\}=\Meet\nolimits_L \{ \pidown^\equiv(a): a\covered x \text{ in $L$}\}\\
&=\pidown^\equiv\left(\Meet\nolimits_L \{ a: a\covered x \text{ in $W$}\}\right)=\pidown^\equiv\circ \popdown_L(x).
\end{align*}
We obtain equality from the first line to the second line because the meet taken in the $c$-Cambrian lattice stays in the $c$-Cambrian lattice (so the outermost $\pidown^\equiv$ acts as the identity).
The equality from the second to the third line results from the fact that $\pidown^\equiv$ and the meet operation commute.
Thus, we obtain the simpler expression for $\pop_{\Camb_c}$ from \cref{eq:pop_Cambrian}.
\end{remark}

We now specialize to the case where $L$ is the weak order on $W$ and prove our main result of this subsection.
Let $W_{\equiv}$ denote a lattice quotient of $W$ coming from a congruence $\equiv$.
\begin{theorem}\label{thm:upper_quotient}
If $W_{\equiv}$ is a lattice quotient of $\Weak(W)$, then \[\max_{x\in W_{\equiv}}\left\lvert\mathcal O_{W_{\equiv}}(x)\right\rvert\leq h.\]
\end{theorem}

\begin{proof}
We may assume that $W_{\equiv}$ is the subposet $\pidown^\equiv(W)$ of $\Weak(W)$, where $\equiv$ is a lattice congruence on $\Weak(W)$. 
We write $\meet_{\equiv}$ and $\lessdot_{\equiv}$ to denote meets and cover relations, respectively, in $W_{\equiv}$ (while $\meet$ and $\lessdot$ will denote meets and cover relations in $\Weak(W)$). 
Define a map $g\colon W\to W$ by 
\[g(x) =
\begin{cases}
\popdown_{W_{\equiv}}(x) & \text{if }x\in W_{\equiv};\\
\popdown_{\Weak(W)}(x) & \text{ otherwise}.
\end{cases}
\]
Note that \[\max_{x\in W_{\equiv}}\left\lvert\mathcal O_{W_{\equiv}}(x)\right\rvert\leq\max_{w\in6 W}\left\lvert\Orb_g(w)\right\rvert\]
(recall that $\Orb_g(w)$ denotes the forward orbit of $w$ under $g$). 

Following \cite{DefantPopCoxeter}, we say a map $f\colon W\to W$ is \dfn{compulsive} if $f(x)\leq\popdown_{\Weak(W)}(x)$ for all $x\in W$. 
\cite[Theorem~1.5]{DefantPopCoxeter} states that if $f$ is compulsive, then \[\max_{w\in W}\left\lvert\Orb_f(w)\right\rvert\leq h.\] 
Hence, it suffices to show that $g$ is compulsive.

Consider $x\in W_{\equiv}\setminus\hat 0$.
By \cref{lem:quotient_covers}, we have \[g(x)=\popdown_{W_{\equiv}}(x)=\bigwedge\nolimits_\equiv\{y\in W_\equiv\mid y\lessdot_{\equiv} x\}=\bigwedge\nolimits_{\equiv}\{\pidown^\equiv(a)\mid a\lessdot x\}.\] But then since $\pidown^\equiv(a)\leq a$ for all $a$, we have 
\[\bigwedge\nolimits_{\equiv}\{\pidown^\equiv(a)\mid a\lessdot x\} \leq \bigwedge\{\pidown^\equiv(a)\mid a\lessdot x\} \leq \bigwedge\{a\mid a\lessdot x\}=\popdown_{\Weak(W)}(x).\] 
On the other hand, $g(z)=\popdown_{\Weak(W)}(z)$ whenever $z\in W\setminus W_{\equiv}$ or $z=\hat 0$. This shows that $g$ is compulsive. 
\end{proof}

\begin{remark}
One cannot weaken the hypothesis in \cref{thm:upper_quotient} by assuming that $L$ is just a meet-semilattice quotient or a join-semilattice quotient. For example, suppose $W=A_2$. Let $\equiv$ be the meet-semilattice congruence on $\Weak(W)$ with one nontrivial equivalence class consisting of $\hat{0}$, $s_1$, and $s_1s_2$. The lattice $L=W/\equiv$ is a $4$-element chain, so $\max_{x\in L}|\mathcal O_{L}(x)|=4>3=h$. 
\end{remark}

\subsection{Heaps and combinatorial AR quivers}\label{sec:heaps}
Our goal is now to show that the upper bound in \Cref{thm:upper_quotient} is tight when $L$ is a Cambrian lattice $\Camb_c$. We will do this by constructing an explicit element $\zz_c\in\Camb_c$ whose forward orbit under $\popdown_{\Camb_c}$ has size exactly $h$. Our construction makes use of the \emph{combinatorial AR quivers} of \cite[Section~9]{cataland}. As we explain in \cref{subsec:heap_rep}, these are intimately related with the representation theory of the tensor algebra $KQ_c$.

It will be convenient to think of the letters in a word over the alphabet $S$ as being distinct entities even if they represent the same simple reflection. For example, the two occurrences of $\s_2$ in the word $\s_2\s_1\s_2\s_3$ are considered to be different. If $\s$ and $\s'$ are letters representing simple reflections $s$ and $s'$, then we let $m(\s,\s')=m(s,s')$ (recall that $m(s,s')$ is the order of $ss'$). 
We can apply a \dfn{commutation move} to a word by swapping two consecutive letters $\s$ and $\s'$ if $m(\s,\s')=2$ (this move is not allowed if $m(\s,\s')=1$). The \dfn{commutation class} of a word $\QQ$ over $S$ is the set of words that can be obtained from $\QQ$ via a sequence of commutation moves. We write $\QQ\equiv\QQ'$ if $\QQ$ and $\QQ'$ are in the same commutation class. 

Let $\QQ$ be a word over $S$. The \dfn{heap} of $\QQ$ is the poset $\Heap(\QQ)$ whose elements are the letters in $\QQ$ (which we view as distinct from one another) and whose order relation is defined so that if $\s$ and $\s'$ are two letters (which could represent the same simple reflection), then $\s<\s'$ if and only if $\s$ appears to the left of $\s'$ in every word in the commutation class of $\QQ$. A word $\QQ'$ is in the same commutation class as $\QQ$ if and only if $\Heap(\QQ)=\Heap(\QQ')$. More generally, $\Heap(\QQ')$ is an order ideal of $\Heap(\QQ)$ if and only if there exists a word $\QQ''$ such that $\QQ'\QQ''\equiv\QQ$.  

Suppose $\HHH$ is the heap of a word $\mathsf{w}$. We write $\R(\HHH)$ for the element of $W$ represented by $\mathsf{w}$ and write $Z(\HHH)$ for the set of simple reflections that appear in $\mathsf{w}$. If $\mathsf{w}\equiv\mathsf{w}'$, then $\mathsf{w}$ and $\mathsf{w'}$ represent the same element of $W$; therefore, $\R(\HHH)$ and $Z(\HHH)$ depend only on $\HHH$ (and not the particular word $\mathsf{w}$). Let $\mathcal J(\HHH)$ denote the set of order ideals of $\HHH$.

Let $\sfc$ be a reduced word for a Coxeter element $c$, and let $\sfc^k$ denote the word obtained by concatenating $\sfc$ with itself $k$ times. Let $S=\{s_1,\ldots,s_n\}$, and let $\s_i^{(j)}$ denote the $j$-th letter in $\sfc^k$ that represents $s_i$. For example, if $\sfc=\s_1\s_3\s_2$, then \[\sfc^3=\s_1^{(1)}\s_3^{(1)}\s_2^{(1)}\s_1^{(2)}\s_3^{(2)}\s_2^{(2)}\s_1^{(3)}\s_3^{(3)}\s_2^{(3)}.\]

\begin{lemma}\label{lem:ranked}
For each $k\geq 1$, the poset $\Heap(\sfc^k)$ is ranked. 
\end{lemma}

\begin{proof}
Because $\Gamma_W$ is a tree, it is straightforward to see that $\Heap(\sfc)$ is ranked. Let $\rank\colon S\to\mathbb Z$ be the unique rank function of $\Heap(\sfc)$ whose minimum value is $1$. Let us extend this rank function to a map $\rank\colon\Heap(\sfc^k)\to\mathbb Z$ by letting $\rank(\s_i^{(j)})=\rank(s_i)+2j-2$. We claim that this map is a rank function on $\Heap(\sfc^k)$. To see this, suppose $\s_i^{(j)}\lessdot\s_{i'}^{(j')}$ is a cover relation in $\Heap(\sfc^k)$. Then $\{s_i,s_{i'}\}$ is an edge in $\Gamma_W$. 
If $\s_i$ appears before $\s_{i'}$ in $\sfc$, then we have $\rank(s_{i'})=\rank(s_i)+1$ and $j'=j$.
If $\s_i$ appears after $\s_{i'}$ in $\sfc$, then we have $\rank(s_{i'})=\rank(s_i)-1$ and $j'=j+1$. In either case, we compute that $\rank(\s_{i'}^{(j')})=\rank(\s_i^{(j)})+1$.  
\end{proof}

Let $\rank\colon\Heap(\sfc^h)\to\mathbb Z$ be the rank function on $\Heap(\sfc^h)$ that was constructed in the proof of \Cref{lem:ranked}. 
When we consider a heap $\HHH$ that can be naturally identified with a convex subset of $\Heap(\sfc^h)$, the \dfn{rank} of an element $\s$ of $\HHH$ will simply be $\rank(\s)$, the rank of $\s$ in $\Heap(\sfc^h)$. We will also find it convenient to write \[\HHH^k=\{\s\in\HHH\mid\rank(\s)=k\}\quad\text{and}\quad\HHH^{\leq k}=\{\s\in\HHH\mid\rank(\s)\leq k\}.\] 

There is an involution $\psi\colon S\to S$ given by $\psi(s)=\wo s\wo$; this map is an automorphism of $\Gamma_W$. We can extend $\psi$ to the set of words over $S$ by letting $\psi(\s_1\cdots \s_M)=\psi(\s_1)\cdots\psi(\s_M)$.

Let $\HHH_c=\Heap(\sort_{\sfc}(\wo))$ and $\widetilde{\HHH}_c=\Heap(\psi(\sort_{\sfc}(\wo)))$; note that $\HHH_c$ and $\widetilde{\HHH}_c$ depend only on the Coxeter element $c$ (and not the reduced word $\sfc$). It follows from \cite[Lemma~2.6.5]{cataland} that 
\begin{equation}\label{eq:Heap_c_sorting}
\sort_{\sfc}(\wo)\psi(\sort_{\sfc}(\wo))\equiv\sfc^h.
\end{equation}
Therefore, we can think of $\HHH_c$ as an order ideal of $\Heap(\sfc^h)$ whose complement (in $\Heap(\sfc^h)$) is $\widetilde{\HHH}_c$. Define \[\zz_c=\R(\HHH_c^{\leq h-1}).\]

\begin{remark}\label{rem:spine}
We have defined $\zz_c$ using Coxeter--Catalan combinatorics because we deemed this approach to be the most amenable for proving that $|\mathcal O_{\Camb_c}(\zz_c)|=h$. However, there is also a more direct lattice-theoretic way to define $\zz_c$. The \dfn{spine} of $\Camb_c$, denoted $\spine(\Camb_c)$, is the union of the maximum-length chains of $\Camb_c$. Thomas and Williams \cite{TW1} proved that $\spine(\Camb_c)$ is a distributive sublattice of $\Camb_c$. It is also known \cite{TW1} that \[\spine(\Camb_c)=\{\R(I)\mid I\in \mathcal J(\HHH_c)\}.\]
One can show that $\zz_c=(\pop^{\uparrow}_{\spine(\Camb_c)})^{h-1}(e)$ (where $e=\hat 0$ is the identity element).  
\end{remark}

\begin{example}\label{exam:A}
Suppose $W=A_8$. The edges in $\Gamma_W$ are $\{s_i,s_{i+1}\}$ for $1\leq i\leq 7$. The Coxeter number of $W$ is $h=9$. The map $\psi$ is given by $\psi(s_k)=s_{9-k}$. Let $c=s_1s_3s_2s_4s_6s_5s_7s_8$ and $\sfc=\s_1\s_3\s_2\s_4\s_6\s_5\s_7\s_8$. The acyclic orientation of $\Gamma_W$ corresponding to $c$ is \[Q_c = \left(\includegraphics[height=0.25cm]{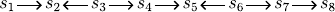}\right).\] 
\Cref{fig:A} shows the Hasse diagram of $\Heap(\sfc^9)$; it is drawn sideways so that each cover relation $\s_i^{(j)}\lessdot\s_{i'}^{(j')}$ is drawn with $\s_i^{(j)}$ to the left of $\s_{i'}^{(j')}$. As explained in \cref{subsec:heap_rep}, this Hasse diagram is referred to as the \emph{combinatorial AR quiver} of $\sfc^9$ in \cite[Section~9.2]{cataland}. Note that $\Heap(\sfc^9)$ is drawn by lining up $9$ copies of $\Heap(\sfc)$ (depicted in $9$ different colors) in a row and adding (black) edges as appropriate. The columns are labeled from $1$ to $20$ from left to right; the column labeled $i$ consists of the elements of rank $i$. The region shaded in yellow is $\HHH_c$, while the unshaded region is $\widetilde\HHH_c$. The region outlined in blue is $\HHH_c^{\leq h-1}=\HHH_c^{\leq 8}$. Thus, \[\zz_c={\color{red}s_1s_3s_2s_4s_6s_5s_7s_8}{\color{NormalGreen}s_1s_3s_2s_4s_6s_5s_7s_8}{\color{ChillBlue}s_1s_3s_2s_4s_6s_5s_7s_8}{\color{MyPurple}s_1s_3s_2s_4s_6}.\]
\end{example}

\begin{figure}[ht]
  \begin{center}{\includegraphics[height=7.287cm]{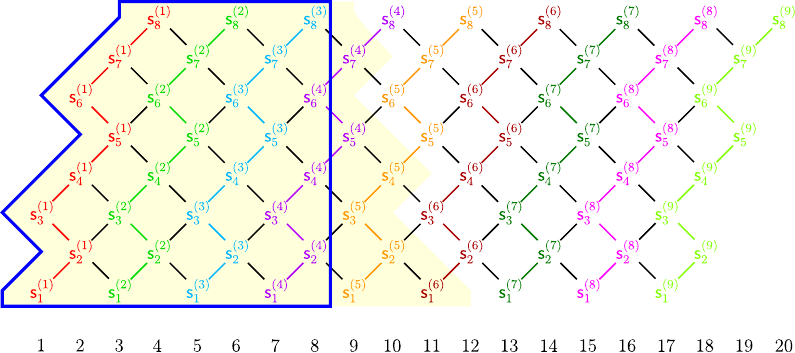}}
  \end{center}
  \caption{The Hasse diagram of $\Heap(\sfc^h)$ from \Cref{exam:A}, in which $W=A_8$. The heap $\HHH_c$ is shaded in yellow, while $\HHH_c^{\leq h-1}$ is outlined in blue.}\label{fig:A}
\end{figure}

\subsection{Realization of the upper bound}\label{subsec:maximizer}

We need a little more preparation before we can prove that $|\mathcal O_{\Camb_c}(\zz_c)|=h$. 

For $k\in\mathbb Z$, let \[\varepsilon(k)=\begin{cases} 1 & \mbox{if }k\mbox{ is odd}; \\ 2 & \mbox{if }k\mbox{ is even}.\end{cases}\] 
Note that for each $\s_i^{(j)}\in\Heap(\sfc^h)$, the parity of $\rank(\s_i^{(j)})$ only depends on $s_i$ (and not $j$). Let \[X_1=\{s_i\in S\mid\rank(\s_i^{(j)})\text{ is odd for all }j\}\quad\text{and}\quad X_2=\{s_i\in S\mid\rank(\s_i^{(j)})\text{ is even for all }j\}.\] Then $S=X_1\sqcup X_2$ is a bipartition of the tree $\Gamma_W$. 

\begin{lemma}\label{lem:column_h-1}
We have $Z(\HHH_c^{h-1})=X_{\varepsilon(h-1)}$. 
\end{lemma}

\begin{proof}
We already know that $Z(\HHH_c^{h-1})\subseteq X_{\varepsilon(h-1)}$, so we just need to prove the reverse containment. We will show that each element of $\widetilde \HHH_c$ has rank at least $h+1$; this will imply that the elements of $\Heap(\sfc^h)$ with rank $h-1$ must belong to $\HHH_c$.

Fix a simple reflection $s_i\in S$, and let $s_{i'}=\psi(s_i)$. Let $j$ be the smallest integer such that $\s_i^{(j)}\in\widetilde \HHH_c$. Let $j'$ be the smallest integer such that $\s_{i'}^{(j')}\in\widetilde \HHH_c$. Then $j'-1$ is the number of letters in $\sort_{\sfc}(\wo)$ that represent $s_{i'}$, so it is also the number of letters in $\psi(\sort_{\sfc}(\wo))$ that represent $s_i$. But the number of letters in $\psi(\sort_{\sfc}(\wo))$ that represent $s_i$ is $h-j+1$, so $j+j'-2=h$. Hence, 
\[\rank(\s_i^{(j)})=\rank(s_i)+2j-2=\rank(s_i)+2h-2j'+2.\]
Because $\psi\colon S\to S$ is an automorphism of $\Gamma_W$, the map $\psi\colon\HHH_c\to\widetilde\HHH_c$ is a poset isomorphism. It follows that $\rank(\s_{i'}^{(j')})-\rank(\s_i^{(j)})=\rank(\s_i^{(1)})-\rank(\s_{i'}^{(1)})=\rank(s_i)-\rank(s_{i'})$. By definition, we have $\rank(\s_{i}^{(j)})=\rank(s_{i})+2j-2$ and $\rank(\s_{i'}^{(j')})=\rank(s_{i'})+2j'-2$, so  \[\rank(s_{i'})+(2j'-2)-(\rank(s_i)+2j-2)=\rank(s_i)-\rank(s_{i'}).\] Rearranging this equation yields $\rank(s_i)=\rank(s_{i'})+j'-j$. Consequently, \[\rank(\s_i^{(j)})=\rank(s_i)+2j-2=\rank(s_{i'})+j+j'-2=\rank(s_{i'})+h.\] In particular, $\rank(\s_i^{(j)})\geq h+1$. This shows that every element of $\widetilde\HHH_c$ that represents $s_i$ has rank at least $h+1$. As $s_i$ was arbitrary, it follows that all of the elements of $\widetilde\HHH_c$ have rank at least $h+1$. 

Suppose $s_m\in X_{\varepsilon(h-1)}$, and let $k$ be the unique integer such that $\rank(\s_m^{(k)})=h-1$. It follows from the preceding paragraph that $\s_m^{(k)}\not\in\widetilde\HHH_c$, so $\s_m^{(k)}\in\HHH_c$. This shows that $s_m\in Z(\HHH_c^{h-1})$. Hence, $X_{\varepsilon(h-1)}\subseteq Z(\HHH_c^{h-1})$. 
\end{proof}

\begin{example}\label{exam:A2}
Let us use \Cref{fig:A} to illustrate the proof of \Cref{lem:column_h-1}. Recall that, in this example, the map $\psi$ is given by $\psi(s_k)=s_{9-k}$. Let $i=3$ so that $i'=6$. Then $j=6$, $j'=5$, $\rank(s_i)=1$, and $\rank(s_{i'})=2$. Note that $j+j'-2=9=h$ and $\rank(\s_i^{(j)})=11=h+\rank(s_{i'})$. We have $Z(\HHH_c^{h-1})=\{s_2,s_4,s_6,s_8\}=X_2=X_{\varepsilon(h-1)}$, as \Cref{lem:column_h-1} claims. 
\end{example}

Let $c_1=\wo(X_1)=\prod_{s\in X_1}s$ and $c_2=\wo(X_2)=\prod_{s\in X_2}s$, and fix reduced words $\sfc_1$ and $\sfc_2$ for $c_1$ and $c_2$, respectively. The element $c^{\bip}=c_1c_2$ is called a \dfn{bipartite Coxeter element} of $W$. (In type A, this definition of $c^\bip$ coincides with that of either $c^\bip_{(n)}$ or $c^{\bip\bip}_{(n)}$ from \cref{subsec:perms}.) Let $\sfc^{\bip}=\sfc_1\sfc_2$. Let $\uu_k$ be the word $\sfc_1\sfc_2\sfc_1\cdots\sfc_{\varepsilon(k)}$ that begins with $\sfc_1$ and alternates between $\sfc_1$ and $\sfc_2$, using $\left\lceil k/2\right\rceil$ copies of $\sfc_1$ and $\left\lfloor k/2\right\rfloor$ copies of $\sfc_2$. In other words, \[\uu_k=\begin{cases} (\sfc^{\bip})^{k/2} & \mbox{if }k\mbox{ is even}; \\ (\sfc^{\bip})^{(k-1)/2}\sfc_1 & \mbox{if }k\mbox{ is odd}.\end{cases}\]
It is known \cite[$\S$V.6~Exercise~2]{Bourbaki} that $\sort_{\sfc^{\bip}}(\wo)=\uu_h$. Hence, for $1\leq j\leq k\leq h$, the word $\uu_k$ is a reduced word for an element $u_k\in W$, and $Z(\Heap(\uu_k)^j)=X_{\varepsilon(j)}$. 

\begin{lemma}
For each $k\in[h-1]$, we have $\Des(u_k)=X_{\varepsilon(k)}$. 
\end{lemma}

\begin{proof}
The reduced word $\uu_k$ ends with $\sfc_{\varepsilon(k)}$, and $Z(\sfc_{\varepsilon(k)})=c_{\varepsilon(k)}$. Since all of the simple reflections in $X_{\varepsilon(k)}$ commute with each other, we have $X_{\varepsilon(k)}\subseteq\Des(u_k)$. 

Now suppose by way of contradiction that $s\in\Des(u_k)\cap X_{\varepsilon(k+1)}$. Let $v=u_ks$, and note that $\ell(v)=\ell(u_k)-1$. Since $\uu_{k+1}=\uu_k\sfc_{\varepsilon(k+1)}$, we have $\ell(u_{k+1})=\ell(u_k)+\ell(c_{\varepsilon(k+1)})$. Let \[v'=\prod_{s'\in X_{\varepsilon(k+1)}\setminus\{s\}}s'=sc_{\varepsilon(k+1)}.\] Then $\ell(v')=\ell(c_{\varepsilon(k+1)})-1$. We have \[u_{k+1}=u_kc_{\varepsilon(k+1)}=vssv'=vv',\] which is a contradiction because \[\ell(v)+\ell(v')=\ell(u_k)-1+\ell(c_{\varepsilon(k+1)})-1=\ell(u_{k+1})-2<\ell(u_{k+1}). \qedhere\] 
\end{proof}

\begin{example}\label{exam:B}
Let $W=D_5$. Our convention is that $S=\{s_0,s_1,s_2,s_3,s_4\}$ and that the edges of $\Gamma_W$ (each of which is unlabeled) are $\{s_0,s_2\}$, $\{s_1,s_2\}$, $\{s_2,s_3\}$, $\{s_3,s_4\}$. The Coxeter number of $W$ is $h=8$. The map $\psi$ is given by $\psi(s_0)=s_{1}$, $\psi(s_1)=s_0$, and $\psi(s_k)=s_k$ for all $k\in\{2,3,4\}$. Let $c=s_0s_2s_1s_3s_4$ and $\sfc=\s_0\s_2\s_1\s_3\s_4$. The acyclic orientation of $\Gamma_W$ corresponding to $c$ is \[\includegraphics[height=1.034cm]{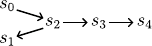}.\]  We have $X_1=\{s_0,s_1,s_3\}$ and $X_2=\{s_2,s_4\}$. The region outlined in blue in \Cref{fig:B} is $\Heap(\sfc^8)$ (drawn sideways, similarly to \Cref{fig:A}). The region shaded in yellow is $\HHH_c$. We have added the extra elements $\s_1^{(0)},\s_3^{(0)},\s_4^{(0)}$ (and appropriate edges) to the left of $\Heap(\sfc^8)$ in order to realize $\Heap(\uu_5)$ as the region enclosed by the black rectangle. The yellow shaded region that lies inside of the black rectangle is $\HHH_c^{\leq 5}$; this illustrates how we can naturally identify $\HHH_c^{\leq 5}$ with a subposet of $\Heap(\uu_5)$. 
\end{example}

\begin{figure}[ht]
  \begin{center}{\includegraphics[height=4.135cm]{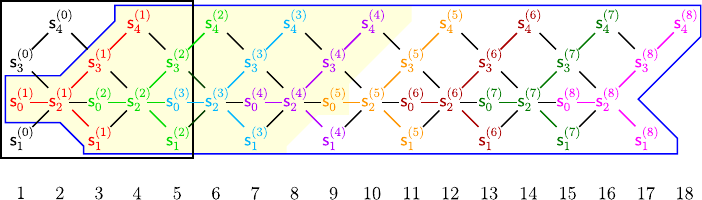}}
  \end{center}
  \caption{An illustration of \Cref{exam:B}, in which $W=D_5$. The region outlined in blue is $\Heap(\sfc^8)$. The heap $\HHH_c$ is shaded in yellow. The region enclosed by the black rectangle is $\Heap(\uu_5)$. }\label{fig:B}
\end{figure}

\begin{lemma}\label{lem:H_c_descents}
For each $k\in[h-1]$, we have $Z(\HHH_c^k)=\Des(\R(\HHH_c^{\leq k}))$. 
\end{lemma}

\begin{proof}
We have $Z(\HHH_c^k)\subseteq X_{\varepsilon(k)}$, so all of the elements of $Z(\HHH_c^k)$ commute with each other. Therefore, $\R(\HHH_c^k)=\prod_{s\in Z(\HHH_c^k)}s$. Because $\HHH_c^{\leq k-1}$ is an order ideal of $\HHH_c^{\leq k}$ whose complement (in $\HHH_c^{\leq k}$) is $\HHH_c^{k}$, we have $\R(\HHH_c^{\leq k})=\R(\HHH_c^{\leq k-1})\R(\HHH_c^k)$, so $Z(\HHH_c^k)\subseteq\Des(\R(\HHH_c^{\leq k}))$. Now choose some $r\in\Des(\R(\HHH_c^{\leq k}))$; we will show that $r\in Z(\HHH_c^k)$. 

Because $Z(\HHH_c^m)\subseteq X_{\varepsilon(m)}$ for every $m\in[k]$,
we can identify $\HHH_c^{\leq k}$ with a subposet of $\Heap(\mathsf{u}_k)$ (as illustrated in \Cref{fig:B}). We will also continue to use \cref{eq:Heap_c_sorting} to view $\HHH_c$ as an order ideal of $\Heap(\sfc^h)$ whose complement is $\widetilde\HHH_c$. We claim that $\HHH_c^{\leq k}$ is an upper set of $\Heap(\mathsf{u}_k)$. To see this, suppose $\s_i^{(j)}\in\HHH_c^{\leq k}$ and $\s_{i'}^{(j')}\in\Heap(\mathsf{u}_k)$ are such that $\s_i^{(j)}<\s_{i'}^{(j')}$ in $\Heap(\mathsf{u}_k)$. Let $s_{i''}$ be a simple reflection in $X_{\varepsilon(h-1)}$ such that $m(s_{i'},s_{i''})\neq 2$. There is a unique integer $j''\in[h]$ such that $\rank(\s_{i''}^{(j'')})=h-1$. Then $\s_{i''}^{(j'')}$ is maximal in $\Heap(\mathsf{u}_{h-1})$. Since $\s_{i'}^{(j')}$ and $\s_{i''}^{(j'')}$ are comparable (because $m(s_{i'},s_{i''})\neq 2$), we must have $\s_{i'}^{(j')}<\s_{i''}^{(j'')}$. It follows from \Cref{lem:column_h-1} that $\s_{i''}^{(j'')}\in\HHH_c^{h-1}$; in particular, $\s_{i''}^{(j'')}\in\HHH_c$. Because $\HHH_c$ is an order ideal of $\Heap(\sfc^h)$, it contains $\s_{i'}^{(j')}$. But $\rank(\s_{i'}^{(j')})\leq k$, so $\s_{i'}^{(j')}\in\HHH_c^{\leq k}$. This proves that $\HHH_c^{\leq k}$ is an upper set of $\Heap(\uu_k)$. Because $\uu_k$ is a reduced word for $u_k$, it follows that $\R(\HHH_c^{\leq k})^{-1}\leq u_k^{-1}$ (in the weak order). Consequently, $r\in\Des(\R(\HHH_c^{\leq k}))\subseteq\Des(u_k)=X_{\varepsilon(k)}$. 

Let $\rr$ be the unique element of rank $k$ in $\Heap(\uu_k)$ that represents $r$. Because $r\in\Des(\R(\HHH_c^{\leq k}))$, there must be an element $\rr'\in\HHH_c^{\leq k}$ that represents $r$. Because $\rank(\rr')\leq k=\rank(\rr)$, we have $\rr'\leq \rr$ in $\Heap(\mathsf u_k)$. But $\HHH_c^{\leq k}$ is an upper set of $\Heap(\uu_k)$, so 
$\rr\in\HHH_c^{\leq k}$. As $\rr$ has rank $k$, it must be in $\HHH_c^k$. This proves that $r\in Z(\HHH_c^k)$, as desired.   
\end{proof}

\begin{lemma}\label{lem:maximizer}
Let $c$ be a Coxeter element of $W$. We have \[\left\lvert\mathcal O_{\Camb_c}(\zz_c)\right\rvert=h.\] 
\end{lemma}

\begin{proof}
As before, let $X_1$ (resp.\ $X_2$) be the set of simple reflections $s$ such that the letters representing $s$ in $\sfc^h$ have odd (resp.\ even) rank. For $0\leq j\leq h-1$, let $v_j=\R(\HHH_c^{\leq j})$. In particular, $v_0=e$, and $v_{h-1}=\zz_c$. Suppose $k\in[h-1]$. We know that $Z(\HHH_c^k)\subseteq X_{\varepsilon(k)}$, so all of the elements of $Z(\HHH_c^k)$ commute with each other. Along with \Cref{lem:H_c_descents}, this implies that \[\wo(v_k)=\wo(\Des(\R(\HHH_c^{\leq k})))=\wo(Z(\HHH_c^k))=\prod_{s\in Z(\HHH_c^k)}s=\R(\HHH_c^k).\] We have $v_k=v_{k-1}\R(\HHH_c^{k})$, so it follows from \cref{eq:pop_weak} that \[\popdown_{\Weak(W)}(v_k)=v_{k-1}\R(\HHH_c^{k})\wo(\Des(v_k))=v_{k-1}\R(\HHH_c^{k})^2=v_{k-1}.\] Because $\HHH_c^{\leq k-1}$ is an order ideal of $\HHH_c$, we know by \Cref{rem:spine} that $v_{k-1}$ is in the spine of $\Camb_c$; in particular, it is in $\Camb_c$. Invoking \cref{eq:pop_Cambrian}, we find that \[\popdown_{\Camb_c}(v_k)=\pi_\downarrow^c(\popdown_{\Weak(W)}(v_k))=\pi_\downarrow^c(v_{k-1})=v_{k-1}.\] As this is true for all $k\in[h-1]$, we deduce that \[\mathcal O_{\Camb_c}(\zz_c)=\mathcal O_{\Camb_c}(v_{h-1})=\{v_{h-1},v_{h-2},\ldots,v_0\}. \qedhere\]   
\end{proof}

Combining \cref{thm:upper_quotient} with \cref{lem:maximizer} yields our last main result.

\begin{theorem}\label{thm:max_orbit}
    Let $W$ be a finite irreducible Coxeter group, and let $c \in W$ be a Coxeter element. The maximum size of a forward orbit under $\popdown_{\Camb_c}$ is exactly the Coxeter number $h$ of $W$. 
\end{theorem}

\subsection{Heaps and representation theory}\label{subsec:heap_rep}
We conclude this section with a brief explanation of how the heap of a $c$-sortable word $\QQ$ relates to representation theory of $KQ_c$. 

Let $\Lambda$ be an arbitrary finite-dimensional algebra. A morphism $f\colon M \rightarrow N$ in $\mods \Lambda$ is said to be \dfn{irreducible} if it is not an isomorphism and 
whenever $L \in \mods\Lambda$, $g\colon M \rightarrow L$, and $h\colon L \rightarrow N$ are such that $f = h\circ g$, we have that either $g$ or $h$ is an isomorphism.
 For $M, N \in \mods\Lambda$, we denote by $\mathrm{irr}(M,N)$ the linear subspace of $\Hom_\Lambda(M,N)$ spanned by the irreducible morphisms. The \dfn{Auslander--Reiten (AR) quiver} of $\Lambda$ is then the directed graph defined as follows.
\begin{itemize}
    \item The vertices of the AR quiver are the indecomposable objects in $\mods\Lambda$.
    \item For any two vertices $M, N$, the number of arrows $M \rightarrow N$ in the AR quiver coincides with the dimension of $\mathrm{irr}(M,N)$.
\end{itemize}

    Now let $\QQ$ be a word over $S$. In \cite[Section~9.2]{cataland}, the Hasse diagram of $\Heap(\QQ)$ is referred to as the \dfn{combinatorial AR quiver} of $\QQ$.
    To explain this, consider the special case when $\QQ = \sort_{\sfc}(\wo)$. By \cite[Theorem~9.3.1]{cataland}, there is a bijection $\Xi\colon \Heap(\sort_{\sfc}(\wo)) \rightarrow \brick KQ_c$. Under this bijection, we have $\s_i^{(j)} \covered \s_{i'}^{(j')}$ if and only if there is an irreducible morphism from $\Xi\left(\s_i^{(j)}\right)$ to $\Xi\left(\s_{i'}^{(j')}\right)$. Thus, under the bijection $\Xi$, the Hasse diagram of $\mathrm{Heap}(\sort_{\sfc}(\wo))$ is precisely the AR quiver of $KQ_c$. (In this setting, the dimension of $\mathrm{irr}(M,N)$ does not exceed 1.) In particular, the drawing of $\mathrm{Heap}(\sort_{\sfc}(\wo))$ shown as the shaded yellow part of \cref{fig:A} agrees with the standard way of embedding the AR quiver of $KQ_c$ in the plane; see, e.g., \cite[Section~3.1]{schiffler}.

    It is also shown in \cite[Theorem~9.3.1]{cataland} that, for all $s_i \in S$, the projective module $P(i)$ satisfies $P(i) = \Xi\left(\s_i^{(1)}\right)$. Because every nonzero morphism can be written as a composition of irreducible morphisms, the first copy of $Q_c = \Heap(\sfc)$ in $\mathrm{Heap}(\sort_{\sfc}(\wo))$ thus provides a visualization of the well-known fact that $\mathrm{End}_{KQ_c}\left(\bigoplus_{i = 1}^n P(i)\right) \cong KQ_c$.

\section{Future Directions} \label{sec:conclusion}

\subsection{Images} 
Consider the \dfn{linear Coxeter element} $c^{\to}=s_1s_2\cdots s_n$ of $A_n$. The Cambrian lattice $\Camb_{c^\to}$ is the $(n+1)$-st \emph{Tamari lattice}. Hong \cite{Hong2022} proved that the size of the image of $\popdown_{\Camb_{c^\to}}$ is the $n$-th \emph{Motzkin number} (i.e., the number of Motzkin paths of length $n$). In \cref{sec:Cambrian_A}, we determined the size of the image of $\popdown_{\Camb_{c^\bip}}$, where $c^\bip$ is a bipartite Coxeter element of $A_n$. Using these formulas, one can verify that $|\popdown_{\Camb_{c^\to}}(\Camb_{c^\to})|\leq|\popdown_{\Camb_{c^\bip}}(\Camb_{c^\bip})|$. Numerical evidence has led us to conjecture that the linear and bipartite Coxeter elements are, in some sense, extremal with regard to the sizes of the images of pop-stack operators.  

\begin{conjecture}
For every Coxeter element $c$ of $A_n$, we have \[|\popdown_{\Camb_{c^\to}}(\Camb_{c^\to})|\leq|\popdown_{\Camb_{c}}(\Camb_{c})|\leq |\popdown_{\Camb_{c^\bip}}(\Camb_{c^\bip})|.\]
\end{conjecture}

\subsection{Forward orbits}
Suppose $L$ is a finite lattice, and let $\upsilon_L=\max_{x\in L}|\mathcal O_L(x)|$. Let \[\Upsilon_L=\{x\in L\mid|\mathcal O_L(x)|=\upsilon_L\}\] be the set of elements of $L$ whose forward orbits under $\popdown_L$ attain the maximum possible size. 

Let $c$ be a Coxeter element of a finite irreducible Coxeter group $W$. We saw in \cref{thm:max_orbit} that $\upsilon_{\Camb_c}$ is the Coxeter number $h$, and we constructed an explicit element ${\bf z}_c\in\Upsilon_{\Camb_c}$. However, we said nothing about the other elements $\Upsilon_{\Camb_c}$. In the case when $W=A_n$ and $c$ is the linear Coxeter element $c^\to$ (i.e., $\Camb_c$ is the $(n+1)$-st Tamari lattice), it is known that $|\Upsilon_{\Camb_c}|$ is the $(n-1)$-st Catalan number \cite{DefantMeeting}. It would be interesting to understand $\Upsilon_{\Camb_c}$ for other choices of $c$. In particular, we have the following conjecture. 

\begin{conjecture}
Assume $n\geq 3$, and let $c^\bip$ be the bipartite Coxeter element of type $A_n$ given by \[c^\bip=\prod_{\substack{i\in[n]\\ i\text{ odd}}}s_i\prod_{\substack{j\in[n]\\ j\text{ even}}}s_j.\] The set of elements of $\Camb_{c^\bip}$ whose forward orbits under the pop-stack operator have size $h$ is given by \[\Upsilon_{\Camb_{c^\bip}}=\begin{cases} \{{\bf z}_c\} & \mbox{if }n\mbox{ is odd}; \\ \{{\bf z}_c,{\bf z}_cs_1\} & \mbox{if }n\mbox{ is even}.\end{cases}\]
\end{conjecture} 

The original use of the term \emph{pop-stack} comes from the setting where $L$ is the weak order on $A_n$; in this case, Ungar proved that $\max_{x\in\Weak(A_n)}|\mathcal O_{\Weak(A_n)}(x)|$ is $n+1$ (which is the Coxeter number of $A_n$). 
\begin{question}
What can be said about $|\Upsilon_{\Weak(A_n)}|$? 
\end{question}

Defant \cite{DefantPopCoxeter} proved that if $W$ is a finite irreducible Coxeter group with Coxeter number $h$, then $\max_{x\in W}|\mathcal O_{\Weak(W)}(x)|=h$. In \cref{thm:upper_quotient}, we found that $\max_{x\in L}|\mathcal O_{L}(x)|\leq h$ whenever $L$ is a lattice quotient of $\Weak(W)$, and we saw in \cref{thm:max_orbit} that this inequality is an equality whenever $L$ is a Cambrian lattice. We are naturally led to ask the following questions. 

\begin{question}
Let $W$ be a finite irreducible Coxeter group with Coxeter number $h$. For which lattice quotients $L$ of $\Weak(W)$ is it the case that $\max_{x\in L}|\mathcal O_{L}(x)|=h$? 
\end{question}

\begin{question}
Let $L'$ be a lattice quotient of a finite lattice $L$. Is it necessarily the case that \[\max_{x'\in L'}|\mathcal O_{L'}(x')|\leq\max_{x\in L}|\mathcal O_L(x)|?\] 
\end{question}

\subsection{Generalizations}
It would be interesting to see how much of our work on Cambrian lattices can be extended to more general families of lattices. For example, it could be interesting to study the pop-stack operators on \emph{$m$-Cambrian lattices}, which were introduced by Stump, Thomas, and Williams \cite{cataland}.

\appendix\section{Proof of \cref{thm:mutation_summary}}\label{sec:appendix}

In this appendix, we prove \cref{thm:mutation_summary}. More precisely, we prove only the first half of the theorem, as the second half is just the dual. We follow arguments similar to those appearing in \cite[Section~3]{HI_pairwise} and \cite[Section~7.2]{KY}. 
The main difference is that we work almost exclusively in the category $\mods\Lambda$ (with the exception of a portion of the proof of \cref{lem:mutation_1} below), while the referenced papers work primarily in the bounded derived category.

 In the arguments that follow, we recall that for any wide subcategory $\W \subseteq \mods\Lambda$ and any $M, N \in \W$, we have $\Hom_\Lambda(M,N) = \Hom_\W(M,N)$ and $\Ext^1_\Lambda(M,N) = \Ext^1_\W(M,N)$. On the other hand, we may have that $\Ext^m_\Lambda(M,N) \supsetneq \Ext^m_\W(M,N)$ for $m > 1$.

\begin{lemma}\label{lem:mutation_1} Let $(\X,\Y) \in \sbp\Lambda$ be SM compatible. Let $\X' \subseteq \X$, $X \in \X \setminus \X'$, and $Z \in \Filt(\X')$. 
        \begin{enumerate}[(1)]
            \item\label{lem:mutation_1_1} The map $\Hom_\Lambda(X'_X,Z) \rightarrow \Ext^1_\Lambda(X,Z)$ induced by the short exact sequence $\eta_{X,\X'}$ from \cref{eqn:mutation_1} is bijective.
            \item\label{lem:mutation_1_2} The map $\Ext^1_\Lambda(X'_X,Z) \rightarrow \Ext^2_\Lambda(X,Z)$ induced by the short exact sequence $\eta_{X,\X'}$ from \cref{eqn:mutation_1} is injective.
        \end{enumerate}
\end{lemma}

\begin{proof}\ 

\ref{lem:mutation_1_1} Since $\Filt(\X)$ is a wide subcategory, we have that $\Ext^1_\Lambda(X,Z) \cong \Hom_\Lambda(\Omega_\X X,Z)$. We have $\ker(\gamma_{X,\X'}) \in \Filt(\X)$ since $\Filt(\X)$ is a wide subcategory. Now, $\Filt(\X')$ is a Serre subcategory of $\Filt(\X)$, so in particular, there is a torsion pair $(\lperp{\X'},\Filt(\X'))$ in $\Filt(\X)$. This implies that $\ker(\gamma_{X,\X'}) \in \lperp{\X'}$. From the exact sequence
    $$0 \rightarrow \Hom_\Lambda(X'_X,Z) \rightarrow \Hom_\Lambda(\Omega_\X X,Z) \rightarrow \Hom_\Lambda(\ker(\gamma_{X,\X'},Z),$$we therefore conclude that the induced map $\Hom_\Lambda(X'_X,Z) \rightarrow \Hom_\Lambda(\Omega_\X X,Z) \cong \Ext^1_\Lambda(X,Z)$ is a bijection.

    \ref{lem:mutation_1_2} We first claim that $\Hom_\Lambda(\Omega^2_\X X,Z) \subseteq \Ext^2_\Lambda(X,Z)$. We know from \cref{thm:torsion_smc} that there exists $\X_u \in \sbrick\Lambda$ such that $(\X,\X_u)$ is a 2-term simple-minded collection. Then, as in \cite[Lemma~7.8]{KY}, there exist a finite-dimensional algebra $\Gamma$ and a triangle functor \[{\mathsf{real}\colon \Db(\mods\Gamma) \rightarrow \Db(\mods\Lambda)}\] such that the following hold:
    \begin{itemize}
        \item $\mods\Gamma = \Filt(\X \oplus \X_u[1])$.
        \item $\mathsf{real}$ acts as the identity on $\mods\Gamma$.
        \item For all $M, N \in \mods\Gamma$, the induced map $\Ext^1_\Gamma(M,N) \rightarrow \Hom_{\Db(\mods\Lambda)}(M,N[1])$ is a bijection, and the induced map $\Ext^2_\Gamma(M,N) \rightarrow \Hom_{\Db(\mods\Lambda)}(M,N[2])$ is an injection.
    \end{itemize}
    Now, $\Filt(\X)$ is a Serre subcategory of $\mods\Gamma$ and a wide subcategory of $\mods\Lambda$. Thus, taking $M = X$ and $N = Z$ above, we find that $\mathsf{real}$ induces an inclusion $\Ext^2_\Gamma(X,Z) \subseteq \Ext^2_\Lambda(X,Z)$. To prove the claim, it therefore suffices to show that $(\Hom_\Lambda(\Omega^2_\X X,Z) =) \Hom_\Gamma(\Omega^2_\X,Z) \subseteq \Ext^1_\Gamma(X,Z)$.

    Consider the following commutative diagram, where the first row consists of projective covers and syzygies in $\mods \Gamma$ and the second row consists of projective covers and syzygies in $\Filt(\X)$.
    $$
    \begin{tikzcd}
        \Omega^2 X \arrow[r,hookrightarrow]\arrow[d,dashed,"g"] & P_X^1 \arrow[r,two heads]\arrow[d,dashed] & \Omega X \arrow[r,hookrightarrow]\arrow[d,"f",dashed] & P_{X}^0 \arrow[r,two heads]\arrow[d,dashed] & X \arrow[d,equals]\\
         \Omega^2_\X X \arrow[r,hookrightarrow] & P_{\X,X}^1 \arrow[r,twoheadrightarrow] & \Omega_\X X \arrow[r,hookrightarrow] & P_{\X,X}^0 \arrow[r,two heads] & X.
    \end{tikzcd}
$$
Since $\Filt(\X)$ is a Serre subcategory of $\mods\Gamma$, every module in $\Filt(\X)$ has the same top in both $\Filt(\X)$ and in $\mods\Gamma$. Thus, the induced map $P_X^0 \rightarrow P_{\X,X}^0$, and so also $f$, must be surjective. Now for $S \in \X$, we get an induced exact sequence
$$0 \rightarrow \Hom_\Gamma(P_{\X,X}^0, S) \rightarrow \Hom_\Gamma(P_X^0,S) \rightarrow \Hom_\Gamma(\ker f,S) \rightarrow \Ext^1_\Gamma(P_{\X,X}^0,S) = 0,$$
where the last term is 0 because $P_{\X,X}^0$ is projective in $\Filt(\X)$. Since $P_X^0$ and $P_{\X,X}^0$ have the same top, this implies that $\Hom_\Gamma(\ker f,S) = 0$, and hence also that $\ker f \in \lperp{\X}$. Moreover, we see that the top of $\Omega_\X X$ is a direct summand of the top of $\Omega X$, and thus that the induced map $P_X^1 \twoheadrightarrow P_{\X,X}^1$ is surjective. From the snake lemma, this means there is a surjective map $\ker(f) \twoheadrightarrow \coker(g)$. On the other hand, $\coker(g) \in \Filt(\X)$ since $\Filt(\X)$ is a Serre subcategory of $\mods\Gamma$. It follows that $g$ is surjective. In particular, the induced map $\Hom_\Gamma(\Omega_\X^2 X,Z) \rightarrow \Hom_\Gamma(\Omega^2 X,Z) = \Ext^1_\Gamma(X,Z)$
is injective. This proves the claim. 

Now consider the following commutative diagram of projective covers and syzygies in $\Filt(\X)$:
$$
    \begin{tikzcd}
       \Omega_\X^2 X \arrow[r,hookrightarrow]\arrow[d,dashed,"\alpha_{X,\X'}"] & P^1_{\X,X} \arrow[r,two heads]\arrow[d,dashed] & \Omega_\X X \arrow[d,two heads,"\gamma_{X,\X'}"]\\
        \Omega_\X X'_X \arrow[r,hookrightarrow] & P_{\X,X'_X}^0 \arrow[r,two heads] & X'_X.
    \end{tikzcd}
$$
The induced map $P_{\X,X}^1 \rightarrow P_{\X,X'_X}^0$ is a (split) epimorphism. Consequently, we have a surjection $\ker(\gamma_{X,\X'}) \twoheadrightarrow \coker(\alpha_{X,\X'})$. On the other hand, we have that $\ker(\gamma_{X,\X'}) \in \lperp{\X'}$, as observed in the proof of \ref{lem:mutation_1_1}. Similarly, we have $\coker(\alpha_{X,\X'}) \in \Filt(\X')$ since $\Filt(\X')$ is a Serre subcategory of $\Filt(\X)$. We conclude that $\coker(\alpha_{X,\X'}) = 0$. In particular, the induced map $$\Ext^1_{\Filt(\X)}(X'_X,Z) = \Hom_\Gamma(\Omega_\X X'_X,Z) \rightarrow \Hom_\Gamma(\Omega^2_\X X,Z) = \Ext^2_{\Filt(\X)}(X,Z)$$
is injective. Together with the claim, the fact that $\Ext^1_\Lambda(X'_X,Z) = \Ext^1_{\Filt(\X)}(X'_X,Z)$ then implies the result.
\end{proof}

\begin{lemma}\label{lem:mutation2} Let $(\X,\Y) \in \sbp\Lambda$ be SM compatible. Let $\X' \subseteq \X$, $Y \in \Y$, and $Z \in \Filt(\X')$. 
        \begin{enumerate}[(1)]
            \item\label{lem:mutation_2_1} The map $\Hom_\Lambda(X'_Y,Z) \rightarrow \Hom_\Lambda(Y,Z)$ induced by $g_{Y,\X'}$ is bijective.
            \item\label{lem:mutation_2_2} The map $\Ext^1_\Lambda(X'_Y,Z) \rightarrow \Ext^1_\Lambda(Y,Z)$ induced by $g_{Y,\X'}$ is injective.
        \end{enumerate}
\end{lemma}

\begin{proof}\ 

    \ref{lem:mutation_2_1} Let $f\colon X'_Y \rightarrow Z$ be such that $f\circ g_{Y,\X'} = 0$. Equivalently, we have $\image(g_{Y,\X'}) \subseteq \ker(f)$. Note that $\ker(f) \in \Filt(\X')$ since $\Filt(\X')$ is a wide subcategory. Since $g_{Y,\X'}$ is a right minimal $\Filt (\X')$-approximation, it follows that $Y_\X' = \ker f$; i.e., $f = 0$. This shows that the induced map is injective; the fact that it is surjective comes from the definition of an approximation.

    \ref{lem:mutation_2_2} Let $\eta \in \Ext^1_\Lambda(X'_Y,Z)$ be such that $\eta \circ g_{Y,\X'} = 0$. Then $\eta$ is a short exact sequence of the form
    $$
    \begin{tikzcd}
        \eta: & Z \arrow[r,hookrightarrow] & E \arrow[r,"q",two heads] & X'_Y,
    \end{tikzcd}
    $$
    so this means that there exists $h\colon Y \rightarrow E$ such that $g_{Y,\X'} = q \circ h$. Now, $E \in \Filt(\X')$, so there exists $h'\colon X'_Y \rightarrow E$ such that $h = h' \circ g_{Y,\X'}$. We therefore have
    $$q \circ h = q \circ h' \circ g_{Y,\X'} = \mathrm{Id}_{X'_Y}\circ g_{Y,\X'}.$$
    It then follows from \ref{lem:mutation_2_1} that $q \circ h' = \mathrm{Id}_{X'_Y}$; i.e., the exact sequence $\eta$ is split.
\end{proof}

\begin{proposition}[\cref{thm:mutation_summary}, \ref{thm:mutation_1a}] \label{prop:mutation}
    Let $(\X,\Y) \in \sbp\Lambda$ be SM compatible, and let $\X'\subseteq \X$. Then $\mu_{\X'}(\X,\Y)$ is a semibrick pair.
\end{proposition}

\begin{proof}
    The proof is similar to that of \cite[Proposition~3.5]{HI_pairwise}. For readability, let $\mu_{\X'}(\X,\Y)_d = \X_1$ and $\mu_{\X'}(\X,\Y)_u = \Y_1$.
\medskip

    \noindent {\bf Claim 1:} The set $\X_2$ is a semibrick.

\medskip
    
    Let $Z_1, Z_2 \in \X_2$. First suppose that $Z_1 \in \X'$. If $Z_2 \in \X'$, then there is nothing to show since $\X' \subseteq \X$ is a semibrick. If $Z_2 = \ker(g_{Y,\X'})$ for some $Y \in \Y$ such that $g_{Y,\X'}$ is surjective, then $\Hom_\Lambda(Z_1,Z_2) = 0$ because $Z_2 \subseteq Y$ and $\Hom_\Lambda(Z_1,Y) = 0$.

    Now suppose that there exists $Y_1 \in \Y$ such that $g_{Y_1,\X'}$ is surjective and $Z_1 = \ker(g_{Y_1,\X'}).$  We first consider the case $Z_2 \in \X'$. Then we have a long exact sequence
    $$0 \rightarrow \Hom_\Lambda(X'_Y,Z_2) \rightarrow \Hom_\Lambda(Y,Z_2) \xrightarrow{0} \Hom_\Lambda(Z,Z_2)\xrightarrow{0}$$$$\xrightarrow{0} \Ext^1_\Lambda(X'_Y,Z_2) \rightarrow \Ext^1_\Lambda(Y,Z_2),$$
    where we know the indicated maps are 0 by \cref{lem:mutation2}.

    It remains to consider the case where there exists $Y_2 \in \Y$ such that $g_{Y_2,\X'}$ is surjective and $Z_2 = \ker(g_{Y_2,\X'})$. We then have a long exact sequence
    $$0 \rightarrow \Hom_\Lambda(Z_1,Z_2) \rightarrow \Hom_\Lambda(Z_1,Y_2) \rightarrow \Hom_\Lambda(Z_1,X'_{Y_2}) = 0,$$
    where the last term is 0 by the case considered above. Similarly, we have a long exact sequence
    $$0 = \Hom_\Lambda(Y_1,Y_2) \rightarrow \Hom_\Lambda(Z_1,Y_2) \rightarrow \Ext^1_\Lambda(X'_{Y_1},Y_2).$$
    Together, these sequences show that $\Hom_\Lambda(Z_1,Z_2) \cong \Hom(Y_1,Y_2)$. In particular, setting $Z_1 = Z_2$, we conclude that $Z_1$ is a brick. This proves the claim.

    \medskip

    \noindent {\bf Claim 2:} We have $\Hom_\Lambda(Z,X_1) = 0 = \Ext^1_\Lambda(Z,X_1)$ for all $X_1 \in \X'$ and $Z \in \X_1$. 
    
    \medskip
    
    Again, there are two possibilities. Suppose first that there exists $X_2 \in \X \setminus \X'$ such that ${Z = E_{X_2,\X'}}$. Then we have a long exact sequence
    $$0 = \Hom_\Lambda(X_2,X_1) \rightarrow \Hom_\Lambda(Z,X_1) \xrightarrow{0} \Hom_\Lambda(X'_{X_2},X_1) \rightarrow$$$$\rightarrow \Ext^1_\Lambda(X_2,X_1) \xrightarrow{0} \Ext^1_\Lambda(Z,X_1) \xrightarrow{0} \Ext^1_\Lambda(X'_{X_2},X_1) \rightarrow \Ext^2_\Lambda(X_2,X_1),$$
    where we know the indicated maps are 0 by \cref{lem:mutation_1}. This proves the claim in this case.

    The other case to consider is that in which there exists $Y \in \Y$ such that $g_{Y,\X'}$ is injective and $Z = \coker(g_{Y,\X'})$. In this case, we have a long exact sequence
    $$0 \rightarrow \Hom_\Lambda(Z,X_1) \xrightarrow{0} \Hom_\Lambda(X'_Y,X_1) \rightarrow \Hom_\Lambda(Y,X_1) \xrightarrow{0}$$$$\xrightarrow{0} \Ext^1_\Lambda(Z,X_1) \xrightarrow{0} \Ext^1_\Lambda(X'_Y,X_1) \rightarrow \Ext^1_\Lambda(Y,X_1),$$
    where again we know the indicated maps are 0 by \cref{lem:mutation2}. This proves the claim.

    \medskip

    \noindent {\bf Claim 3:} We have $\Hom_\Lambda(Z,\ker(g_{Y_1,\X'})) = 0 = \Ext^1_\Lambda(Z,\ker(g_{Y_1,\X'}))$ for all $Z \in \X_1$ and $Y_1 \in \Y$ such that $g_{Y_1,\X'}$ is surjective. 
    
    \medskip
    
    Again, there are two possibilities. Suppose first that there exists $X_2 \in \X \setminus \X'$ such that ${Z = E_{X_2,\X'}}$. Then $\Hom_\Lambda(Z,\ker(g_{Y_1,\X'})) = 0$ because $Z \in \Filt(\X)$. Moreover, we have a long exact sequence
    $$0 = \Hom_\Lambda(Z,\X'_{Y_1}) \rightarrow \Ext^1_\Lambda(Z, \ker(g_{Y_1,\X'})) \rightarrow \Ext^1_\Lambda(Z,Y_1) = 0,$$
    where the first term is 0 by Claim 2 and the last term is 0 because $Z \in \Filt(\X)$. This proves the claim in this case.

    The other case to consider is that in which there exists $Y_2 \in \Y$ such that $g_{Y_2,\X'}$ is injective and $Z = \coker(g_{Y_2,\X'})$. The fact that $\Hom_\Lambda(Z,\ker(g_{Y_1,\X'})) = 0$ then follows from the fact that $Y_1 \in \X^\perp$. Moreover, we have an exact sequence $$0 = \Hom_\Lambda(Z,X'_{Y_1}) \rightarrow \Ext^1_\Lambda(Z,\ker(g_{Y_1,\X'})) \rightarrow $$$$\rightarrow\Ext^1_\Lambda(Z,Y_1) \rightarrow \Ext^1_\Lambda(Z,X'_{Y_1}) = 0,$$
    where the first and last terms are 0 by Claim 2. Likewise, we have an exact sequence
    $$0 = \Hom_\Lambda(Y_2,Y_1) \rightarrow \Ext^1_\Lambda(Z,Y_1) \rightarrow \Ext^1_\Lambda(X'_{Y_2},Y_1) = 0,$$
    where the first and last terms are 0 since $Y_1 \in Y_2^\perp \cap \X^{\perp_{0,1}}$. Together, these two exact sequences yield that $\Ext^1_\Lambda(Z,\ker(g_{Y_1,\X'})) = 0$, as desired.

    \medskip

    \noindent {\bf Claim 4:} The set $\X_1$ is a semibrick. 

    \medskip
    
    Let $Z_1, Z_2 \in \X_1$. First suppose that there exist $X_1, X_2 \in \X \setminus \X'$ such that $Z_1 = E_{X_1,\X'}$ and $Z_2 = E_{X_2,\X'}$. Then there is an exact sequence
    $$0 = \Hom_\Lambda(Z_1,X'_{X_2}) \rightarrow \Hom_\Lambda(Z_1,Z_2) \rightarrow \Hom_\Lambda(Z_1,X_2) \rightarrow \Ext^1_\Lambda(Z_1,X'_{X_2}) = 0,$$
    where the first and last terms are 0 by Claim 2. We also have a long exact sequence
    $$0 \rightarrow \Hom(X_1,X_2) \rightarrow \Hom(Z_1,X_2) \rightarrow \Hom(X'_{X_1},X_2) = 0.$$
    Together, these sequences show that $\Hom_\Lambda(Z_1,Z_2) \cong \Hom_\Lambda(X_1,X_2)$. In particular, by setting $Z_1 = Z_2$, we conclude that $Z_1$ is a brick. This proves the claim in this case.

    Now suppose there exists $Y_1 \in \Y$ such that $g_{Y_1,\X'}$ is injective and $Z_1 = \coker(g_{Y_1,\X'})$ and that there exists $X_2 \in \X \setminus \X'$ such that $Z_2 = E_{X_2,\X'}$. Then $\Hom_\Lambda(Z_1,X_2) = 0$ since $X_2 \in (\X')^\perp$. We
    then have an exact sequence
    $$0 = \Hom_\Lambda(Z_1,X'_{X_2}) \rightarrow \Hom_\Lambda(Z_1,Z_2) \rightarrow \Hom_\Lambda(Z_1,X_2) = 0,$$
    where the first term is 0 by Claim 2. Similarly, we have an exact sequence
    $$0 = \Hom_\Lambda(Z_2,X'_{Y_1}) \rightarrow \Hom_\Lambda(Z_2,Z_1) \rightarrow \Ext^1_\Lambda(Z_2,Y_1) = 0,$$
    where the first term is 0 by Claim 2 and the last term is 0 because $Y \in \X^{\perp_{0,1}}$. We conclude that $Z_1 \in Z_2^\perp \cap \lperp{Z_2}$.

    It remains to consider the case where there exist $Y_1, Y_2 \in \Y$ such that $g_{Y_1,\X'}$ and $g_{Y_2,\X'}$ are injective, $Z_1 = \coker(g_{Y_1,\X'})$, and $Z_2 = \coker(g_{Y_2,\X'})$.  As in earlier cases, we have an exact sequence
    $$0 = \Hom_\Lambda(Z_1,X'_{Y_2}) \rightarrow \Hom_\Lambda(Z_1,Z_2) \rightarrow \Hom_\Lambda(Z_1,Y_2[1]) \rightarrow \Ext^1_\Lambda(Z_1,X'_{Y_2}) = 0,$$
    where the first and last terms are 0 by Claim 2. Similarly, we have an exact sequence
    $$0 = \Hom(X'_{Y_1},Y_2) \rightarrow \Hom_\Lambda(Y_1,Y_2) \rightarrow \Ext^1_\Lambda(Z_1,Y_2) \rightarrow \Ext^1_\Lambda(X'_{Y_1},Y_2) = 0,$$
    where the first and last terms are 0 since $Y \in \X^{\perp_{0,1}}$. We again conclude that $\Hom_\Lambda(Z_1,Z_2) \cong \Hom_\Lambda(Y_1,Y_2)$. This concludes the proof.
\end{proof}

\begin{proposition}[\cref{thm:mutation_summary},\ref{thm:mutation_1b}--\ref{thm:mutation_1e}]\label{prop:mutation_completable}
    Let $(\X,\Y) \in \sbp\Lambda$ be SM compatible, and let $\X' \subseteq \X$. Write $\mu_{\X'}(\X,\Y) = (\X_1,\Y_1)$. 
    \begin{enumerate}[(1)]
        \item\label{prop:mutation_completable_1} If $(\X'_1,\Y'_1)$ is an SM compatible semibrick pair such that $\X_1 \subseteq \X'_1$ and $\Y_1 \subseteq \Y'_1$, then we have $\X \subseteq \mu_{\X'}(\X'_1,\Y'_1)_d$ and $\Y \subseteq \mu_{\X'}(\X'_1,\Y'_1)_u$.
        \item\label{prop:mutation_completable_2} $(\X,\Y) \in \smc\Lambda$ if and only if $(\X_1,\Y_1) \in \smc\Lambda$.
        \item\label{prop:mutation_completable_3} $(\X,\Y)$ is completable if and only if $(\X_1,\Y_1)$ is completable.
        \item\label{prop:mutation_completable_4} If $(\X,\Y) \in \smc\Lambda$, then $\tor(\mu_{\X'}(\X,\Y)) = \tor(\X,\Y) \cap \lperp{\X'}$.
    \end{enumerate}
\end{proposition}

\begin{proof}\ 

    \ref{prop:mutation_completable_1} First note that $\X' \subseteq \Y'_1$ and therefore that $\X' \subseteq \mu_{\X'}(\X'_1,\Y'_1)_d$ by construction.
    
    Now let $X \in \X \setminus \X'$ so that $E_{X,\X'} \in \Y_1'$. Then \cref{lem:mutation_1} tells us that the map $X'_X \hookrightarrow E_{X,\X'}$ is a minimal right $\Filt(\X')$-approximation of $E_{X,\X'}$ that is injective and has cokernel $X$. It follows that $X \in \mu_{\X'}(\X'_1,\Y'_1)_d$.

    It remains to consider $Y \in \Y$. Suppose first that $g_{Y,\X'}$ is injective. For any ${Z \in \Filt(\X')}$, the quotient map $q\colon X'_Y \rightarrow \coker(g_{Y,\X'})$ induces a bijection $\Hom_\Lambda(Z,X'_{Y_2}) \rightarrow \Hom_\Lambda(Z,\coker(g_{Y,\X'}))$. Thus, $q$ is a minimal right $\Filt(\X')$-approximation of $\coker(g_{Y,\X'})$ that is surjective and has kernel $Y$. We conclude that $Y \in \mu_{\X'}(\X'_1,\Y'_1)_u$.

    It remains to consider the case when $g_{Y,\X'}$ is surjective. Let $Y' = \ker(g_{Y,\X'})$. We consider the commutative diagram
    $$
    \begin{tikzcd}
        Y' \arrow[r,hookrightarrow]\arrow[d,equals] & Y \arrow[r,two heads,"g_{Y,\X'}"]\arrow[d] & X'_Y \arrow[d,"\gamma"]\\
        Y' \arrow[r,hookrightarrow] & I_{\Y'_1,Y'} \arrow[r,two heads]  &\Sigma_{\Y'_1}Y',
    \end{tikzcd}
    $$
    where the bottom row consists of the injective envelope and first cosyzygy of $Y'$ in $\Filt(\Y'_1)$. The vertical maps are induced by the fact that the top row is an element of $\Ext^1_\Lambda(X'_Y,Y')$, so in particular, the right square is a pullback. Now for $Z \in \Filt(\X')$, we have an exact sequence
    $$0 = \Hom_\Lambda(Z,Y) \rightarrow \Hom_\Lambda(Z,X'_Y) \rightarrow \Ext^1_\Lambda(Z,Y') \rightarrow \Ext^1_\Lambda(Z,Y) = 0,$$
    where the first and last terms are 0 since $Y \in \X^{\perp_{0,1}}$. Moreover, we have a canonical isomorphism $\Ext^1_\Lambda(Z,Y') \cong \Hom_\Lambda(Z,\Sigma_{\Y'_1} Y')$ since $\Filt(\X') \subseteq \Filt(\Y')$ is a wide subcategory of $\mods\Lambda$. We conclude that $\gamma$ induces a bijection $\Hom_\Lambda(Z,X'_Y) \cong \Hom_\Lambda(Z,\Sigma_{\Y'_1}Y')$, so it is a minimal right $\Filt(\X')$-approximation. Thus, the top row of our diagram coincides with the short exact sequence $\eta_{\Y'_1,Y'}$ described above \cref{def:mutation}. We conclude that $Y \in \mu_{\X'}(\X'_1,\Y'_1)_u$.

    \ref{prop:mutation_completable_2} Let $X \in \X \setminus \X'$. Then the short exact sequence $\eta_{X,\X'}$ corresponds to a triangle
    $$X[-1] \rightarrow X'_X \rightarrow E_{X,\X'} \rightarrow X$$
    in $\Db(\mods\Lambda)$. Thus, any triangulated subcategory that is closed under direct summands and contains (some shift of) $\X'$ will contain $X$ if and only if it contains $E_{X,\X'}$. Similarly, for $Y \in \Y$, we have a triangle
    $$Y \xrightarrow{g_{Y,\X'}} X'_Y \rightarrow Z \rightarrow Y[1],$$
    where either $Z = \coker(g_{Y,\X'})$ or $Z = \ker(g_{Y,\X'})[1]$. Thus again, any triangulated subcategory that is closed under direct summands and contains $\X'$ will contain $Y$ if and only if it contains $Z$. This proves the result.

    \ref{prop:mutation_completable_3} This is an immediate consequence of \ref{prop:mutation_completable_1} and \ref{prop:mutation_completable_3}.

    \ref{prop:mutation_completable_4} Recall that $\tor(\mu_{\X'}(\X,\Y)) = \lperp{\Y_1}$ and $\tor(\X,\Y) = \lperp{\Y}$ by \cref{thm:torsion_smc}.

    The fact that $\lperp{Y_1} \subseteq \lperp{\X'}$ follows from the definitions. Hence, let $M \in \lperp{\X'}$. Then for all $Y \in \Y$ such that $g_{Y,\X'}$ is injective, we necessarily have that $M \in \lperp{Y}$. It therefore suffices to show that for $Y \in \Y$ such that $g_{Y,\X'}$ is surjective, we have that $M \in \lperp{Y}$ if and only if $M \in \lperp{\ker(g_{Y,\X'})}$. The inclusion $\lperp{Y} \subseteq\lperp{\ker(g_{Y,\X'})}$ holds in general, so suppose that $M \in \lperp{\ker(g_{Y,\X'})}$. Then the fact that $M \in \lperp{\X'}$ implies that $M \in \lperp{Y}$ as well. This concludes the proof. 
\end{proof}

\section*{Acknowledgments}
Colin Defant was supported by the National Science Foundation under Award No.\ 2201907 and by a Benjamin Peirce Fellowship at Harvard University. Eric Hanson was partially supported by Canada Research Chairs CRC-2021-00120 and by NSERC Discovery Grants RGPIN-2022-03960 and RGPIN/04465-2019. A portion of this work was completed while Eric Hanson was a postdoctoral fellow at l'Universit\'e du Qu\'ebec \`a Montr\'eal and l'Universit\'e de Sherbrooke. We are grateful to Nathan Williams for helpful conversations.

\newcommand{\etalchar}[1]{$^{#1}$}
\providecommand{\bysame}{\leavevmode\hbox to3em{\hrulefill}\thinspace}
\providecommand{\MR}{\relax\ifhmode\unskip\space\fi MR }


\begin{thebibliography}{ABB{\etalchar{+}}19}

\bibitem[AM17]{AiharaMizuno}
A.~Aihara and Y.~Mizuno, \emph{Classifying tilting complexes over preprojective
  algebras of {D}ynkin type}, Algebra Number Theory \textbf{11} (2017), no.~6,
  1287--1315.

  

\bibitem[AN09]{al-nofayee}
S.~Al-Nofayee, \emph{Simple objects in the heart of a {$t$}-structure}, J. Pure
  Appl. Algebra \textbf{213} (2009), no.~1, 54--59.

\bibitem[AHI{\etalchar{+}}]{AHIKM}
T.~Aoki, A.~Higashitani, O.~Iyama, R.~Kase, and Y.~Mizuno, \emph{Fans and
  polytopes in tilting theory {I}: Foundations}, arXiv:2203.15213.

\bibitem[Asa20]{asai}
S.~Asai, \emph{Semibricks}, Int. Math. Res. Not. IMRN \textbf{2020} (2020),
  no.~16, 4993--5054.

\bibitem[AP22]{AP}
S.~Asai and C.~Pfeifer, \emph{Wide subcategories and lattices of torsion
  classes}, Algebr. Represent. Theory \textbf{25} (2022), 1611--1629.

\bibitem[ABB{\etalchar{+}}19]{ABBHL}
A.~Asinowski, C.~Banderier, S.~Biley, B.~Hackl, and S.~Linusson,
  \emph{Pop-stack sorting and its image: permutations with overlapping runs},
  Acta Math. Univ. Comenianae \textbf{88} (2019), 395--402.

\bibitem[ASS06]{ASS}
I.~Assem, D.~Simson, and A.~Skowroński, \emph{Elements of the representation
  theory of associative algebras, volume 1: techniques of representation
  theory}, London Math. Soc. Stud. Texts, vol.~65, Cambridge University Press,
  Cambridge, 2006.

\bibitem[ARS95]{ARS}
M.~Auslander, I.~Reiten, and S.~O. Smalø, \emph{Representation theory of
  {A}rtin algebras}, Cambridge Stud. Adv. Math., vol.~36, Cambridge University
  Press, Cambridge, 1995.

\bibitem[Bar19]{BarnardCJC}
E.~Barnard, \emph{The canonical join complex}, Electron. J. Combin. \textbf{26}
  (2019).

\bibitem[BCZ19]{BCZ}
E.~Barnard, A.~Carroll, and S.~Zhu, \emph{Minimal inclusions of torsion
  classes}, Algebr. Comb. \textbf{2} (2019), no.~5, 879--901.

\bibitem[BH]{BaH}
E.~Barnard and E.~J. Hanson, \emph{Exceptional sequences in semidistributive
  lattices and the poset topology of wide subcategories}, arXiv:2209.11734.

\bibitem[BH22]{BaH_preproj}
E.~Barnard and E.~J. Hanson, \emph{Pairwise compatibility for 2-simple minded collections {II}:
  preprojective algebras and full rank semibirck pairs}, Ann. Comb. \textbf{26}
  (2022), 803--855.

\bibitem[BR18]{bicat}
E. Barnard and N. Reading, \emph{Coxeter-bi{C}atalan combinatorics}, J.
  Algebraic Combin. \textbf{47} (2018), no.~2, 241--300.

\bibitem[BTZ21]{BTZ}
E.~Barnard, G.~Todorov, and S.~Zhu, \emph{Dynamical combinatorics and torsion
  classes}, J. Pure Appl. Algebra \textbf{225} (2021).

\bibitem[BB05]{BjornerBrenti}
A.~Bj\"orner and F.~Brenti, \emph{Combinatorics of coxeter groups}, Graduate
  Texts in Mathematics, vol. 231, Springer, Berlin, Heidelberg, 2005.

\bibitem[Bou02]{Bourbaki}
N.~Bourbaki, \emph{Lie groups and {L}ie algebras} Chapters 4-6, Springer--Verlag, Berlin, 
Copyright Information
Springer-Verlag Berlin, Heidelberg, 2002.

\bibitem[BY13]{BY}
T.~Br\"ustle and D.~Yang, \emph{Oriented exchange graphs, \textnormal{In (D. J. Benson, H. Krause, and A. Skowro\'nski, eds.) Advances
  in Representation Theory of Algebras}}, EMS Series of Congress Reports,
  vol.~9, 2013.

\bibitem[CS]{ChoiSun}
Y.~Choi and N.~Sun, \emph{The image of the pop operator on various lattices},
  To appear in \emph{Adv. Appl. Math.}

\bibitem[CG19]{ClaessonPop}
A.~Claesson and B.~\'A. Gu{\dh}mundsson, \emph{Enumerating permutations
  sortable by $k$ passes through a pop-stack}, Adv. Appl. Math. \textbf{108}
  (2019), 79--96.

\bibitem[CGP21]{ClaessonPop2}
A.~Claesson, B.~\'A Gu{\dh}mundsson, and J.~Pantone, \emph{Counting pop-stacked
  permutations in polynomial time}, Experiment. Math. (2021).

\bibitem[Def22a]{DefantMeeting}
C.~Defant, \emph{Meeting covered elements in $\nu$-tamari lattices}, Adv. Appl.
  Math. \textbf{134} (2022).

\bibitem[Def22b]{DefantPopCoxeter}
C.~Defant, \emph{Pop-stack-sorting for {C}oxeter groups}, Comb. Theory \textbf{2}
  (2022).

\bibitem[DKW]{ungar_games}
C.~Defant, N.~Kravitz, and N.~Williams, \emph{The {U}ngar games},
  arXiv:2302.06552.

\bibitem[DL]{ungar_chains}
C.~Defant and R.~Li, \emph{Ungarian {M}arkov chains}, arXiv:2301.08206.

\bibitem[DW23]{DefantWilliamsSemidistrim}
C.~Defant and N.~Williams, \emph{Semidistrim lattices}, Forum Math. Sigma
  \textbf{11} (2023), no.~50, 1--35.


\bibitem[DIJ17]{DIJ}
L.~Demonet, O.~Iyama, and G.~Jasso, \emph{$\tau$-tilting finite algebras,
  bricks, and $g$-vectors}, Int. Math. Res. Not. IMRN \textbf{2019} (2017),
  no.~3, 852--892.

\bibitem[DIR{\etalchar{+}}23]{DIRRT}
L.~Demonet, O.~Iyama, N.~Reading, I.~Reiten, and H.~Thomas, \emph{Lattice
  theory of torsion classes: beyond $\tau$-tilting theory}, Trans. Amer. Math.
  Soc. Ser. B \textbf{10} (2023), 542--612.

\bibitem[Dic66]{dickson}
S.~E. Dickson, \emph{A torsion theory for abelian categories}, Trans. Amer.
  Math. Soc. \textbf{121} (1966), no.~1, 223--235.

\bibitem[DR75]{DlabRingelFinite}
V.~Dlab and C.~M. Ringel, \emph{On algebras of finite representation type}, J.
  Algebra \textbf{33} (1975), 306–394.

\bibitem[DR76]{DlabRingelIndecomposable}
V.~Dlab and C.~M. Ringel, \emph{Indecomposable representations of graphs and algebras}, Mem. Am.
  Math. Soc. \textbf{6} (1976).

\bibitem[Eno21]{enomoto_bruhat}
H.~Enomoto, \emph{Bruhat inversions in {W}eyl groups and torsion-free classes
  over preprojective algebras}, Comm. Algebra \textbf{49} (2021), no.~5,
  2156--2189.

\bibitem[Eno22]{enomoto_ff}
H.~Enomoto, \emph{Rigid modules and {ICE}-closed subcategories in quiver
  representations}, J. Algebra \textbf{594} (2022), 364--388.

\bibitem[Eno23]{enomoto}
H.~Enomoto, \emph{From the lattice of torsion classes to the posets of wide
  subcategories and {ICE}-closed subcategories}, Algebr. Represent. Theory
  (2023).

\bibitem[ES22]{ES}
H.~Enomoto and A.~Sakai, \emph{{ICE}-closed subcategories and wide
  $\tau$-tilting modules}, Math. Z. \textbf{300} (2022), 541--577.

\bibitem[FJN95]{FJN}
R.~Freese, J.~Jezek, and J.~Nation, \emph{Free lattices}, Mathematical Surveys
  and Monographs, vol.~42, American Mathematical Society, 1995.

\bibitem[GIMO19]{GIMO}
A.~Garver, K.~Igusa, J.~P. Matherne, and J.~Ostroff, \emph{Combinatorics of
  exceptional sequences in type {A}}, Electron. J. Combin. \textbf{26} (2019),
  no.~1.

\bibitem[GM19]{GM}
A.~Garver and T.~McConville, \emph{Lattice properties of oriented exchange
  graphs and torsion classes}, Algebr. Represent. Theory \textbf{22} (2019),
  43--78.

\bibitem[GL91]{GL}
W.~Geigle and H.~Lenzing, \emph{Perpendicular categories with applications to
  representations and sheaves}, J. Algebra \textbf{144} (1991), no.~2,
  273--343.

\bibitem[Hana]{facial_torsion}
E.~J. Hanson, \emph{A facial order for torsion classes}, arXiv:2305.06031.

\bibitem[Hanb]{hanson}
E.~J. Hanson, \emph{$\tau$-exceptional sequences and the shard intersection order in
  type {A}}, arXiv:2303.11517.

\bibitem[HI21a]{HI_pairwise}
E.~J. Hanson and K.~Igusa, \emph{Pairwise compatibility for 2-simple minded
  collections}, J. Pure Appl. Algebra \textbf{225} (2021), no.~6.

\bibitem[HI21b]{HansonIgusa_picture}
E.~J. Hanson and K.~Igusa, \emph{$\tau$-cluster morphism categories and picture groups}, Comm.
  Alg. \textbf{49} (2021), no.~10, 4376--4415.

\bibitem[HW]{WemyssHara}
W.~Hara and M.~Wemyss, \emph{Spherical objects in dimension two and three},
  arXiv:2206.11552.

\bibitem[HY23]{HY}
E.~J. Hanson and X.~You, \emph{Morphisms and extensions between bricks over
  preprojective algebras of type {A}}, J. Algebra (2023).

\bibitem[Hon22]{Hong2022}
L.~Hong, \emph{The pop-stack-sorting operator on {T}amari lattices}, Adv. Appl.
  Math. \textbf{139} (2022).

\bibitem[ITW]{ITW}
K.~Igusa, G.~Todorov, and J.~Weyman, \emph{Picture groups of finite type and
  cohomology in type ${A}_n$}, arXiv:1609.02636.

\bibitem[IT09]{IT}
C.~Ingalls and H.~Thomas, \emph{Noncrossing partitions and representations of
  quivers}, Compos. Math. \textbf{145} (2009), no.~6, 1533--1562.

\bibitem[IRRT18]{IRRT}
O.~Iyama, N.~Reading, I.~Reiten, and H.~Thomas, \emph{Lattice structure of
  {W}eyl groups via representation theory of preprojective algebras}, Compos.
  Math. \textbf{154} (2018), no.~6, 1269--1305.

\bibitem[IRTT15]{IRTT}
O.~Iyama, I.~Reiten, H.~Thomas, and G.~Todorov, \emph{Lattice structure of
  torsion classes for path algebras}, Bull. Lond. Math. Soc. \textbf{47}
  (2015), no.~4, 639--650.

\bibitem[Jas15]{jasso}
G.~Jasso, \emph{Reduction of $\tau$-tilting modules and torsion pairs}, Int.
  Math. Res. Not. IMRN \textbf{2015} (2015), no.~16, 7190--7237.

\bibitem[KQ15]{KingQiuExchangeGraphs}
A.~King and Y.~Qiu, \emph{Exchange graphs and {E}xt quivers}, Adv. Math.
  \textbf{285} (2015), 1106--1154.

\bibitem[KY14]{KY}
S.~Koenig and D.~Yang, \emph{Silting objects, simple-minded collections,
  $t$-structures and co-$t$-structures for finite-dimensional algebras}, Doc.
  Math. \textbf{19} (2014), 403--438.

\bibitem[KS]{KontsevichSoibelmanStability}
M.~Kontsevich and Y.~Soibelman, \emph{Stability structures, motivic
  {D}onaldson-{T}homas invariants and cluster transformations},
  arXiv:0811.2435.

\bibitem[K{\"u}l17]{kulshammer}
J.~K{\"u}lshammer, \emph{Pro-species of algebras {I}: Basic properties},
  Algebr. Represent. Theory \textbf{20} (2017), no.~5, 1215--1238.

\bibitem[LW]{LuoWei}
Y.~Luo and J.~Wei, \emph{Boolean lattices of torsion classes over
  finite-dimensional algebras}, arXiv:2303.15802.

\bibitem[M{\v S}17]{MS}
F.~Marks and J.~{\v S}{\v t}ov\'i{\v c}ek, \emph{Torsion classes, wide
  subcategories, and localisations}, Bull. Lond. Math. Soc. \textbf{49} (2017),
  no.~3.

\bibitem[Miz14]{mizuno}
Y.~Mizuno, \emph{Classifying {$\tau$}-tilting modules over preprojective
  algebras of {D}ynkin type}, Math. Z. \textbf{277} (2014), no.~3-4, 665--690.

\bibitem[Miz22]{mizuno2}
Y.~Mizuno, \emph{Arc diagrams and 2-term simple-minded collections of
  preprojective algebras of type {A}}, J. Algebra \textbf{595} (2022),
  444--478.

\bibitem[MT20]{MizunoThomas}
Y.~Mizuno and H.~Thomas, \emph{Torsion pairs for quivers and the {W}eyl
  groups}, Selecta Math. \textbf{3} (2020), no.~46.

\bibitem[M\"uh19]{Muhle2019}
H.~M\"uhle, \emph{The core label order of a congruence-uniform lattice},
  Algebra Universalis \textbf{80} (2019).

\bibitem[PS19]{Pudwell}
L.~Pudwell and R.~Smith, \emph{Two-stack-sorting with pop stacks}, Australas.
  J. Combin. \textbf{74} (2019), 179--195.

\bibitem[Rea04]{ReadingLatticeCongruences}
N.~Reading, \emph{Lattice congruences of the weak order}, Order \textbf{21}
  (2004), 315--344.

\bibitem[Rea06]{ReadingCambrian}
N.~Reading, \emph{Cambrian lattices}, Adv. Math. \textbf{205} (2006), 313--353.

\bibitem[Rea07]{reading2007clusters}
N.~Reading, \emph{Clusters, {C}oxeter-sortable elements and noncrossing
  partitions}, Trans. Amer. Math. Soc. \textbf{359} (2007), no.~12, 5931--5958.

\bibitem[Rea11]{ReadingShard}
N.~Reading, \emph{Noncrossing partitions and the shard intersection order}, J.
  Algebraic Combin. \textbf{33} (2011), 483--530.

\bibitem[Rea15]{ReadingNoncrossing}
N.~Reading, \emph{Noncrossing arc diagrams and canonical join representations},
  SIAM J. Discrete Math. \textbf{29} (2015).

\bibitem[Rea16]{reading_book}
N.~Reading, \emph{Lattice theory of poset of regions, \textnormal{In (G.~Gr\"atzer and F.~Wehrung, eds.) Lattice
  Theory: Special Topics and Applications}}, vol.~2, Birkh\"auser, 2016.

\bibitem[RS11]{ReadingSpeyer2011}
N.~Reading and D.~E Speyer, \emph{Sortable elements in infinite {C}oxeter
  groups}, Trans. Amer. Math. Soc. \textbf{363} (2011), no.~2, 699--761.

\bibitem[RST21]{RST}
N.~Reading, D.~E Speyer, and H.~Thomas, \emph{The fundamental theorem of finite
  semidistributive lattices}, Selecta Math. \textbf{27} (2021), no.~59.

\bibitem[Ric02]{RickardEquivalences}
J.~Rickard, \emph{Equivalences of derived categories for symmetric algebras},
  J. Algebra \textbf{257} (2002), 460--481.

\bibitem[Rin76]{ringel}
C.~M. Ringel, \emph{Representations of {$K$}-species and bimodules}, J. Algebra
  \textbf{41} (1976), 269--302.

\bibitem[{The}23]{sagemath}
{The Sage Developers}, \emph{{S}agemath, the {S}age {M}athematics {S}oftware
  {S}ystem ({V}ersion 9.2)}, 2023, {\tt https://www.sagemath.org}.

\bibitem[Sak]{sakai}
A.~Sakai, \emph{Classifying $t$-structures via {ICE}-closed subcategories and a
  lattice of torsion classes}, arXiv:2307.11347.

\bibitem[Sch14]{schiffler}
R.~Schiffler, \emph{Quiver representations}, CMS Books in Mathematics/Ouvrages
  de Mathématiques de la SMC, Springer, Cham, 2014.

\bibitem[STW]{cataland}
C.~Stump, H.~Thomas, and N.~Williams, \emph{Cataland: Why the {F}uß}, To
  appear in Mem. Amer. Math. Soc.

\bibitem[Tho06]{ThomasTrim}
H.~Thomas, \emph{An analogue of distributivity for ungraded lattices}, Order
  \textbf{23} (2006), 249--269.

\bibitem[Tho21]{thomas_intro}
H.~Thomas, \emph{An introduction to the lattice of torsion classes}, Bull.
  Iranian Math. Soc. (2021).

\bibitem[TW19a]{TW2}
H.~Thomas and N.~Williams, \emph{Independence posets}, J. Comb. \textbf{10}
  (2019), no.~3, 545--578.

\bibitem[TW19b]{TW1}
H.~Thomas and N.~Williams, \emph{Rowmotion in slow motion}, Proc. Lon. Math. Soc. \textbf{119}
  (2019), no.~5, 1149--1178.

\bibitem[Ung82]{Ungar}
P.~Ungar, \emph{$2{N}$ noncollinear points determine at least $2{N}$
  directions}, J. Combin. Theory Ser. A \textbf{33} (1982), 343--347.

\end{thebibliography}
\end{document}